\documentclass[english,french]{article}
\usepackage[T1]{fontenc}
\usepackage[latin9]{inputenc}
\usepackage{geometry}
\usepackage{color}
\usepackage{babel}
\makeatletter
\addto\extrasfrench{%
   \providecommand{\og}{\leavevmode\flqq~}%
   \providecommand{\fg}{\ifdim\lastskip>\z@\unskip\fi~\frqq}%
}

\makeatother
\usepackage{float}
\usepackage{units}
\usepackage{amsthm}
\usepackage{amsmath}
\usepackage{amssymb}
\usepackage{graphicx}
\usepackage[numbers]{natbib}
\usepackage{xargs}[2008/03/08]
\usepackage[unicode=true,pdfusetitle,
 bookmarks=true,bookmarksnumbered=false,bookmarksopen=false,
 breaklinks=false,pdfborder={0 0 1},backref=false,colorlinks=true]
 {hyperref}

\makeatletter

\newcommand{\noun}[1]{\textsc{#1}}

\numberwithin{equation}{section}
\numberwithin{figure}{section}
 \theoremstyle{definition}
 \newtheorem*{defn*}{\protect\definitionname}
  \theoremstyle{remark}
  \newtheorem*{rem*}{\protect\remarkname}
 \newcommand\thmsname{\protect\theoremname}
 \newcommand\nm@thmtype{theorem}
 \theoremstyle{plain}
 
 \newenvironment{namedthm}[1][Undefined Theorem Name]{
   \ifx{#1}{Undefined Theorem Name}\renewcommand\nm@thmtype{theorem*}
   \else\renewcommand\thmsname{#1}\renewcommand\nm@thmtype{namedtheorem}
   \fi
   \begin{\nm@thmtype}}
   {\end{\nm@thmtype}}
\theoremstyle{plain}
\newtheorem{thm}{\protect\theoremname}[section]
  \theoremstyle{definition}
  \newtheorem{defn}[thm]{\protect\definitionname}
  \theoremstyle{remark}
  \newtheorem{rem}[thm]{\protect\remarkname}
  \theoremstyle{plain}
  \newtheorem{prop}[thm]{\protect\propositionname}
  \theoremstyle{plain}
  \newtheorem{lem}[thm]{\protect\lemmaname}
\newenvironment{lyxlist}[1]
{\begin{list}{}
{\settowidth{\labelwidth}{#1}
 \setlength{\leftmargin}{\labelwidth}
 \addtolength{\leftmargin}{\labelsep}
 }}
{\end{list}}
  \theoremstyle{plain}
  \newtheorem{cor}[thm]{\protect\corollaryname}
\newcommand{\lyxaddress}[1]{
\par {\raggedright #1
\vspace{1.4em}
\noindent\par}
}

\AtBeginDocument{\frenchbsetup{StandardLists}}
\usepackage{nicefrac}
\usepackage{kpfonts}
\usepackage{esint}

\@ifundefined{showcaptionsetup}{}{%
 \PassOptionsToPackage{caption=false}{subfig}}
\usepackage{subfig}
\AtBeginDocument{
  
}

\makeatother

  \addto\captionsenglish{\renewcommand{\corollaryname}{Corollary}}
  \addto\captionsenglish{\renewcommand{\definitionname}{Definition}}
  \addto\captionsenglish{\renewcommand{\lemmaname}{Lemma}}
  \addto\captionsenglish{\renewcommand{\propositionname}{Proposition}}
  \addto\captionsenglish{\renewcommand{\remarkname}{Remark}}
  \addto\captionsenglish{\renewcommand{\theoremname}{Theorem}}
  \addto\captionsfrench{\renewcommand{\corollaryname}{Corollaire}}
  \addto\captionsfrench{\renewcommand{\definitionname}{Définition}}
  \addto\captionsfrench{\renewcommand{\lemmaname}{Lemme}}
  \addto\captionsfrench{\renewcommand{\propositionname}{Proposition}}
  \addto\captionsfrench{\renewcommand{\remarkname}{Remarque}}
  \addto\captionsfrench{\renewcommand{\theoremname}{Théorème}}
  \providecommand{\corollaryname}{Corollaire}
  \providecommand{\definitionname}{Définition}
  \providecommand{\lemmaname}{Lemme}
  \providecommand{\propositionname}{Proposition}
  \providecommand{\remarkname}{Remarque}
  \providecommand{\theoremname}{Théorème}
\providecommand{\theoremname}{Théorème}

\begin{document}
\global\long\def\tx#1{\mathrm{#1}}
\global\long\def\nf#1#2{\mathrm{\nicefrac{#1}{#2}}}
\global\long\def\ww#1{\mathbb{#1}}
\global\long\def\dd#1{\,\mbox{d}#1}
\newcommandx\sect[2][usedefault, addprefix=\global, 1=j, 2=\beta]{V_{#1}^{#2}}
\newcommandx\tsect[2][usedefault, addprefix=\global, 1=j, 2=\beta]{\tilde{V}_{#1}^{#2}}
\newcommandx\ssect[2][usedefault, addprefix=\global, 1=j]{\mathcal{V}_{#1}^{#2}}
\newcommandx\efeu[2][usedefault, addprefix=\global, 1=j, 2=\pm]{\Omega_{#1}^{#2}}
\global\long\def\fol#1{\mathcal{F}_{#1}}
\global\long\def\lif#1{\mathcal{L}_{#1}}
\global\long\def\holo#1{\frak{h}_{#1}}
\global\long\def\fholo#1{\hat{\frak{h}}_{#1}}
\global\long\def\ii{\tx i}
\global\long\def\diff#1{\tx{Diff}\left(#1\right)}
\global\long\def\adh#1{\tx{adh}\left(#1\right)}
\global\long\def\re#1{\Re\left(#1\right)}
\global\long\def\im#1{\Im\left(#1\right)}
\global\long\def\supp#1{\tx{Supp}\left(#1\right)}
\global\long\def\germ#1{\ww C\left\{  #1\right\}  }
\newcommandx\beam[2][usedefault, addprefix=\global, 1=z_{*}, 2=\delta]{S_{\lambda}\left(#1,#2\right)}
\global\long\def\zsk{\nicefrac{\ww Z}{k\ww Z}}
\newcommandx\bstrong[1][usedefault, addprefix=\global, 1=\mathcal{U}]{\overline{\partial#1}}
\newcommandx\bweak[1][usedefault, addprefix=\global, 1=\mathcal{U}]{\widehat{\partial#1}}
\foreignlanguage{english}{}\global\long\def\ppp#1#2{\frac{\partial#2}{\partial#1}}
 \newcommandx\sat[2][usedefault, addprefix=\global, 1=\fol{}]{\tx{Sat}_{#1}\left(#2\right)}
\global\long\def\id{\tx{Id}}
\newcommandx\csad[3][usedefault, addprefix=\global, 1=\fol{}, 2=S, 3=p]{\tx{CS}\left(#1,#2,#3\right)}

\title{Germes de feuilletages présentables du plan complexe%
\thanks{PrePrint%
}}

\author{Loïc Teyssier%
\thanks{LaboratoireI.R.M.A. \textendash{} Université de Strasbourg%
}}

\date{Septembre 2013}
\maketitle
\begin{abstract}
Soit $\fol{}$ un germe de feuilletage singulier du plan complexe.
Sous l'hypothèse que $\fol{}$ est une courbe généralisée, \noun{D.~Mar\'{i}n}
et \noun{J.\textendash{}F.~Mattei} ont établi l'incompressibilité
de $\fol{}$ dans un voisinage épointé d'un ensemble fini de courbes
analytiques. On montre ici que cette hypothèse ne peut être ignorée,
en exhibant divers exemples de feuilletages réduits après un éclatement
qui ne satisfont pas cette propriété. Même si nous montrons que les
nœuds\textendash{}cols sont incompressibles individuellement, le fait
que leurs feuilles ne se rétractent pas tangentiellement sur toutes
les composantes du bord de leur domaine de définition empêche la généralisation
totale de la construction de Mar\'{i}n\textendash{}Mattei. Finalement
nous caractérisation une classe presque complète des feuilletages,
dits fortement présentables, pour lesquels la construction de la monodromie
de Mar\'{i}n\textendash{}Mattei est possible.\\

\textbf{\hfill{}Abstract\hfill{}}

\selectlanguage{english}%
Let $\fol{}$ be a germ of a singular foliation of the complex plane.
Assuming that $\fol{}$ is a generalized curve \noun{D.~Mar\'{i}n}
and \noun{J.\textendash{}F.~Mattei} proved the incompressibility
of the foliation in a neighborhood from which a finite set of analytic
curves is removed. We show in the present work that this hypothesis
cannot be eluded, by building examples of foliations, reduced after
one blow\textendash{}up, for which the property does not hold. Even
if we manage to prove that the individual saddle\textendash{}node
foliation is incompressible, their leaves not retracting tangentially
on all the components of the definition domain boundary forbids a
generalization of Mar\'{i}n\textendash{}Mattei's construction. We
finally characterize a near\textendash{}complete class of foliations,
called strongly presentable, for which the construction of Mar\'{i}n\textendash{}Mattei's
monodromy can be carried out.
\end{abstract}

\section{Introduction}

Par beaucoup d'aspects, un feuilletage holomorphe $\fol{}$ d'un ouvert
$\mathcal{U}$ semble généraliser, en son lieu régulier, la notion
de fibration localement triviale. Prenons à titre d'illustration le
feuilletage donné par les niveaux d'une submersion holomorphe 
\begin{eqnarray*}
f\,:\,\mathcal{B} & \longrightarrow & \ww C
\end{eqnarray*}
admettant une fibre singulière (non nécessaire irréductible) $\mathcal{S}=f^{-1}\left(0\right)$.
Ici $f$ est holomorphe sur (un voisinage de) la boule euclidienne
fermée $\adh{\mathcal{B}}$ de rayon choisi suffisamment petit pour
que les feuilles de $\fol{}$ soient transverses à la sphère $\partial\mathcal{B}$.
Un résultat classique de \noun{J.~Milnor}~\citep{Mimi} assure l'incompressibilité
de $\fol{}$ : il existe une famille de \textbf{tubes de Milnor} $\mathcal{T}_{\eta}$
de $\mathcal{S}$, images réciproques d'un petit disque $\eta\ww D$,
telle que le groupe fondamental de chaque feuille régulière de $\fol{}|_{\mathcal{T}_{\eta}}$
s'injecte dans celui de $\mathcal{T}_{\eta}\backslash\mathcal{S}$,
lui\textendash{}même isomorphe à $\pi_{1}\left(\mathcal{B}\backslash\mathcal{S}\right)$. 

Dans deux travaux récents~\citep{MarMat,MarMatMono}, \noun{D.~Mar\'{i}n}
et \noun{J.\textendash{}F.~Mattei} ont dégagé des conditions suffisantes
sous lesquelles ce résultat se généralise lorsque $\fol{}$ n'admet
pas d'intégrale première holomorphe non triviale. Précisons cela.
\begin{defn*}
Dans tout cet article $\fol{}$ désigne un germe de feuilletage holomorphe
singulier\emph{ }en $\left(0,0\right)$. On dit que $\fol{}$ est
\textbf{incompressible} s'il existe 
\begin{itemize}
\item une union finie $\mathcal{S}\subset\mathcal{B}$ de courbes analytiques
invariantes par $\fol{}$ et contenant la singularité, disons d'équation
$\left\{ f=0\right\} $ pour fixer les notations, appelées \textbf{séparatrices
distinguées},
\item une famille de tubes de Milnor $\left(\mathcal{T}_{\eta}\right)_{0<\eta\leq\eta_{0}}$
de $\mathcal{S}$ (au sens précédent~: $T_{\eta}=f^{-1}\left(\eta\ww D\right)$),
sur lesquels $\fol{}$ est bien défini,
\item une voisinage $\mathcal{U}$ de la singularité,
\end{itemize}

tels que, en notant $\mathcal{T}:=\mathcal{T}_{\eta_{0}}$,
\begin{enumerate}
\item $\mathcal{U}\subset\mathcal{T}$, cette inclusion induisant un isomorphisme
au niveau des groupes fondamentaux $\pi_{1}\left(\mathcal{U}\backslash\mathcal{S}\right)\simeq\pi_{1}\left(\mathcal{T}\backslash\mathcal{S}\right)$,
\item $\mathcal{T}_{\eta}^ {}\subset\mathcal{U}$ pour tout $\eta$ assez
petit,
\item pour chaque feuille $\mathcal{L}$ de $\fol{}|_{\mathcal{U}\backslash\mathcal{S}}$
le morphisme canonique induit par l'inclusion $\iota\,:\,\mathcal{L}\hookrightarrow\mathcal{U}\backslash\mathcal{S}$
\begin{eqnarray*}
\iota^{*}\,:\,\pi_{1}\left(\mathcal{L}\right) & \longrightarrow & \pi_{1}\left(\mathcal{U}\backslash\mathcal{S}\right)
\end{eqnarray*}
soit injectif.
\end{enumerate}

Pour préciser les notations nous serons parfois amenés à dire que
$\fol{}$ est incompressible dans $\left(\mathcal{U},\mathcal{S}\right)$.

\end{defn*}
\begin{rem*}
La condition~(1) stipule que l'on ne s'autorise pas à prendre des
voisinages $\mathcal{U}$ volontairement tordus\emph{ }pour accommoder
la topologie éventuellement compliquée des feuilles. La topologie
de l'espace ambiant doit être \og la plus simple possible \fg{}.
\end{rem*}
Toute singularité d'un germe en $p\in\ww C^{2}$ de feuilletage holomorphe
peut être \og réduite \fg{} à travers une application rationnelle
propre%
\footnote{Voir par exemple l'algorithme de \noun{A.~Seidenberg}~\citep{Seiden}.%
} $E\,:\,\mathcal{M}\to$$\left(\ww C^{2},p\right)$, où $\mathcal{M}$
est un voisinage conforme d'un arbre $E^{-1}\left(0\right)$ de diviseurs
$\ww P_{1}\left(\ww C\right)$ à croisement normaux. Le tiré en arrière
$E^{*}\fol{}$ ne possède alors que des singularités (dites réduites)
placées sur le diviseur exceptionnel $E^{-1}\left(0\right)$, dont
le quotient $\lambda$ des valeurs propres de la partie linéaire n'est
pas un rationnel positif. 
\begin{namedthm}
[Théorème de Mar\'{i}n-Mattei]Tout feuilletage dont la réduction
ne comporte pas de nœud\textendash{}col%
\footnote{$\lambda=0$.%
} ou de selle quasi\textendash{}résonante%
\footnote{$\lambda\in\ww R_{<0}\backslash\ww Q$, non linéarisable.%
} est incompressible.
\end{namedthm}
Ce théorème est traité dans~\citep{MarMat} pour les feuilletages
n'ayant pas de composante dicritique dans leur réduction, et dans~\citep{MarMatMono}
pour les cas restants en faisant l'hypothèse technique supplémentaire
que les éventuelles composantes initiales%
\footnote{Voir Section~\ref{sub:Comp-init} pour une définition.%
} sont dynamiquement isolées. Nous reviendrons en Section~\ref{sec:Composantes-initiales}
sur ce dernier point~; bornons\textendash{}nous à mentionner pour
l'instant que cette hypothèse est superflue.
\begin{rem*}
~
\begin{enumerate}
\item Il est possible de décrire explicitement l'ensemble $\mathcal{S}$.
En notant $E$ le morphisme de réduction de la singularité de $\fol{}$,
l'ensemble $\mathcal{S}$ est l'union des adhérences des images par
$\mathcal{E}$ des séparatrices de $E^{*}\fol{}$ croisant des composantes
non dicritiques, auxquelles s'ajoute l'image par $\mathcal{E}$ d'un
germe de feuille transverse par composante dicritique. 
\item L'hypothèse du théorème portant sur le type des singularités réduites
est générique à donnée combinatoire de l'arbre de réduction et à nombre
de singularités réduites fixé.
\end{enumerate}
\end{rem*}
L'incompressibilité assure l'existence d'un revêtement universel feuilleté.
Par cela on entend que le revêtement universel $\pi_{\mathcal{U}}\,:\,\widetilde{\mathcal{U}\backslash\mathcal{S}}\to\mathcal{U}\backslash\mathcal{S}$
est aussi, en restriction, un revêtement universel de chaque feuille
de $\fol{}$. Le groupe des automorphismes de ce revêtement est alors
constitué de symétries du feuilletage $\pi_{\mathcal{U}}^{*}\fol{}$,
et à ce titre agit naturellement sur l'espace de ses feuilles $\tilde{\Omega}_{\mathcal{U}}$.
On dispose alors d'une action
\begin{eqnarray*}
\frak{m}_{\mathcal{U}}\,:\,{\tt Aut}\left(\pi_{\mathcal{U}}\right) & \longrightarrow & {\tt Aut}\left(\tilde{\Omega}_{\mathcal{U}}\right)
\end{eqnarray*}
que l'on nomme \textbf{monodromie} de $\left(\fol{},\mathcal{U},\mathcal{S}\right)$.
Le quotient $\Omega_{\mathcal{U}}:=\nf{\tilde{\Omega}_{\mathcal{U}}}{\frak{m}_{\mathcal{U}}}$
s'identifie canoniquement à l'espace des feuilles de $\fol{}$.

Cette construction ne revêt qu'un intérêt modeste si l'on ne dote
pas les espaces de feuilles d'une structure analytique. Cette structure
supplémentaire va faire de la monodromie un invariant analytique du
triplet $\left(\fol{},\mathcal{U},\mathcal{S}\right)$, dont la \og germification \fg{}
quand $\eta_{0}\to0$ est un classifiant local (au voisinage de $\mathcal{S}$)
complet pour un choix générique de feuilletages~\citep{MarMatMono}.
Sa construction est donc une étape importante vers une compréhension
plus globale des germes de singularités de feuilletages. 

L'ingrédient nécessaire à l'existence d'une structure analytique canonique
sur $\tilde{\Omega}_{\mathcal{U}}$ est celle d'une courbe transverse
$\mathcal{C}$ (non nécessairement irréductible) qualifiée ici de
complètement connexe. Les composantes de $\pi^{-1}\left(\mathcal{C}\right)$
doteront l'espace des feuilles de cartes analytiques, à condition
que chaque feuille de $\pi^{*}\fol{}$ coupe au plus une fois chaque
composante de $\pi^{-1}\left(\mathcal{C}\right)$. Ce sont ces propriétés
qui sont isolées par la définition suivante~:
\begin{defn}
\label{def_TCC}Un germe de courbe analytique $\mathcal{C}$ (non
nécessairement irréductible) est une \textbf{transversale complètement
connexe }de $\fol{}$ s'il existe un couple $\left(\mathcal{U},\mathcal{S}\right)$,
dans lequel $\fol{}$ est incompressible, tel que:
\begin{enumerate}
\item $\mathcal{C}\backslash\mathcal{S}$ soit une courbe analytique lisse
transverse aux feuilles de $\fol{}$,
\item $\sat{\mathcal{U}\cap\mathcal{C}\backslash\mathcal{S}}=\mathcal{U}\backslash\mathcal{S}$,
\item $\mathcal{C}\backslash\mathcal{S}$ soit $1$\textendash{}connexe
dans $\mathcal{U}\backslash\mathcal{S}$ relativement à $\fol{}$.
\end{enumerate}
\end{defn}
Notons qu'il existe toujours une courbe $\mathcal{C}$ satisfaisant
les deux premières conditions~(1) et~(2), quitte à munir $\mathcal{C}$
de suffisamment de composantes (voir~\citep[p161]{Lolo}), courbe
que l'on nomme alors \textbf{transversale complète}. La notion de
$1$\textendash{}connexité est définie dans~\citep{MarMat} (voir
aussi la Définition~\ref{def_1_connexe}), et revient ici à demander
que chaque feuille de $\pi^{*}\fol{}$ coupe au plus une fois chaque
composante de $\pi^{-1}\left(\mathcal{C}\right)$. Un corollaire du
théorème de Mar\'{i}n\textendash{}Mattei est alors l'existence, sous
les mêmes hypothèses, d'une transversale complètement connexe~\citep[Théorème 6.1.1, p900]{MarMat}.\bigskip{}

Le but de cet article est d'établir d'une part que la condition de
ne pas posséder de nœud\textendash{}col dans la réduction du lieu
singulier n'est pas superflue, mais d'autre part qu'elle n'est pas
toujours nécessaire, dans la perspective de bâtir cet invariant analytique.
Nous allons généraliser sensiblement les résultats exposés ci\textendash{}dessus,
en donnant une caractérisation presque complète des feuilletages sous\textendash{}tendant
une monodromie.

\bigskip{}

Je tiens à remercier chaleureusement \noun{D.~Mar\'{i}n} et \noun{J.\textendash{}F.~Mattei
}pour les discussions que nous avons eues autour de ce sujet, plus
particulièrement pour la construction des exemples donnés à la fin,
et pour leurs encouragements répétés à mener à terme la rédaction
du présent travail.

\subsection{Présentation des principaux résultats}

Nous souhaitons donner dans ce texte la preuve des trois théorèmes
principaux suivants.
\begin{namedthm}
[Théorème A]Un germe de feuilletage de type nœud\textendash{}col
ou selle quasi\textendash{}résonnante est incompressible.
\end{namedthm}
Un tel résultat pourrait laisser espérer que le théorème de Mar\'{i}n\textendash{}Mattei
se généralise sans contrainte. Ce n'est malheureusement pas le cas.
\begin{namedthm}
[Théorème B]Il existe des feuilletages $\fol{}$ singuliers, non
dicritiques, qui sont compressibles.
\end{namedthm}
Contrairement aux autres singularités réduites, qui possèdent deux
séparatrices transverses passant par la singularité et tangentes aux
espaces propres de leur partie linéaire, un nœud\textendash{}col générique
n'en admet qu'une (on dit que ce dernier est \textbf{divergent}%
\footnote{La terminologie sera explicitée et justifiée en Section~\ref{sec:Notations}.%
}), appelée \textbf{séparatrice forte} et tangente à l'espace propre
de la valeur propre non nulle de sa partie linéaire, même si certains
en possèdent deux (ceux\textendash{}ci sont nommés \textbf{convergents}$^{\tx 5}$). 

Les premiers exemples que nous avons pu construire pour établir le
Théorème~B possèdent tous un nœud\textendash{}col divergent dans
leur réduction. On pourrait alors penser qu'il suffit de proscrire
les nœuds\textendash{}cols divergents pour assurer l'incompressibilité
du feuilletage, suivant l'idée que ceux\textendash{}ci privent de
façon injuste l'ensemble $\mathcal{S}$ d'une séparatrice. Il n'en
est rien~: nous construisons également des exemples de feuilletages
compressibles n'ayant aucun nœud\textendash{}col divergent.

\bigskip{}

Pour autant nous avons pu affaiblir l'hypothèse du théorème de Mar\'{i}n\textendash{}Mattei:
l'incompressibilité et l'existence d'une transversale complètement
connexe vont dépendre de manière cruciale de la façon dont les nœuds\textendash{}cols
sont positionnés dans l'arbre de réduction.
\begin{defn}
Soit $\fol{}$ un germe de feuilletage du plan complexe.
\begin{enumerate}
\item $\fol{}$ est \textbf{présentable} s'il est incompressible et s'il
admet une transversale complètement connexe.
\item $\fol{}$ est \textbf{fortement présentable} si les séparatrices fortes
des nœuds\textendash{}cols apparaissant dans une réduction minimale
ne sont jamais des composantes du diviseur exceptionnel%
\footnote{C'est en particulier le cas si $\fol{}$ est une singularité réduite.%
}. 
\end{enumerate}
\end{defn}
\begin{figure}[H]
\hfill{}\includegraphics[width=10cm]{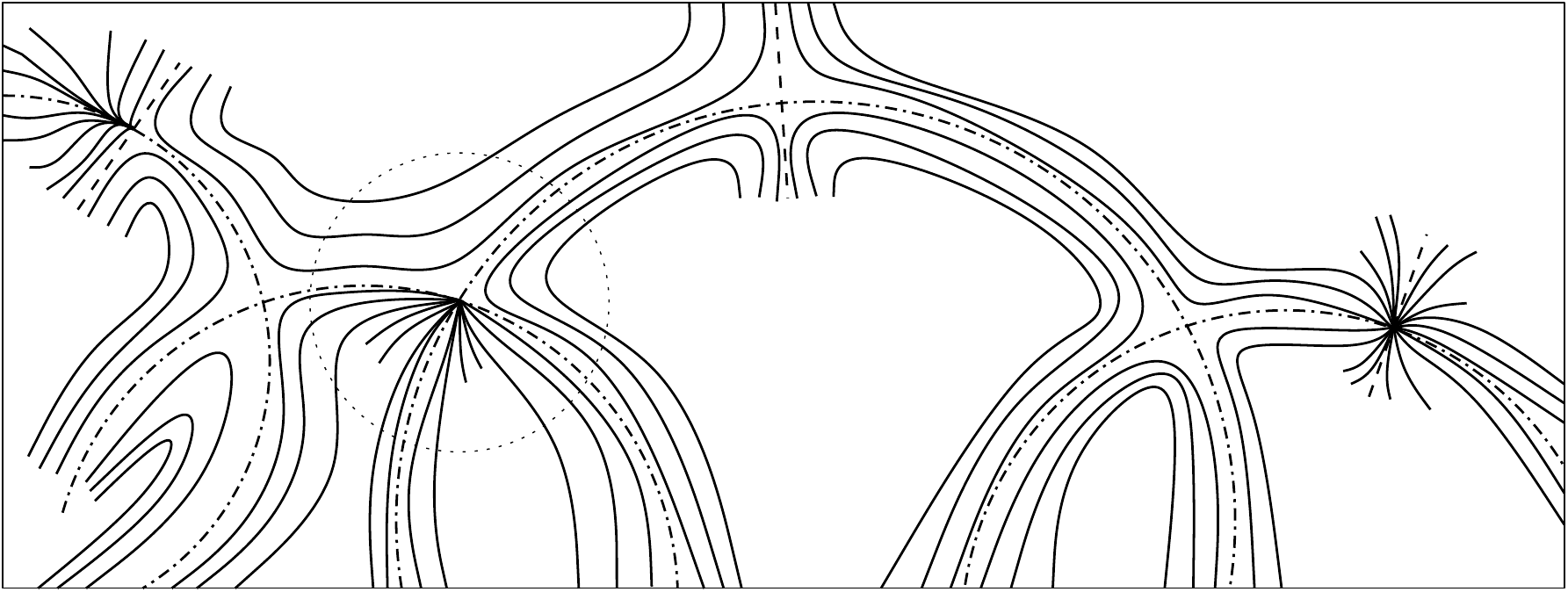}\hfill{}

\caption{Ce feuilletage n'est pas fortement présentable car il exhibe un nœud\textendash{}col
dans un coin.}
\end{figure}

Un feuilletage fortement présentable non réduit ne contient donc jamais
de nœud\textendash{}col divergent dans sa réduction, et les singularités
apparaissant aux points de croisement de deux composantes du diviseur
exceptionnel (les \og coins \fg{}) ne sont pas des nœuds\textendash{}cols.
\begin{rem}
\label{rem_eclate_presentable}La définition de feuilletage fortement
présentable n'est curieusement pas invariante par éclatement ponctuel.
Plus précisément, l'éclatement d'un nœud\textendash{}col produit un
diviseur possédant d'une part une selle non linéarisable, d'autre
part un nœud\textendash{}col dont la séparatrice forte est un diviseur.
C'est pourquoi nous imposons une réduction minimale dans la définition
d'un feuilletage fortement présentable (ou en tout cas une réduction
de la singularité où aucun nœud\textendash{}col n'a été éclaté après
apparition). \end{rem}
\begin{namedthm}
[Théorème C]~Tout germe de feuilletage fortement présentable est
présentable.
\end{namedthm}
Comme le souligne la remarque précédente la condition d'être fortement
présentable n'est pas nécessaire pour être présentable. Cependant
les exemples que nous construisons montrent qu'elle n'est pas superflue.
Caractériser complètement les feuilletages présentables est une question
difficile, car les propriétés les définissant sont globales alors
que les limitations (techniques) imposées par la forte présentabilité
sont locales.

\subsection{Structure de l'article et esquisse des preuves}

Le corps de cet article commence par une section consacrée à introduire
les notations et résultats classiques que nous utiliserons. Ensuite
nous procéderons aux démonstrations proprement dites. Afin de permettre
au lecteur pressé de se faire une idée des techniques employées pour
parvenir aux résultats annoncés, nous présentons rapidement dans le
reste du paragraphe le squelette de notre argumentation.

\subsubsection{Le Théorème~A est démontré en Sections~\ref{sec:Incompress_ND}
et~\ref{sec:Incompress_NC}}

On montre directement que le tiré\textendash{}en\textendash{}arrière
d'un feuilletage nœud\textendash{}col convergent ou selle quasi\textendash{}résonante
par le revêtement universel d'un polydisque $\rho\ww D\times r\ww D$
assez petit, épointé de la séparatrice forte $\left\{ x=0\right\} $,
ne possède que des feuilles simplement connexes. On invoque un argument
de transversalité avec les fibres de la projection $\Pi\,:\,\left(\log x,y\right)\mapsto\log x$
pour prouver dans un premier temps que le bord d'une feuille $\lif{}$
est contenu dans le bord du polydisque. Un argument variationnel immédiat
permet alors d'identifier des familles de chemins complètement contenues
dans $\Pi\left(\lif{}\right)$, que l'on nomme \og faisceaux de stabilité \fg{},
le long desquels le module de l'ordonnée de $\lif{}$ diminue~: ces
chemins sont donc contenus dans $\lif{}$. Pour prouver que tout cycle
$\gamma$ de $\lif{}$ est tangentiellement trivial, on construit
une homotopie entre $\Pi\circ\gamma$ et un lacet bordant une région
d'intérieur vide en suivant des faisceaux de stabilité, ce qui garantit
qu'elle se relève dans $\lif{}$ en une trivialisation de $\gamma$. 

Le cas divergent se ramène au cas précédent en redressant au\textendash{}dessus
de secteurs dans la variable $x$ les resommées sectorielles de la
séparatrice \og faible \fg{}. Puisque les faisceaux de stabilité
restent à l'intérieur d'un secteur donné, et comme le redressement
des séparatrices se fait à travers une application fibrée dans la
coordonnée $x$, l'argument est essentiellement le même. $\mathcal{U}\backslash\left\{ x=0\right\} $
est alors l'union des pré\textendash{}images du polydisque $\rho\ww D\times r\ww D$
par les redressements sectoriels.

\subsubsection{Le Théorème~C est établi en Section~\ref{sec:presentable}}

Pour que le raisonnement soit le plus clair possible nous résumons
la construction initiale de Mar\'{i}n\textendash{}Mattei en Section~\ref{sec:Construction_originale}.
En adaptant la construction, on montre assez facilement l'existence
d'une transversale complètement connexe à un feuilletage fortement
présentable, en incorporant un nouveau type de bloc élémentaire $B$
renfermant des nœuds\textendash{}cols convergents ou des selles quasi\textendash{}résonantes.
La propriété d'incompressibilité s'obtient en utilisant également
les propriétés d'origine de collage bord\textendash{}à\textendash{}bord.
Il faut pour cela garantir que le bloc élémentaire $B$ (ou plutôt
son bord) se prête à cet assemblage. Le cas des selles quasi\textendash{}résonantes
est très proche du cas des selles résonantes traitées dans~\citep{MarMat},
puisque les feuilles se rétractent radialement sur leur bord (feuilles
de type collier). Expliquons maintenant le cas, plus complexe, du
nœud\textendash{}col.

Le bord de $B$ est formé de deux composantes connexes : une composante
forte $\partial B\cap\left\{ \left|y\right|=r\right\} $ et une composante
faible $\partial B\cap\left\{ \left|x\right|=\rho\right\} $. On peut
faire en sorte de choisir $B$ pour que l'une ou l'autre de ces composantes
soit de type suspension. Cela signifie grossièrement que cette composante,
disons forte (\emph{resp.} faible), s'obtient comme le balayage par
le transport holonome d'un petit disque conforme transverse $\Sigma$
en effectuant une fois le tour du cercle $\left\{ \left|y\right|=r\right\} $
(\emph{resp.} $\left\{ \left|x\right|=\rho\right\} $). Dès que cette
propriété est assurée il faut contrôler que l'intersection de $\Sigma$
avec son image par l'holonomie est connexe (afin de ne pas créer de
topologie artificielle dans l'espace ambiant). Nous montrons ces propriétés
en utilisant les outils (notamment la notion de rugosité) introduits
par \noun{D.~Mar\'{i}n} et \noun{J.\textendash{}F.~Mattei}. 

Ceci étant dit, arrive l'obstruction majeure forçant les nœuds\textendash{}cols
à être placés de la bonne façon. Il faut en effet pouvoir garantir
la $1$\textendash{}connexité dans $B$ de la composante du bord que
l'on souhaite assembler aux autres blocs. Cette propriété est vérifiée
par la composante faible. En contraste, la composante forte du bord
n'est jamais $1$\textendash{}connexe dans $B$. Dans chaque feuille
d'un nœud\textendash{}col existent en effet des \og chemins inamovibles \fg{}
dont les extrémités $\left\{ p_{1},p_{2}\right\} $ sont situées dans
une transversale donnée $\left\{ y=\mbox{cte}\right\} $ et qui sont
homotopes dans $B$, privé de la séparatrice qui n'est pas un diviseur,
à un chemin les joignant dans la transversale. Cependant il n'est
pas possible d'opérer cette homotopie tangentiellement au feuilletage.
La coexistence d'un comportement \og col \fg{} et \og nœud \fg{}
au sein de certaines feuilles prévient en effet cette possibilité.
Le point de vue topologique est le suivant : certaines feuilles de
nœud\textendash{}col ne peuvent se rétracter tangentiellement sur
la composante forte du bord, en d'autres termes les feuilles ne sont
pas de type collier vis\textendash{}à\textendash{}vis de cette composante
(alors qu'ils le sont pour la composante faible). Ainsi, quelle que
soit la forme du bloc $B$, la composante forte de son bord ne sera
jamais $1$\textendash{}connexe dans $B$, ce qui interdit d'appliquer
le procédé de localisation sur lequel se base le théorème de Mar\'{i}n\textendash{}Mattei.
Néanmoins cette obstruction est d'ordre technique, comme l'a souligné
la Remarque~\ref{rem_eclate_presentable}.

\subsubsection{Les exemples du Théorème~B sont produits en Section~\ref{sec:Exemples}}

Nous présentons quelques exemples, dont le Théorème~B découle. Dans
un premier temps (Section~\ref{sub:ex_dvg}), la compressibilité
des feuilles est fournie par un argument de taille de groupe (impossibilité
d'injecter un groupe libre de rang $2$ dans un groupe commutatif),
en réalisant un tel groupe de difféomorphismes comme l'holonomie projective
d'un feuilletage admettant au moins un nœud\textendash{}col divergent
dans sa réduction. L'argument est ici d'utiliser la non existence
d'une seconde séparatrice pour forcer la compressibilité. Il est à
noter que dans le second exemple nous parvenons en outre à exhiber
un lacet non trivial dans le noyau de $\iota^{*}$.

Nous construisons ensuite (Section~\ref{sub:ex_cvg}) une famille
d'exemples de feuilletages compressibles n'ayant que des nœuds\textendash{}cols
convergents dans leur réduction. À cette fin nous exploitons la mise
en défaut, déjà évoquée, de la $1$\textendash{}connexité de la composante
forte du bord dans un bloc renfermant un nœud\textendash{}col. On
met en regard deux nœuds\textendash{}cols modèles partageant la même
séparatrice forte de sorte à connecter deux chemins inamovibles pour
former un cycle. Celui\textendash{}ci est trivial dans l'espace privé
des séparatrices mais pas dans la feuille le contenant. 

\tableofcontents{}

\section{\label{sec:Notations}Notations et rappels}

Nous décrivons brièvement les deux types de singularités élémentaires
apparaissant comme singularités finales dans la réduction d'un germe
de feuilletage holomorphe singulier, essentiellement afin de fixer
quelques notations. Les feuilles du feuilletage $\fol{}$ associé
à la $1$\textendash{}forme différentielle
\begin{eqnarray*}
\omega & = & P\dd x-Q\dd y
\end{eqnarray*}
sont les surfaces de Riemann intégrant la distribution des noyaux
de $\omega$. Si $\fol{}$ n'est pas régulier en un point $p$, c'est\textendash{}à\textendash{}dire
$\omega\left(p\right)=0$, on peut supposer que $P\wedge Q=1$ en
tant que germes en ce point, auquel cas leurs zéros communs sont isolés.
On se place alors sur un voisinage de $p$ suffisamment petit pour
que ce soit la seule singularité. La singularité $p$ est dite \textbf{élémentaire}
lorsque la partie linéaire du feuilletage en $p$, naturellement identifiée
à (la classe sous l'action de $\ww C_{\neq0}$ par homothéties de)
la matrice $\left[\begin{array}{cc}
\ppp xP & \ppp xQ\\
\ppp yP & \ppp yQ
\end{array}\right]\left(p\right)$, admet au moins une valeur propre non nulle, disons $\lambda_{2}$.
Quand le quotient de l'autre valeur propre par celle\textendash{}ci
\begin{eqnarray*}
\lambda & := & \nf{\lambda_{1}}{\lambda_{2}}
\end{eqnarray*}
 n'est pas un rationnel strictement positif on parle de singularité
\textbf{réduite}. Si $\lambda_{j}\neq0$ il existe une unique feuille
de $\fol{}$ dont l'adhérence est une variété analytique lisse tangente
en $p$ à l'espace propre associé à $\lambda_{j}$. Une telle séparatrice
sera qualifiée de \textbf{séparatrice forte}.

\subsection{Singularités réduites non dégénérées~: $\lambda\neq0$}

Après avoir \og redressé \fg{} les séparatrices fortes sur les axes
$\left\{ xy=0\right\} $ d'un système local de coordonnés analytiques,
le feuilletage est donné par la $1$\textendash{}forme différentielle
\begin{eqnarray}
\omega_{R} & := & \lambda x\dd y-y\left(1+R\right)\dd x\label{eq:prepa_non-degenere}
\end{eqnarray}
où $R\in x^{k}\germ{x,y}$ pour $k\in\ww N_{>0}$ (arbitrairement
grand).

Lorsque $\lambda\notin\ww R$ la singularité est \textbf{hyperbolique},
et le théorème de Poincaré assure sa linéarisabilité. Lorsque $\lambda>0$
la singularité est un \textbf{nœud}, également linéarisable. Lorsque
$\lambda\in\ww Q_{<0}$ et que la singularité n'est pas linéarisable
on parle de \textbf{selle résonnante}. Enfin si $\lambda\in\ww R_{<0}\backslash\ww Q$
la singularité est toujours formellement linéarisable, mais pour un
ensemble de mesure nulle de tels $\lambda$ il se peut que la série
linéarisante diverge (présence de petits diviseurs). On parle alors
de \textbf{selle quasi\textendash{}résonnante}.

On notera que sous la forme préparée~\ref{eq:prepa_non-degenere}
les feuilles sont transverses aux fibres de la projection canonique

\begin{eqnarray*}
\Pi\,:\,\left(x,y\right) & \longmapsto & x\,,
\end{eqnarray*}
en dehors de la séparatrice $\left\{ x=0\right\} $.

\subsection{Nœuds\textendash{}cols~: $\lambda=0$}

Tout germe de feuilletage $\fol{}$ de type nœud\textendash{}col,
d'invariant formel $\left(k,\mu\right)\in\ww N_{>0}\times\ww C$ \emph{a
priori }fixé, est, dans une certaine coordonnée locale centrée en
$\left(0,0\right)$, donné par une forme différentielle dans l'écriture
préparée de Dulac~\citep{Dulac}
\begin{eqnarray}
\omega_{R} & := & x^{k+1}\dd y-\left(y+R\right)\dd x\,,\label{eq:prepa_degenere}
\end{eqnarray}
où $R\in x^{k}\left(\mu y+x\germ{x,y}\right)$ est arbitraire. La
séparatrice forte est redressée sur $\left\{ x=0\right\} $. Si on
s'autorise des changements de coordonnées formels alors on peut toujours
conjuguer $\omega_{R}$ à $\omega_{0}$, forme que l'on appellera
\textbf{modèle formel}. Les feuilles de $\omega_{0}$ sont les composantes
connexes des niveaux de l'intégrale première multivaluée
\begin{eqnarray*}
H_{0}\left(x,y\right) & := & yx^{-\mu}\exp\left(kx^{-k}\right)\,.
\end{eqnarray*}

Ici encore les feuilles autres que la séparatrice forte sont transverses
aux fibres de $\Pi$. L'entier $k$ est un invariant topologique,
et gouverne le nombre $2k$ (génériquement optimal) de secteurs sur
lesquels $\omega_{R}$ sera conjuguée analytiquement au modèle formel~\citep{HuKiMa}.
Nous utiliserons essentiellement l'existence de \og séparatrices
sectorielles \fg{} démontrée dans la référence citée. Précisons cela.
Le feuilletage admet une séparatrice formelle $\left\{ y=\hat{s}\left(x\right)\right\} $
correspondant à la feuille d'adhérence analytique $\left\{ y=0\right\} $
de $\omega_{0}$, que l'on nomme \textbf{séparatrice faible}. On dit
que le nœud\textendash{}col est \textbf{divergent} ou \textbf{convergent}
selon que la série $\hat{s}$ diverge ou converge. Pour un nœud\textendash{}col
convergent le changement de coordonnée analytique 
\begin{eqnarray*}
\left(x,y\right) & \longmapsto & \left(x,y-\hat{s}\left(x\right)\right)
\end{eqnarray*}
 transforme la forme différentielle initiale en $\omega_{yR}$ (pour
un autre germe encore noté $R$ par simplicité), auquel cas la séparatrice
faible initiale est redressée sur $\left\{ y=0\right\} $. Dans le
cas contraire la série $\hat{s}$ est une série formelle $k$\textendash{}Gevrey,
qui s'avère être $k$\textendash{}sommable. Il existe ainsi $k$ fonctions
holomorphes et bornées $\left(s_{j}\right)_{j\in\zsk}$, chacune vivant
sur un secteur ouvert
\begin{eqnarray*}
\sect & := & \left\{ x\,:\,0<\left|x\right|<\rho\,,\,\left|\arg x-\frac{\pi}{k}\left(2j+1\right)\right|<\frac{\pi}{k}+\beta\right\} \,,
\end{eqnarray*}
 $\rho>0$ et $0<\beta<\frac{\pi}{2k}$ étant fixés assez petits,
et dont le graphe est tangent à $\fol{}$. Leur développement asymptotique
en $0$ coïncide avec $\hat{s}$. Lorsque $k=1$ le secteur $V_{0}$
doit être compris comme un secteur d'ouverture supérieure à $2\pi$
ne se recollant pas. Ceci étant dit, comme la plupart du travail sera
menée dans la coordonnée $\log x$, il n'y aura pas d\textquoteright{}ambiguïté.

\subsection{Holonomie}

Le principal objet dynamique associé à un feuilletage est son holonomie.
Si $\mathcal{S}$ est une séparatrice (forte ou, lorsqu'elle converge,
faible) disons $\mathcal{S}=\left\{ x=0,\, y\neq0\right\} $, on construit
l'holonomie comme suit. On choisit un générateur $\gamma\,:\,\left[0,1\right]\to\mathcal{S}$
de son groupe fondamental, dont l'image est incluse dans un voisinage
$\mathcal{U}$ simplement connexe donné de $\left(0,0\right)$~;
il existe un voisinage $\mathcal{W}\subset\mathcal{U}$ de l'image
de $\gamma$ sur lequel $\ker\omega_{R}$ est transverse aux fibres
de la projection $\pi\,:\,\left(x,y\right)\mapsto y$. Quitte à réduire
la taille de $\mathcal{W}$ on pourra relever $\gamma$ par $\pi$
dans les feuilles de $\omega_{R}|_{\mathcal{V}}$ en s'appuyant sur
un point $p$ quelconque de la transversale $\Sigma:=\pi^{-1}\left(\gamma\left(0\right)\right)\cap\mathcal{W}$.
Rien n'indique que ce relevé $\gamma_{p}$ sera un lacet et l'holonomie
forte sera l'application holomorphe injective fixant $p_{*}:=\gamma\left(0\right)$,
définie par
\begin{eqnarray*}
\holo{\gamma}\,:\,\Sigma & \longrightarrow & \pi^{-1}\left(p_{*}\right)\\
p & \longmapsto & \gamma_{p}\left(1\right)\,.
\end{eqnarray*}
Le germe d'application holomorphe ainsi défini ne dépend que de la
classe d'homotopie%
\footnote{Toutes les homotopies considérées dans ce texte, à part quelques exceptions
explicitement annoncées, seront supposées à extrémités fixes. %
} dans $\mathcal{S}$ de $\gamma$. Plus généralement cette construction
donne lieu à un morphisme de groupes
\begin{eqnarray*}
\holo{\bullet}\,:\,\pi_{1}\left(\mathcal{S},p_{*}\right) & \longrightarrow & \diff{\Sigma,p_{*}}\\
\gamma & \longmapsto & \holo{\gamma}\,.
\end{eqnarray*}
Ici $\diff{\Sigma,p_{*}}$ est le groupe des germes de difféomorphismes
au voisinage de $p_{*}$ et laissant ce point fixe.

\subsection{\label{sub:Comp-init}Réduction d'une singularité et branches mortes}

Soit $\fol{}$ un germe de feuilletage singulier à l'origine de $\ww C^{2}$
et $E\,:\,\mathcal{M}\to\left(\ww C^{2},0\right)$ un morphisme minimal
de réduction de $\fol{}$. On note $\hat{\fol{}}:=E^{*}\fol{}$ le
feuilletage holomorphe singulier sur $\mathcal{M}$ transformé de
$\fol{}$ par $E$. La réduction s'obtient par une succession d'éclatements
ponctuels des singularités non réduites apparaissant au fur et à mesure.
À chaque étape intermédiaire de la réduction de $\fol{}$, l'image
réciproque de $0$ est une union finie de diviseurs $\ww P_{1}\left(\ww C\right)$
d'auto\textendash{}intersection négative, que l'on nomme \textbf{composantes},
se croisant transversalement en un certain nombres de points, appelés
\textbf{coins}. 
\begin{rem}
\label{rem_reduction=00003Darbre}\citep{Seiden} Par un coin passent
exactement deux composantes, qui seront qualifiées d'adjacentes. De
plus le graphe d'incidence des composantes est un arbre connexe.\end{rem}
\begin{defn}
\label{def_compo_initiale}~
\begin{enumerate}
\item Une \textbf{branche morte} $B$ de $\hat{\fol{}}$ est une union maximale
de composantes adjacentes, chacune d'un des trois types suivants~:

\begin{itemize}
\item possédant au plus deux coins et aucune autre singularité de $\hat{\fol{}}$,
\item possédant exactement un coin et aucune autre singularité de $\hat{\fol{}}$
(\textbf{extrémité} de $B$),
\item possédant un coin et une autre singularité de $\hat{\fol{}}$ (\textbf{composante
d'attache} de $B$), 
\end{itemize}

ayant exactement une extrémité et une composante d'attache. Le graphe
d'incidence de ces composantes est donc un arbre ayant la combinatoire
d'une chaîne, c'est pourquoi on appellera parfois \textbf{maillons}
les composantes formant $B$. La singularité de $\hat{\fol{}}$ distincte
du coin de la composante d'attache sera appelée \textbf{point d'attache}.

\item Une \textbf{composante initiale} de $\hat{\fol{}}$ est une composante
non dicritique du diviseur exceptionnel $E^{-1}\left(0\right)$, à
laquelle est attachée au moins deux branches mortes et ne portant
en sus de ces points d'attache qu'une seule singularité de $\hat{\fol{}}$.
Notons que rien n'interdit à $\mathcal{C}$ de croiser une composante
dicritique de $\hat{\fol{}}$.
\end{enumerate}
\end{defn}

\section{\label{sec:Incompress_ND}Incompressibilité des singularités non
dégénérées solitaires}

Ici on étudie un feuilletage pris sous la forme~\eqref{eq:prepa_non-degenere}.
Soit $\mathcal{V}\subset\mathcal{B}$ un polydisque $\rho_{0}\ww D\times r_{0}\ww D$
choisi assez petit pour garantir que $R$ est holomorphe sur un voisinage
de son adhérence. On considère le domaine
\begin{eqnarray*}
\mathcal{V}_{}^{*} & := & \mathcal{V}\backslash\left\{ x=0\right\} 
\end{eqnarray*}
 sur lequel $\fol{}$ est partout transverses aux fibres de $\Pi$.
Dans tout l'article on travaillera sur le revêtement universel de
$\mathcal{V}^{*}$ donné par
\begin{eqnarray*}
\mathcal{E}\,:\,\tilde{\mathcal{V}} & \longrightarrow & \mathcal{V}^{*}\\
\left(z,y\right) & \longmapsto & \left(\exp z,y\right)\,.
\end{eqnarray*}
On chapeautera d'un \og \textasciitilde{} \fg{} les objets tirés
en arrière dans cette coordonnées, et définissons en particulier la
projection
\begin{eqnarray*}
\tilde{\Pi}\,:\,\left(z,y\right) & \longmapsto & z\,.
\end{eqnarray*}

Le but de cette section est de donner la preuve du Théorème~A quand
la singularité n'est pas un nœud\textendash{}col. Celle\textendash{}ci
découle du résultat quantitatif ci\textendash{}dessous.
\begin{prop}
\label{prop:lif_scnx_ND}Il existe $\rho_{0}>0$ et $r_{0}>0$ assez
petits tels que, en notant $\mathcal{U}\left(\rho,r\right):=\rho\ww D\times r\ww D$
pour chaque $0<\rho\leq\rho_{0}$ et $0<r\leq r_{0}$, toute feuille
de $\tilde{\fol{}}:=\mathcal{E}^{*}\fol{}|_{\mathcal{U}\left(\rho,r\right)}$
est simplement connexe. Le couple $\left(\rho_{0},r_{0}\right)$ convient
dès que
\begin{eqnarray*}
\sup_{\left(x,y\right)\in\mathcal{U}\left(\rho_{0},r_{0}\right)}\left|R\left(x,y\right)\right| & < & 1\,.
\end{eqnarray*}

\end{prop}
Notons que l'ensemble $\mathcal{S}$ des séparatrices de $\fol{}$
coïncide avec $\mathcal{B}\cap\left\{ xy=0\right\} $, de sorte que
les tubes de Milnor\textbf{ }$\mathcal{T}_{\eta}$ sont simples à
décrire. Bien sûr pour chaque $\left(\rho,r\right)$ on peut trouver
$\mathcal{T}_{\eta}\subset\mathcal{U}\left(\rho,r\right)\subset\mathcal{T}_{\rho}$
pour chaque $\eta$ assez petit (quitte à considérer une boule $\mathcal{B}$
assez petite fixée une fois pour toutes), et le groupe fondamental
de $\mathcal{U}\left(\rho,r\right)\backslash\mathcal{S}$ est isomorphe
à celui de $\mathcal{T}_{\rho}\backslash\mathcal{S}$. Le reste de
cette section est consacré à la preuve de la Proposition~\ref{prop:lif_scnx_ND}.

\subsection{Faisceau de stabilité et réduction de la preuve}

Par soucis de simplicité on écrira $\rho$ et $r$ à la place de $\rho_{0}$
et $r_{0}$, c'est\textendash{}à\textendash{}dire $\mathcal{V}=\mathcal{U}\left(\rho,r\right)$.
\begin{defn}
Étant donné un point $z_{*}\in\tilde{\Pi}\left(\tilde{\mathcal{V}}\right)$,
on appelle \textbf{faisceau de stabilité} de sommet $z_{*}$ et d'ouverture
$\frac{\pi}{2}>\delta>0$ la région de $\tilde{\mathcal{V}}$ contenant
$z_{*}$ donnée par 
\begin{eqnarray*}
\beam:= & \left\{ z_{*}-t\theta\frac{\lambda}{\left|\lambda\right|}\,:\,\left|\arg\theta\right|<\delta\,,\, t\geq0\right\} \cap\Pi\left(\tilde{\mathcal{V}}\right) & \,.
\end{eqnarray*}

\end{defn}
\begin{figure}[H]
\hspace*{\fill}\includegraphics[width=7cm]{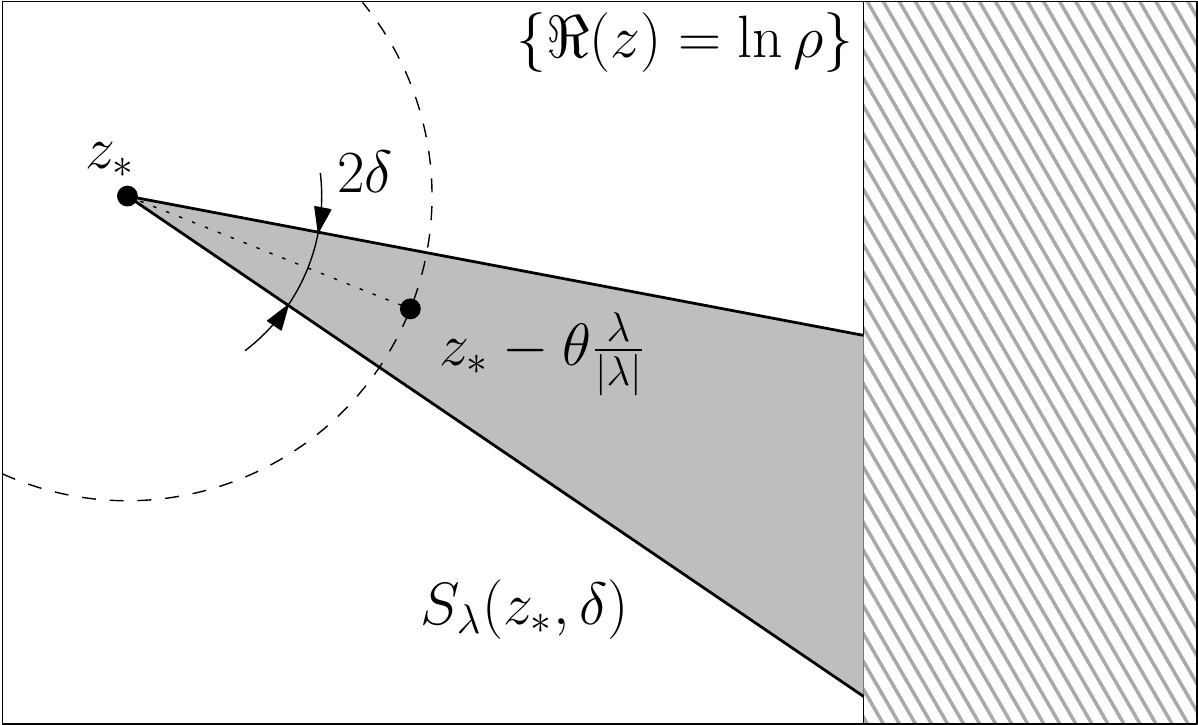}\hspace*{\fill}

\caption{\label{fig:stability_beam_ND}Un faisceau de stabilité (région grisée).}
\end{figure}

Le nom de faisceau de stabilité est justifié par le fait suivant: 
\begin{lem}
\label{lem:stability_beam_ND}Pour tous $\rho>0$ et $r>0$ tels que
\begin{eqnarray*}
M:=\sup_{\left(x,y\right)\in\mathcal{U}\left(\rho,r\right)}\left|R\left(x,y\right)\right| & < & 1
\end{eqnarray*}
on pose $\delta:=\arccos M$. Alors pour tout $p_{*}=\left(z_{*},y_{*}\right)\in\tilde{\mathcal{U}}$,
chaque chemin $\gamma$ basé en $z_{*}$ et d'image incluse dans $\beam$
se relève dans $\tilde{\fol{}}$ en s'appuyant sur $p_{*}$.\end{lem}
\begin{rem}
~
\begin{enumerate}
\item La conclusion du lemme ci\textendash{}dessus est bien que $\gamma$
se relève \emph{en totalité} dans la feuille de $\tilde{\fol{}}$
contenant $p$.
\item L'existence des faisceaux de stabilité impose une condition très forte
sur la forme du bord d'une feuille~: celui\textendash{}ci ne peut
pas être trop irrégulier (convexité conique). Il en résulte que le
revêtement universel d'une feuille typique ressemble à la Figure~\ref{fig:revet_univ_feuille_ND}.
\end{enumerate}
\end{rem}
\begin{figure}[H]
\hspace*{\fill}\includegraphics[width=7cm]{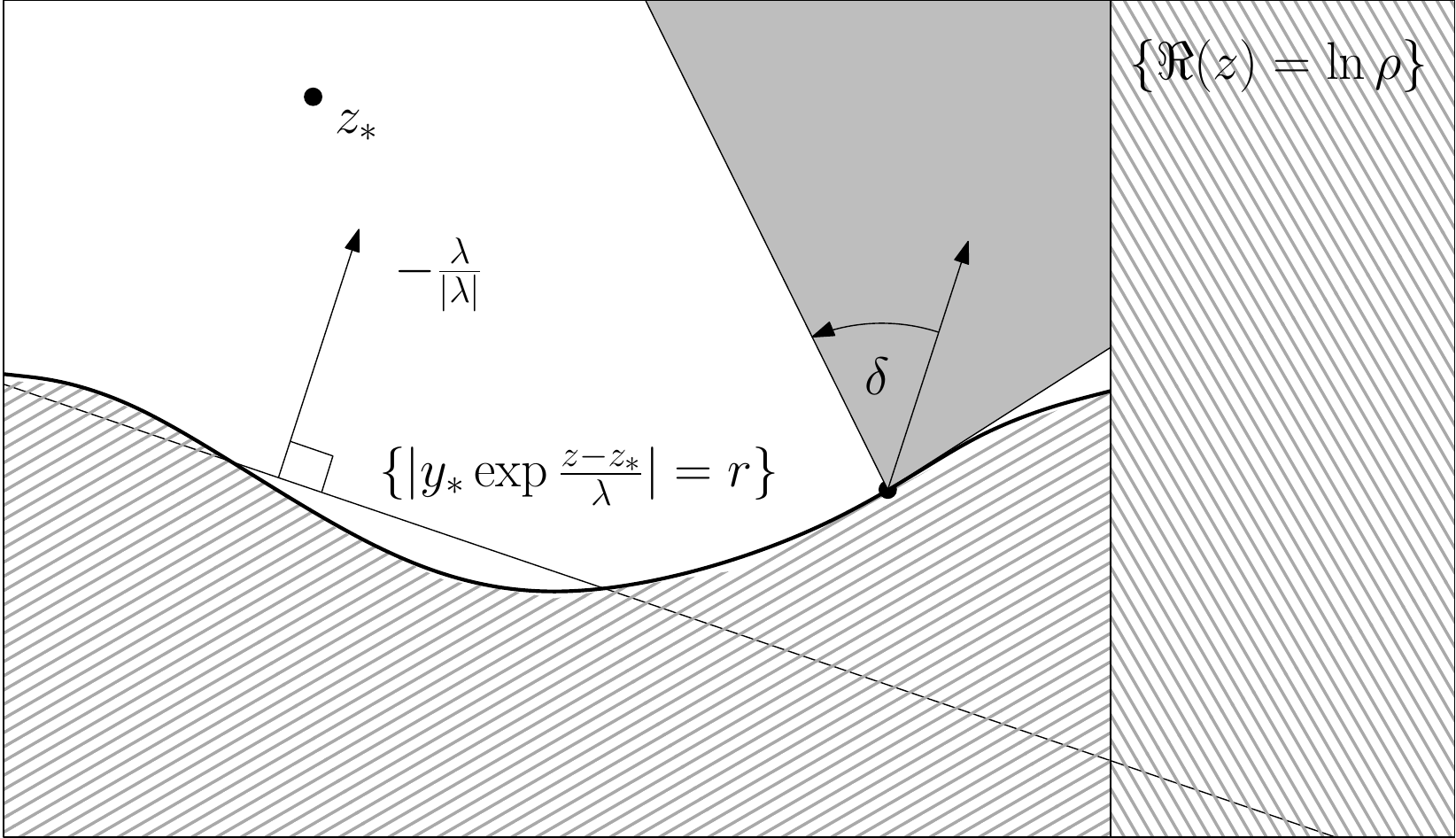}\hspace*{\fill}

\caption{\label{fig:revet_univ_feuille_ND}Le revêtement universel d'une feuille
typique, passant par $p_{*}=\left(z_{*},y_{*}\right)\in\tilde{\mathcal{U}}$
d'une singularité non dégénérée (complémentaire des régions hachurées).
La droite réelle $\left\{ \left|y_{*}\exp\frac{z-z_{*}}{\lambda}\right|=r\right\} $
représente la trace de la feuille du feuilletage linéaire $\tilde{\omega}_{0}$
sur le bord $\tilde{\Pi}\left(\tilde{\mathcal{U}}\right)\times r\ww S^{1}$.}
\end{figure}

Avant de donner la preuve du lemme nous achevons celle de la Proposition~\ref{prop:lif_scnx_ND}.
Pour un lacet $\tilde{\gamma}$ d'une feuille $\tilde{\lif{}}$ de
$\tilde{\fol{}}$ nous notons $\gamma:=\tilde{\Pi}\circ\tilde{\gamma}$
son projeté, dont l'image est un compact contenu dans une bande $\left\{ a\leq\re z\leq b\right\} $
pour deux réels $a\leq b<\ln\rho$. Tout faisceau de stabilité $\beam[\gamma\left(t\right)]$
coupe l'une ou l'autre des droites formant le bord de la bande, disons
$\left\{ \re z=a\right\} $ pour fixer les idées, en suivant une droite
parallèle de direction $\vartheta$ fixée. Pour $t\in\left[0,1\right]$
on considère une paramétrisation $s\in\left[0,1\right]\mapsto h_{s}\left(t\right)$
du segment reliant $\gamma\left(t\right)$ à $\left\{ \re z=a\right\} $
parallèlement à $\vartheta$. Alors $\left(h_{s}\right)_{s\in\left[0,1\right]}$
est une homotopie (libre) entre $\gamma$ et un chemin dont l'image
est un segment $I$. Puisque l'homotopie se déroule dans une union
de faisceaux de stabilité elle se relève en une homotopie dans $\tilde{\lif{}}$
entre $\tilde{\gamma}$ est un chemin tangent à une feuille de $\tilde{\fol{}}|_{\tilde{\Pi}^{-1}\left(I\right)}$.
Or ce dernier est un feuilletage réel lisse unidimensionnel et transverse
à la projection $\tilde{\Pi}$~: puisque $I$ est contractile ses
feuilles le sont également, et $\tilde{\gamma}$ est finalement homotopiquement
trivial dans $\tilde{\lif{}}$.

\subsection{\label{sub:decroissance_ND}Démonstration du Lemme~\ref{lem:stability_beam_ND}}

Le feuilletage $\tilde{\fol{}}$ est induit par la $1$\textendash{}forme
différentielle
\begin{eqnarray*}
\tilde{\omega}\left(z,y\right) & = & \lambda\dd z-y\left(1+\tilde{R}\left(z,y\right)\right)\dd y
\end{eqnarray*}
où par hypothèse
\begin{eqnarray*}
\tilde{R}\left(z,y\right) & := & R\left(\exp z,y\right)
\end{eqnarray*}
est bornée par $1$ sur $\tilde{\mathcal{V}}$. Prenons $\left(z_{*},y_{*}\right)\in\tilde{\lif{}}$
et pour $\theta\in\ww S^{1}$, $t\geq0$ on note 
\begin{eqnarray*}
z_{\theta}\left(t\right) & := & z_{*}-t\theta\frac{\lambda}{\left|\lambda\right|}\,.
\end{eqnarray*}
Puisque les feuilles de $\tilde{\fol{}}$ sont transverses aux fibres
de la projection $\tilde{\Pi}$, le bord de $\tilde{\lif{}}$ est
contenu dans le bord $\left\{ \left|y\right|=r\mbox{ ou }\re z=\ln\rho\right\} $.
Pour garantir que les chemins $t\geq0\mapsto z_{\theta}\left(t\right)$
des faisceaux de stabilité se relèvent dans $\tilde{\lif{}}$ en $t\geq0\mapsto\left(z_{\theta}\left(t\right),y_{\theta}\left(t\right)\right)$,
il suffit d'assurer que $\left|y_{\theta}\right|$ est une fonction
décroissante. Étudions\textendash{}donc les variations de 
\begin{eqnarray*}
\varphi\left(t\right) & := & \ln\left|y_{\theta}\left(t\right)\right|=\frac{1}{2}\ln\left(y_{\theta}\left(t\right)\overline{y_{\theta}\left(t\right)}\right)\,.
\end{eqnarray*}
En utilisant la relation 
\begin{eqnarray*}
\frac{\dot{y_{\theta}}}{\dot{z_{\theta}}} & = & \frac{y_{\theta}\left(1+\tilde{R}\left(z_{\theta},y_{\theta}\right)\right)}{\lambda}
\end{eqnarray*}
on trouve
\begin{eqnarray*}
\dot{\varphi} & = & \re{-\frac{\theta}{\left|\lambda\right|}\left(1+\tilde{R}\left(z_{\theta},y_{\theta}\right)\right)}\,.
\end{eqnarray*}
Maintenant le membre de droite est strictement négatif dès que $\left|\arg\theta\right|<\delta$.
Dès lors, tout rayon $t\ge0\mapsto z_{\theta}\left(t\right)$ du faisceau
de stabilité $\beam$ se relève dans la feuille $\tilde{\lif{}}$,
puisque $\left|y_{\theta}\right|$ est décroissante, tant que $\re{z_{\theta}\left(t\right)}$
reste plus petit que $\ln\rho$. La conclusion suit.

\section{\label{sec:Incompress_NC}Incompressibilité du nœud\textendash{}col
solitaire}

On répète les arguments développés en section précédente dans le cas
d'un nœud\textendash{}col. Ici encore le Théorème~A découle de la
\begin{prop}
\label{prop:lif_scnx_NC}Il existe $\rho_{0}>0$ et $r_{0}>0$ assez
petits, et pour tous $0<\rho\leq\rho_{0}$ et $0<r\leq r_{0}$ un
domaine $\mathcal{U}\left(\rho,r\right)\subset\mathcal{V}$, fibré
en disques au\textendash{}dessus de $\rho\ww D\times\left\{ 0\right\} $,
tels que les proposition suivantes tiennent.
\begin{enumerate}
\item Chaque feuille de $\tilde{\fol{}}:=\mathcal{E}^{*}\fol{}|_{\mathcal{U}\left(\rho,r\right)}$
est simplement connexe. 
\item Si le nœud\textendash{}col est convergent, alors $\mathcal{U}\left(\rho,r\right)$
est le polydisque standard et $\left(r_{0},\rho_{0}\right)$ convient
dès que
\begin{eqnarray*}
\sup_{\left(x,y\right)\in\mathcal{U}\left(r_{0},\rho_{0}\right)}\left|R\left(x,y\right)\right| & < & 1\,.
\end{eqnarray*}

\end{enumerate}
\end{prop}
L'ensemble $\mathcal{S}$ des séparatrices de $\fol{}$ est inclus
dans $\mathcal{B}\cap\left\{ xy=0\right\} $ et les tubes de Milnor\textbf{
}$\mathcal{T}_{\eta}$ sont encore simples à décrire. Le reste de
cette section est consacré à la preuve de la Proposition~\ref{prop:lif_scnx_NC}.

\subsection{Faisceau de stabilité et réduction de la preuve}

Par soucis de simplicité on écrira $\rho$ et $r$ à la place de $\rho_{0}$
et $r_{0}$. Fixons un point $p_{*}:=\left(z_{*},y_{*}\right)\in\tilde{\mathcal{V}}$.
Pour $\theta\in\ww S^{1}$ on construit le chemin $z_{\theta}\,:\, t\geq0\mapsto z_{\theta}\left(t\right)$
solution de
\begin{eqnarray}
\dot{z}_{\theta}\left(t\right) & = & -\theta\exp\left(kz_{\theta}\left(t\right)\right)\label{eq:iso_arg}
\end{eqnarray}
avec la condition initiale $z_{\theta}\left(0\right)=z_{*}$. 

\begin{figure}[H]
\hspace*{\fill}\includegraphics[scale=0.9]{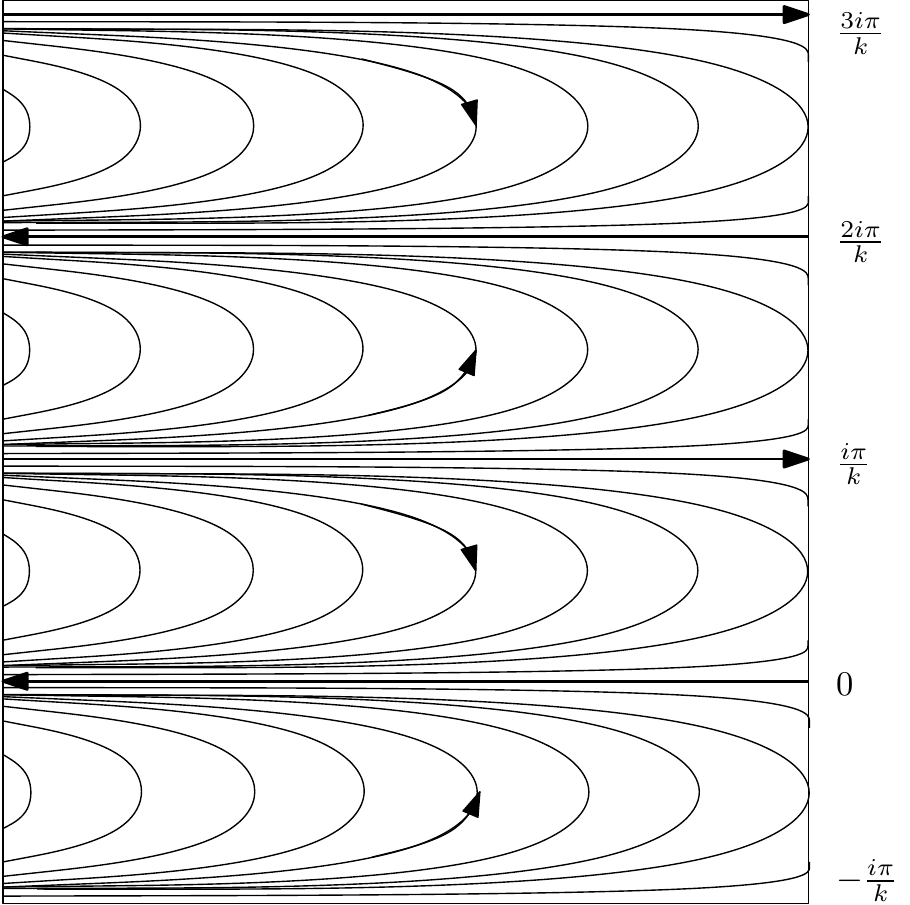}\hspace*{\fill}

\caption{\label{fig:z_1}Les courbes intégrales $z_{1}$. On obtient $z_{\theta}$
en translatant $z_{1}$ de $-\ii\frac{\arg\theta}{k}$}
\end{figure}

On a la relation implicite
\begin{eqnarray*}
\exp\left(kz_{\theta}\left(t\right)\right) & = & \frac{\exp\left(kz_{*}\right)}{1+k\theta t\exp\left(kz_{*}\right)}\,.
\end{eqnarray*}
On dispose également de l'asymptotique 
\begin{eqnarray*}
\begin{cases}
\re{z_{\theta}\left(t\right)} & \sim_{t\to+\infty}-\frac{1}{k}\ln t\\
\im{z_{\theta}\left(t\right)} & \sim_{t\to+\infty}\im{z_{*}}-\frac{1}{k}\arg\theta
\end{cases} &  & \,.
\end{eqnarray*}
Cependant il se peut qu'avant de tendre vers $-\infty$ la partie
réelle de $z_{\theta}$ dépasse $\ln\rho$, c'est en particulier le
cas lorsque $\theta\exp\left(kz_{*}\right)<0$ puisqu'alors $\exp\left(kz_{\theta}\right)$
admet un pôle, et donc la trajectoire $z_{\theta}$ quitte le domaine
$\left\{ \re z<\ln\rho\right\} $.
\begin{defn}
On appelle \textbf{faisceau de stabilité} de sommet $z_{*}$ et d'ouverture
$\frac{\pi}{2}>\delta>0$ la région de $\tilde{\mathcal{U}}\left(\rho,r\right)$
contenant $z_{*}$ donnée par
\begin{eqnarray*}
\beam & := & \left\{ z_{\theta}\left(t\right)\,:\,\left|\arg\theta\right|<\delta\,,\, t\geq0\,,\,\left(\forall\tau\leq t\right)\, z_{\theta}\left(\tau\right)\in\tilde{\mathcal{U}}\right\} \,.
\end{eqnarray*}

\end{defn}
\begin{figure}[H]
\hfill{}\subfloat{\includegraphics[width=6.5cm]{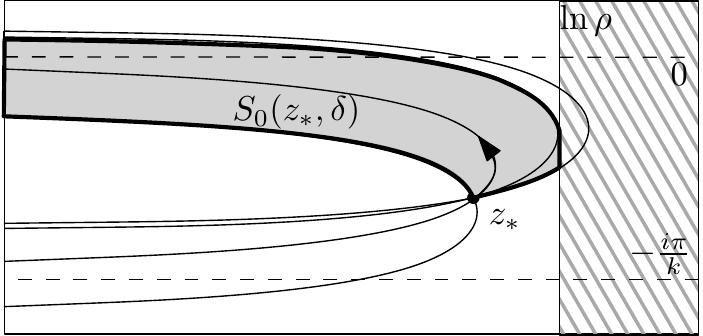}}\hfill{}

\subfloat{\includegraphics[width=6.5cm]{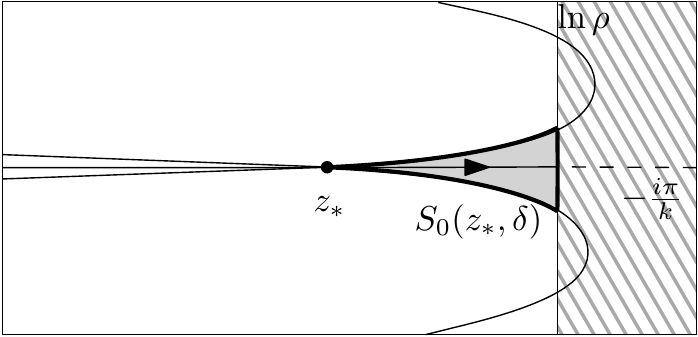}}\hfill{}\subfloat{\includegraphics[width=6.5cm]{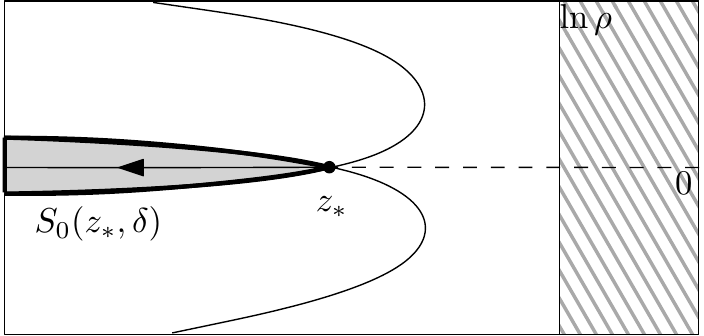}}

\caption{Des faisceaux de stabilité (région uniformément grisée).}
\end{figure}

\begin{lem}
\label{lem:stability_beam_NC}Il existe un domaine $\mathcal{U}\left(\rho,r\right)$,
fibré en $r\ww D$ au\textendash{}dessus de $\rho\ww D\times\left\{ 0\right\} $,
tel que, pour tout $p_{*}=\left(z_{*},y_{*}\right)\in\tilde{\mathcal{U}}\left(r,\rho\right)$,
chaque chemin $\gamma$ basé en $z_{*}$ et inclus dans $\beam[z_{*}][\delta]$
se relève dans $\tilde{\fol{}}$ en s'appuyant sur $p$. Si le nœud\textendash{}col
est convergent alors $\mathcal{U}\left(\rho,r\right)$ est le polydisque
standard.\end{lem}
\begin{rem*}
L'existence des faisceaux de stabilité impose ici aussi une condition
très forte sur la régularité du bord d'une feuille. Il en résulte
que le revêtement universel d'une feuille typique ressemble à la Figure~\ref{fig:revet_univ_feuille}.
La présence des \og langues \fg{} d'étendue infinie sur la gauche
de la figure provient du comportement \og col \fg{}. En effet dans
les secteurs $\left\{ \left|\arg\left(x^{k}\right)-\pi\right|<\frac{\pi}{2}\right\} $,
délimités par des pointillés sur la figure, l'ordonnée de la feuille
est de l'ordre de $\exp\frac{-1}{kx^{k}}$ et tend fortement vers
l'infini quand $x$ se rapproche du point singulier. Au contraire
dans les secteurs \og nœuds \fg{} $\left\{ \left|\arg\left(x^{k}\right)\right|<\frac{\pi}{2}\right\} $
les feuilles tendent platement vers $0$ (voir par exemple~\citep{Tey-SN}).
Nous donnons plus de détails en Section~\ref{sub:blocs_adapt_cont}
concernant la topologie du bord de $\tilde{\Pi}\left(\tilde{\lif{}}\right)$.
\end{rem*}
\begin{figure}[H]
\hspace*{\fill}\includegraphics[width=7cm]{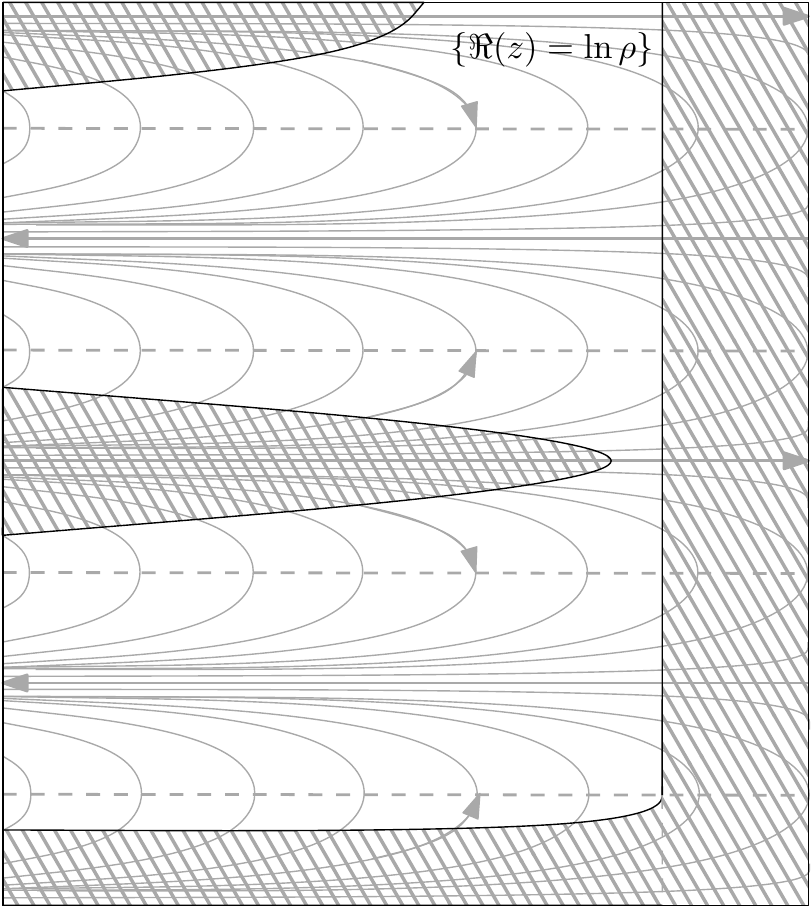}\hspace*{\fill}

\caption{\label{fig:revet_univ_feuille}Le revêtement universel d'une feuille
typique d'un nœud\textendash{}col (complémentaire des régions hachurées).}
\end{figure}

Ce lemme sera prouvé dans les prochaines sous\textendash{}sections
en deux étapes: d'abord le cas des nœuds\textendash{}cols convergents
puis celui des divergents. Expliquons au préalable en quoi il suffit
à garantir la simple\textendash{}connexité d'une feuille $\tilde{\lif{}}$.
La stratégie est similaire à celle introduite pour les singularités
non dégénérées~: il suffit de construire une homotopie (libre) $\left(h_{s}\right)_{s\in\left[0,1\right]}$
entre la projection $\gamma:=\tilde{\Pi}\circ\tilde{\gamma}$ d'un
cycle tangent $\tilde{\gamma}$ et un chemin bordant une région d'intérieur
vide, telle que pour $t$ fixé $s\mapsto h_{s}\left(t\right)$ est
un chemin de $\beam[h_{0}\left(t\right)][\delta]$, qui se relève
donc dans $\tilde{\lif{}}$. Ceci achève la preuve de la Proposition~\ref{prop:lif_scnx_NC}
en invoquant encore une fois l'argument d'un feuilletage réel unidimensionnel
au\textendash{}dessus d'un compact contractile. Le fait que les faisceaux
de stabilité ne sont pas tous dirigés selon une direction donnée à
l'avance constitue la difficulté supplémentaire par rapport au cas
non dégénéré.

\bigskip{}

Le procédé est illustré en Figure~\ref{fig:homotopy_lif}, et correspond
aux deux étapes suivantes.
\begin{itemize}
\item Les parties de l'image de $\gamma$ contenues dans les bandes \og col \fg{}
$\left\{ \cos\left(kz\right)\leq0\right\} $ sont envoyées, en suivant
les chemins $z_{1}$, dans $\left\{ \cos\left(kz\right)=0\right\} \cup\left\{ \re z=\ln\rho\right\} $.
\item Les parties de l'image de $\gamma$ contenues dans les bandes \og nœud \fg{}
$\left\{ \cos\left(kz\right)\geq0\right\} $ sont envoyées, en suivant
les chemins $z_{1}$, dans $\left\{ \re z=a\right\} \cup\Gamma$ où
$\Gamma$ est l'union des images des trajectoires $z_{1}$ issues
des points $\ln\rho+\ii\nf{\pi}{2k}+\ii\nf{\pi}k\ww Z$.
\end{itemize}
\begin{figure}[H]
\hspace*{\fill}\includegraphics[width=12cm]{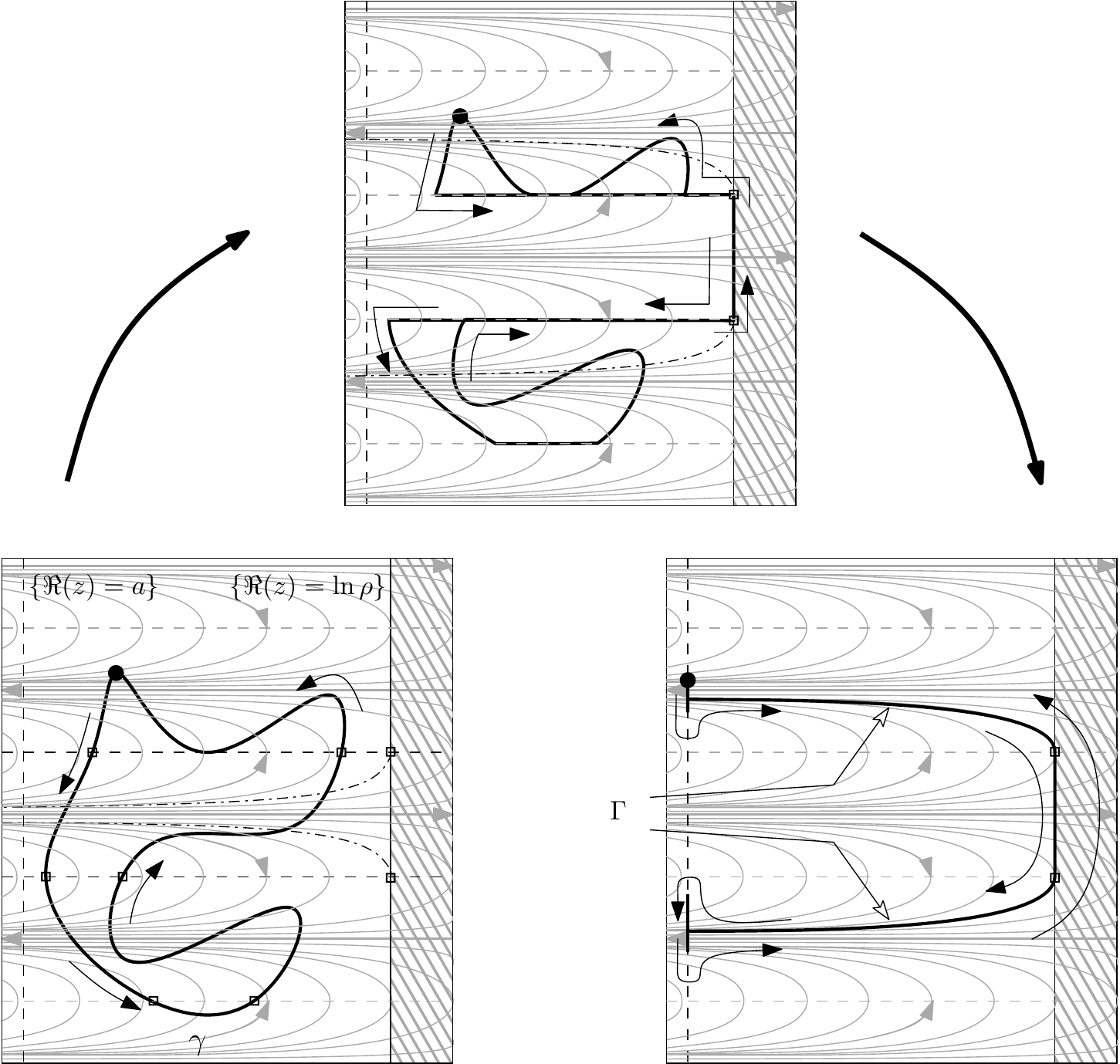}\hspace*{\fill}

\caption{\label{fig:homotopy_lif}}
\end{figure}

\subsection{Le cas convergent}

On prend ici $\mathcal{U}$ sous la forme d'un polydisque $r\ww D\times\rho\ww D$
avec $r,\rho>0$. Dans ce cas on a 
\begin{eqnarray}
\tilde{\omega}_{yR} & := & \mathcal{E}^{*}\omega_{yR}\nonumber \\
 & = & \exp\left(kz\right)\dd y-y\left(1+\tilde{R}\left(z,y\right)\right)\dd z\,.\label{eq:NC_lift}
\end{eqnarray}

Puisque les feuilles de $\tilde{\fol{}}$ sont transverses aux fibres
de la projection $\tilde{\Pi}$, le bord de $\lif p$ est contenu
dans le bord $\left\{ \left|y\right|=r\mbox{ ou }\re z=\ln\rho\right\} $.
Pour garantir que les chemins $t\geq0\mapsto z_{\theta}\left(t\right)$
des faisceaux de stabilité se relèvent dans $\lif p$ en $t\geq0\mapsto\left(z_{\theta}\left(t\right),y_{\theta}\left(t\right)\right)$,
il suffit d'assurer que $\left|y_{\theta}\right|$ est une fonction
décroissante. Mais ceci découle encore des variations de $\varphi:=\ln\left|y_{\theta}\right|$,
à savoir
\begin{eqnarray*}
\dot{\varphi} & = & \re{-\theta\left(1+\tilde{R}\left(z_{\theta},y_{\theta}\right)\right)}\,.
\end{eqnarray*}
Le Lemme~\ref{lem:stability_beam_NC} est donc démontré.

\subsection{Cas divergent}

On va montrer que $\beam[z_{*}][\delta]$ est encore un faisceau de
stabilité à condition de modifier $\mathcal{U}$. Puisque $z_{\theta}=z_{1}-\ii\frac{\arg\theta}{k}$
la propriété suivante est vraie.
\begin{lem}
\label{lem:sectorial_beam}Pour $j\in\ww Z$ on note $\tsect$ la
bande
\begin{eqnarray*}
\tsect & := & \left\{ z\,:\,\re z<\ln\rho\,,\,\left|\im z-\left(2j+1\right)\frac{\pi}{k}\right|<\frac{\pi}{k}+\beta\right\} \,,
\end{eqnarray*}
dont l'image par $z\mapsto x=\exp z$ coïncide avec $\sect$. Étant
donnés $0<\beta'<\beta<\frac{\pi}{2k}$ on peut choisir $\delta>0$
suffisamment petit de sorte que, pour tout $z_{*}\in\tsect[][\beta']$,
le faisceau de stabilité $\beam$ soit inclus dans $\tsect$. 
\end{lem}
En opérant le changement de variables sectoriel pour $j\in\zsk$ 
\begin{eqnarray*}
S_{j}\,:\,\left(z,y\right)\in\tilde{V}_{j}^{\beta} & \longmapsto & \left(z,y-\tilde{s}_{j}\left(\exp z\right)\right)\,,
\end{eqnarray*}
qui redresse la séparatrice sectorielle sur $\left\{ y=0\right\} $,
on se ramène au cas précédent (en effet la preuve du Lemme~\ref{lem:stability_beam_NC}
dans le cas convergent n'utilise pas le caractère holomorphe de la
perturbation $yR$). On prend pour $\mathcal{U}\left(r,\rho\right)$
l'intérieur de l'adhérence du domaine suivant:
\begin{eqnarray*}
 & \mathcal{E}\left(\bigcup_{j\in\zsk}S_{j}^{-1}\left(\sect\times r\ww D\right)\right)\,.
\end{eqnarray*}
 Comme la famille $\left(\tsect[][\beta']\right)_{j\in\ww Z}$ recouvre
$\tilde{\Pi}\left(\tilde{\mathcal{V}}\right)$ et $S_{j}$ est fibrée
en la variable $z$, le Lemme~\ref{lem:stability_beam_NC} est maintenant
démontré.

\section{\label{sec:Construction_originale}Résumé de la construction de Mar\'{i}n\textendash{}Mattei}

La construction consiste à recoller une quantité finie de blocs locaux
$\left(\mathcal{B}_{\alpha}\right)_{\alpha\in\mathcal{A}}$ contenant
les singularités réduites de la réduction de $\fol{}$, deux\textendash{}à\textendash{}deux
disjoints, avec des blocs réguliers recouvrant le reste du diviseur
exceptionnel, afin de localiser (grâce à un théorème de type Van~Kampen)
la propriété d'incompressibilité recherchée. Ces blocs ne peuvent
être arbitraires, mais doivent au contraire satisfaire de «bonnes»
propriétés, que nous détaillons plus bas, et posséder au moins un
degré de liberté dans leur construction (la \og taille \fg{}) gouverné
par une fonction de contrôle. L'assemblage des blocs se fait par induction,
en choisissant une composante du diviseur puis en parcourant l'arbre
de réduction de proche en proche. On arrive dans une nouvelle composante
$D_{j+1}$ par le passage d'un coin, muni d'une taille de sortie spécifiée
sur $D_{j}$, et on assemble les autres blocs présents sur $D_{j+1}$
à celui\textendash{}ci en ajustant leur taille. On passe alors à une
composante adjacente à $D_{j+1}$ si toutes n'ont pas été visitées. 

Tout ceci n'est que partiellement exact, car les blocs réguliers qui
englobent les branches mortes (Section~\ref{sub:Comp-init}) n'ont
aucun degré de liberté~: leur taille est imposée. Une obstruction
d'ordre technique apparaît donc quand $D_{j+1}$ est attachée a au
moins deux branches mortes et ne possède d'autre singularité que celle
partagée avec $D_{j}$, autrement dit une composante initiale. On
doit donc démarrer l'induction par une composante initiale (Section~\ref{sub:Comp-init}),
imposant une taille à sa descendance. Dans le cas d'une courbe généralisée
non dicritique il est bien connu~\citep[p866]{MarMat} qu'il n'existe
qu'au plus une composante initiale, ce qui rend possible l'induction.
De plus la composante initiale est attachée à exactement deux branches
mortes, l'une d'elles ayant comme extrémité le diviseur apparu lors
du premier éclatement. Nous généralisons mot pour mot cette propriété
à toutes les singularités de feuilletages fortement présentables en
Section~\ref{sec:Composantes-initiales}.

\subsection{\label{sub:decoup}Découpage du diviseur exceptionnel en blocs}

Soit $\fol{}$ un feuilletage fortement présentable non réduit. Notons
\begin{eqnarray*}
\hat{\fol{}} & := & E^{*}\fol{}
\end{eqnarray*}
le feuilletage réduit et $\hat{\mathcal{S}}$ l'union des transformées
strictes par $E$ de chaque composante irréductible de $\mathcal{S}$.
Autour de chaque singularité $s$ de $\hat{\fol{}}$ on se donne un
système de coordonnées locales comme en Section~\ref{sec:Notations},
c'est\textendash{}à\textendash{}dire un biholomorphisme $\psi_{s,D}\,:\,\mathcal{U}_{s,D}\to\rho\ww D\times r\ww D$
qui transforme $\hat{\fol{}}$ en le feuilletage donné par $\omega_{R}=0$
et qui envoie une composante du diviseur $D\ni s$ sur $\left\{ y=0\right\} $.
On choisit $\left(\rho_{0},r_{0}\right)$ une fois pour toute de façon
à ce que les Propositions~\ref{prop:lif_scnx_ND} et~\ref{prop:lif_scnx_NC}
tiennent pour toutes les singularités réduites. Nous déterminerons
de proche en proche les valeurs de $0<\rho_{s}\leq\rho_{0}$ et $0<r_{s}\leq r_{0}$
pour lesquelles la construction suivante s'appliquera. 

On notera $\Gamma_{s,D}$ le cercle conforme $\partial\mathcal{U}_{s,D}\cap D$,
pré\textendash{}image par $\psi_{s}$ de $\rho_{s}\ww S^{1}\times\left\{ 0\right\} $.
Dans toute la suite on suppose que le réel positif $\eta_{0}$, paramétrant
l'étendue de la famille de tubes de Milnor $\left(\mathcal{T}_{\eta}\right)_{0<\eta\leq\eta_{0}}$,
est inférieur à une «hauteur d'uniformité des blocs de Milnor» $\eta_{1}>0$
(voir~\citep[Section 2.2, p867]{MarMat}). La valeur de $\eta_{1}$
est telle que les composantes connexes de $T:=E^{-1}\left(\mathcal{T}_{\eta}\right)\backslash\bigcup_{s\in D}\psi_{s,D}^{-1}\left(\rho_{s}\ww S^{1}\times r_{s}\ww D\right)$
sont en bijection avec l'adhérence $\left(K_{\alpha}\right)_{\alpha\in\mathcal{A}}$
des composantes de $E^{-1}\left(0\right)\backslash\bigcup_{s\in D}\Gamma_{s,D}$.
On note alors $\mathcal{T}_{\eta}\left(K_{\alpha}\right)$ les composantes
de $T$ correspondant à $K_{\alpha}$.

Chaque $K_{\alpha}$ est un compact dont le bord est une union finie
de cercles conformes $\mathcal{C}_{\alpha,j}$ paramétrés par des
lacets simples analytiques $\gamma_{\alpha,j}$. À chaque composante
initiale correspond un unique $K_{\alpha}$, celui contant les deux
points d'attache des branches mortes, qui sera lui aussi qualifié
d'initial.
\begin{defn}
Le parcours des $\left(K_{\alpha}\right)_{\alpha\in\mathcal{A}}$
dans l'induction identifie de manière unique, pour chaque $K_{\alpha}$
non initial, une composante $\mathcal{C}_{\alpha,\ell}$ appelée \textbf{composante
d'entrée }de $K_{\alpha}$. Celle\textendash{}ci correspond à la composante
du bord de $K_{\alpha}$ qui sera recollée avec la construction réalisée
à l'étape précédente.
\end{defn}
Les blocs $\mathcal{B}_{\alpha}$ que nous allons décrire seront des
sous\textendash{}ensembles bien particuliers fibrés au\textendash{}dessus
des $K_{\alpha}\backslash\hat{\mathcal{S}}$.

\subsection{Blocs adaptés feuilletés}
\begin{defn}
\label{def_1_connexe}\citep[Définition 1.2.2, p861]{MarMat}Prenons
un feuilletage $\fol{}$ sur un domaine $\mathcal{U}$ et soient $A\subset B$
deux sous\textendash{}ensembles de $\mathcal{U}$. On dit que $A$
est\textbf{ 1\textendash{}connexe} dans $B$ (relativement à $\fol{}$)
si pour chaque feuille $\lif{}$ de $\fol{}$ et tout chemins $\alpha$
de $A$ et $\beta$ de $B\cap\lif{}$ qui sont homotopes dans $B$,
il existe un chemin de $A\cap\lif{}$ homotope à la fois à $\alpha$
dans $A$ et à $\beta$ dans $B\cap\lif{}$.
\end{defn}
Afin de simplifier le présent texte nous ne parlons que de «blocs»,
sans distinguer les deux types «élémentaires» et «fondamentaux» (ces
derniers regroupant plusieurs blocs élémentaires contigus recouvrant
une branche morte). Cette distinction est d'ordre technique et peut
être ignorée ici (voir \citep[Section 2.2, p867]{MarMat} pour plus
de détails).
\begin{defn}
\label{def_blocs}\citep[Definition 2.1.1, p864]{MarMat}La notation
$\partial A$ représente le fermé $\adh{A\backslash\tx{int}\left(A\right)}$.
\begin{enumerate}
\item Nous dirons que $\mathcal{B}_{\alpha}$ est un \textbf{blocs feuilleté
adapté }si les conditions suivantes sont vérifiées~:

\begin{lyxlist}{00.00.0000}
\item [{(BF1)}] chaque composante connexe de $\partial\mathcal{B}_{\alpha}$
est incompressible dans $\mathcal{B}_{\alpha}$,
\item [{(BF2)}] $\hat{\fol{}}$ est transverse à $\partial\mathcal{B}_{\alpha}$
\item [{(BF3)}] $\hat{\fol{}}$ est incompressible dans $\mathcal{B}_{\alpha}$,
\item [{(BF4)}] chaque composante connexe de $\partial\mathcal{B}_{\alpha}$
est $1$\textendash{}connexe dans $\mathcal{B}_{\alpha}$,
\end{lyxlist}
\item Nous dirons que la collections de blocs $\left(\mathcal{B}_{\alpha}\right)_{\alpha\in\mathcal{A}}$
réalise un \textbf{assemblage bord\textendash{}à\textendash{}bord
}de\textbf{ }blocs feuilletés adaptés\textbf{ }si

\begin{itemize}
\item pour chaque $\alpha,\,\beta\in\mathcal{A}$ l'intersection $\mathcal{B}_{\alpha}\cap\mathcal{B}_{\beta}$
est ou bien vide, ou bien une composante connexe de $\partial\mathcal{B}_{\alpha}$
et de $\partial\mathcal{B}_{\beta}$,
\item $E\left(\bigcup_{\alpha\in\mathcal{A}}\mathcal{B}_{\alpha}\right)$
est un voisinage de la singularité $\mathcal{U}$ épointé de $\mathcal{S}$.
\end{itemize}
\end{enumerate}
\end{defn}
Nous énonçons le théorème de localisation~~:
\begin{thm}
\emph{\citep[Théorème 2.1.2, p864]{MarMat}} Si $\mathcal{U}$ est
un assemblage bord\textendash{}à\textendash{}bord de blocs feuilletés
adaptés alors chaque feuille de $\hat{\fol{}}$ est incompressible
dans $\left(\bigcup_{\alpha\in\mathcal{A}}\mathcal{B}_{\alpha},\mathcal{S}\right)$.
\end{thm}

\subsection{Rugosité des ensembles de type suspension}
\begin{defn}
\label{def_suspension}\citep[Definition 3.1.1, p869]{MarMat}Une
composante $B$ de $\partial\mathcal{B}_{\alpha}$ est de \textbf{type
suspension} au\textendash{}dessus de $\mathcal{C}_{\alpha,j}$ s'il
existe un disque analytique fermé $\Sigma$ inclus dans $\Pi^{-1}\left(\gamma_{\alpha,j}\left(0\right)\right)\cap B$,
centré en $\gamma_{\alpha,j}\left(0\right)$ et sur lequel l'holonomie
$\holo{\gamma_{\alpha,j}}$ est holomorphe, tels que $B$ soit l'union
des images des chemins tangents servant à calculer l'holonomie à partir
de $\Sigma$. On écrit alors
\begin{eqnarray*}
B & = & \tx{Susp}_{\gamma_{\alpha,j}}\left(\Sigma\right)\,.
\end{eqnarray*}

\end{defn}
Afin de ne pas créer de topologie artificielle dans l'espace ambiant
il faut assurer que l'intersection $\Sigma\cap\holo{\gamma_{\alpha,j}}\left(\Sigma\right)$
est connexe. Pour cela on introduit une notion de taille contrôlée,
garantissant que cette intersection est étoilée par rapport à son
centre.
\begin{defn}
\label{def_rugo}Pour $\theta\in\ww R$ on pose
\begin{eqnarray*}
\left\{ \left\{ \theta\right\} \right\}  & := & \begin{cases}
\left|\hat{\theta}\right| & \mbox{ si }\theta\in\hat{\theta}+2\pi\ww Z\mbox{ avec }\left|\hat{\theta}\right|<\frac{\pi}{2}\\
\infty & \mbox{ sinon}
\end{cases}\,.
\end{eqnarray*}

\begin{enumerate}
\item Soit $\gamma\,:\,\left[0,1\right]\to\ww C_{\neq0}$ un chemin analytique
lisse par morceaux.

\begin{enumerate}
\item En un point lisse $\gamma'\left(t\right)\neq0$ on définit 
\begin{eqnarray*}
\mathbf{e}\left(\gamma;t\right) & := & \left\{ \left\{ \arg\left(\frac{\dot{\gamma}}{\ii\gamma}\left(t\right)\right)\right\} \right\} \,.
\end{eqnarray*}
Cette quantité admet une limite à gauche et à droite en tout $t\in\left[0,1\right]$,
et on notera $\mathbf{e}\left(\gamma;t\right)$ la plus grande des
deux.
\item Définissons finalement la \textbf{rugosité }de $\gamma$ par 
\begin{eqnarray*}
\mathbf{e}\left(\gamma\right) & := & \max\mathbf{e}\left(\gamma;\left[0,1\right]\right)\,.
\end{eqnarray*}

\end{enumerate}
\item La rugosité est visiblement une propriété intrinsèque de la courbe
orientée paramétrée par $\gamma$. Si $\Gamma$ est l'image du chemin
$\gamma$ on posera alors
\begin{eqnarray*}
\mathbf{e}\left(\Gamma\right) & := & \min\left\{ \mathbf{e}\left(\gamma\right),\mathbf{e}\left(\gamma^{-}\right)\right\} \,,
\end{eqnarray*}
où $\gamma^{-}$ est le chemin d'orientation opposée à $\gamma$.
Si $\Delta$ est un voisinage simplement connexe de $0\in\ww C$,
de bord analytique par morceaux, et si $\mathbf{e}\left(\partial\Delta\right)<\infty$,
alors $\Delta$ est étoilé par rapport à $0$.
\item On appelle \textbf{fonction de contrôle} d'un ensemble de type suspension
$B:=\tx{Susp}_{\gamma}\Sigma$ au\textendash{}dessus d'un lacet simple
$\gamma$, analytique par morceaux, l'élément de $\left[0,\infty\right]$
donné par 
\begin{eqnarray*}
\frak{c}\left(B\right) & := & \max\left\{ \,\mathbf{e}\left(\partial\adh{f\circ E\left(\Sigma\right)}\right)\,,\,\left|\left|B\right|\right|_{\mathcal{S}}\,\right\} 
\end{eqnarray*}
où $\left|\left|B\right|\right|_{\mathcal{S}}:=\sup\left|f\circ E\left(\Sigma\right)\right|$
est la \textbf{taille} de $B$ vis\textendash{}à\textendash{}vis de
la fibration définie par l'équation des séparatrices distinguées  $\mathcal{S}=\left\{ f=0\right\} $.
\end{enumerate}
\end{defn}

\subsection{Blocs adaptés feuilletés contrôlés}
\begin{defn}
\label{def_bloc_controle}\citep[Théorème 3.2.1, p871]{MarMat}Prenons
deux réels $\eta_{0}\leq\eta_{1}$ et $\varepsilon>0$, ainsi que
$\alpha\in\mathcal{A}$ correspondant à une composante $D$ du diviseur
exceptionnel.
\begin{enumerate}
\item On dira que $\mathcal{B}_{\alpha}\subset\mathcal{T}_{\eta_{0}}\left(K_{\alpha}\right)$
est un \textbf{bloc feuilleté contrôlé} si~:

\begin{lyxlist}{00.00.0000}
\item [{(BC1)}] pour tout $\eta_{0}\geq\eta>0$ assez petit $\mathcal{T}_{\eta}\left(K_{\alpha}\right)\subset\mathcal{B}_{\alpha}$
et les inclusions induisent des isomorphismes $\pi_{1}\left(\mathcal{T}_{\eta}\left(K_{\alpha}\right)\backslash D\right)\simeq\pi_{1}\left(\mathcal{B}_{\alpha}\backslash D\right)$
et $\pi_{1}\left(\partial\mathcal{T}_{\eta}\left(K_{\alpha}\right)\backslash D\right)\simeq\pi_{1}\left(\partial K_{\alpha}\backslash D\right)$,
\item [{(BC2)}] $\mathcal{B}_{\alpha}$ est un bloc feuilleté adapté,
\item [{(BC3)}] chaque composante $B_{j}$ de $\partial\mathcal{B}_{\alpha}$
est de type suspension au\textendash{}dessus de $\mathcal{C}_{\alpha,j}$.
\item [{(BC4)}] $\frak{c}\left(B_{j}\right)\leq\varepsilon$.
\end{lyxlist}
\item On dira qu'une composante non initiale $K_{\alpha}$ est \textbf{contrôlable}
s'il existe $c_{\alpha}>0$ et une fonction croissante $\frak{d}_{\alpha}\,:\,\ww R_{>0}\to\ww R_{>0}$,
de limite nulle en $0$, telles que pour tout sous\textendash{}ensemble
$B\subset\mathcal{T}_{\eta}\left(K_{\alpha}\right)$ de type suspension
au\textendash{}dessus de la composante d'entrée $\mathcal{C}_{\alpha,\ell}$
de $K_{\alpha}$, avec $\frak{c}\left(B\right)\leq c_{\alpha}$, il
existe un bloc feuilleté contrôlé dont la composante $B_{\ell}$ du
bord au\textendash{}dessus de $\mathcal{C}_{\alpha,\ell}$ satisfait
les propriétés supplémentaires~:

\begin{lyxlist}{00.00.0000}
\item [{(BC3')}] $B_{\ell}$ est $1$\textendash{}connexe dans $B$,
\item [{(BC4')}] $\frak{c}\left(B_{\ell}\right)\leq\frak{d}_{\alpha}\left(\frak{c}\left(B\right)\right)$.
\end{lyxlist}
\end{enumerate}
\end{defn}
Le Théorème~3.2.1 de~\citep[p871]{MarMat} prouve que chaque $K_{\alpha}$
non initial est contrôlable lorsque $\hat{\fol{}}$ ne contient ni
nœud\textendash{}col ni selle quasi\textendash{}résonnante. Le reste
de la Section~3.2 de~\citep{MarMat} prouve que ce résultat entraîne
le théorème de Mar\'{i}n\textendash{}Mattei. Nous montrons dans la
Section~\ref{sec:presentable} que ce théorème se généralise à tous
les feuilletages fortement présentables.

\section{\label{sec:Composantes-initiales}Composantes initiales}

Nous renvoyons à la Section~\ref{sub:Comp-init} pour les définitions
de «branche morte» et «composante initiale». Cette section est dévolue
à la preuve de la propriété suivante~:
\begin{thm}
\label{thm:composante_init}Soit $\fol{}$ un germe de feuilletage
fortement présentable réduit par un morphisme minimal $E\,:\,\mathcal{M}\to\left(\ww C^{2},0\right)$.
Alors $E^{-1}\left(0\right)$ contient au maximum une composante initiale.
Lorsqu'elle existe, exactement deux branches mortes s'y attachent.
Les singularités présentes aux coins de ces branches mortes (autres
que le point d'attache), sont des selles rationnelles linéarisables.
L'extrémité d'une de ces branches est le premier diviseur créé par
la réduction.
\end{thm}
On verra en Section~\ref{sec:Exemples} comment construire des feuilletages
comportant un nombre arbitraire de composantes initiales, chacune
pouvant elle\textendash{}même être attachée à un nombre arbitraire
de branches mortes.

\subsection{Préliminaires }

\subsubsection{\label{sub:Camacho-Sad}Formule de Camacho\textendash{}Sad}

La notion d'indice de Camacho\textendash{}Sad, ainsi que les formules
afférentes, sont introduites dans le fameux article~\citep{CamaSad}.
Sans entrer dans les détails de la construction d'origine, on se servira
du fait suivant comme définition.
\begin{lem}
Soit $\fol{}$ un germe de singularité réduite en un point $p\in\ww C^{2}$,
et $S$ un germe de séparatrice locale, lisse en ce point. Quitte
à changer le système de coordonnées conformes au voisinage de la singularité,
on peut toujours supposer que $p=\left(0,0\right)$ et que le feuilletage
$\fol{}$ est induit par un germe de $1$\textendash{}forme différentielle
\begin{eqnarray*}
\left(\lambda_{1}x+\cdots\right)\dd y-\lambda_{2}y\left(1+\cdots\right)\dd x & \,\,\,\,\,,\,\lambda_{2}\neq0 & \,.
\end{eqnarray*}
La valeur de l'indice $\csad$ de Camacho\textendash{}Sad de $\fol{}$
relativement à $S$ au point $p$ est donnée par les égalités suivantes.
\begin{enumerate}
\item Si $S\subset\left\{ x=0\right\} $ on a 
\begin{eqnarray*}
\csad & = & \frac{\lambda_{1}}{\lambda_{2}}\,.
\end{eqnarray*}

\item De plus, si $\fol{}$ est un nœud\textendash{}col convergent d'invariant
formel $\mu\in\ww C$ et si $S\subset\left\{ y=0\right\} $ est la
séparatrice faible, on a
\begin{eqnarray*}
\csad & = & \mu\,.
\end{eqnarray*}

\end{enumerate}
\end{lem}
\begin{thm}
\emph{(Formule de Camacho\textendash{}Sad)} \emph{\citep{CamaSad}}
Soit $\mathcal{C}$ une courbe complexe lisse et compacte plongée
dans une surface complexe $\mathcal{M}$. On note $\frak{c}\left(\mathcal{C},\mathcal{M}\right)$
la classe de Chern de ce plongement. Soit $\fol{}$ un feuilletage
holomorphe sur $\mathcal{M}$ laissant $\mathcal{C}$ invariante,
et dénotons ${\tt Sing}\left(\fol{},\mathcal{C}\right)$ l'ensemble
des points singuliers de $\fol{}$ sur $\mathcal{C}$. Alors
\begin{eqnarray*}
\frak{c}\left(\mathcal{C},\mathcal{M}\right) & = & \sum_{p\in{\tt Sing}\left(\fol{},\mathcal{C}\right)}\csad[\fol{}][\mathcal{C}][p]\,.
\end{eqnarray*}

\end{thm}

\subsubsection{Autres ingrédients}

Rappelons le résultat suivant, qui sera utile dans cette section~:
\begin{lem}
\label{lem:CamaLinsNetoSad}\emph{\citep[Lemma 1, p20]{CaLiNetSad}}
Soit $s$ une singularité de germe de feuilletage holomorphe du plan
complexe, telle que:
\begin{enumerate}
\item sa réduction ne contient pas de nœud\textendash{}col,
\item exactement deux séparatrices s'y croisent, et elles sont transverses,
irréductibles et lisses. 
\end{enumerate}
Alors $s$ est réduite.
\end{lem}
Ce lemme va nous permettre de contrôler de manière très précise le
processus de formation des branches mortes dans la réduction d'un
feuilletage fortement présentable. Un autre ingrédient, élémentaire,
pour arriver à cette fin est le~:
\begin{lem}
\label{lem:BM_pas_sep}Soit $p$ une singularité non réduite présente
sur une composante $D$  à une étape intermédiaire de la réduction
d'un germe de feuilletage. Si la réduction de $p$ produit une branche
morte attachée à $D$ alors la seule séparatrice passant par $p$
est le diviseur $D$.
\begin{proof}
Soit $S$ un germe de séparatrice passant par $p$. Après réduction
de $s$ le transformé strict de $S$ ne peut que passer par le point
d'attache. Donc $S\subset D$.
\end{proof}
\end{lem}
Le but de ce paragraphe est de montrer le résultat suivant~:
\begin{cor}
\label{cor:pas_attach_BM}Soit $D$ une composante à une étape de
la réduction d'un feuilletage fortement présentable. Alors la réduction
d'une singularité $p\in D$ ne peut créer de branche morte attachée
à $D$ en $p$.
\end{cor}
Cette propriété se déduit directement du Lemme~\ref{lem:BM_pas_sep}
et du suivant (avec $S:=D$)~:
\begin{lem}
\label{lem:sing_pas_morte}Soit $s$ une singularité de germe de feuilletage
fortement présentable du plan complexe par laquelle ne passe qu'une
seule séparatrice $S$, lisse et irréductible.
\begin{enumerate}
\item Si le feuilletage est réduit, $s$ est une singularité de type nœud\textendash{}col
divergent dont $S$ est la séparatrice forte.
\item Si le feuilletage n'est pas réduit et si son arbre de réduction est
une branche morte attachée au transformé strict $\hat{S}$ de $S$,
alors $\hat{S}$ est la séparatrice forte d'un nœud\textendash{}col
présent dans le dernier diviseur créé, qui est le point d'attache
de la branche morte. Les autres singularités de $\hat{\fol{}}$ sont
des selles rationnelles linéarisables.
\end{enumerate}
\end{lem}
\begin{figure}[H]
\includegraphics[height=5cm]{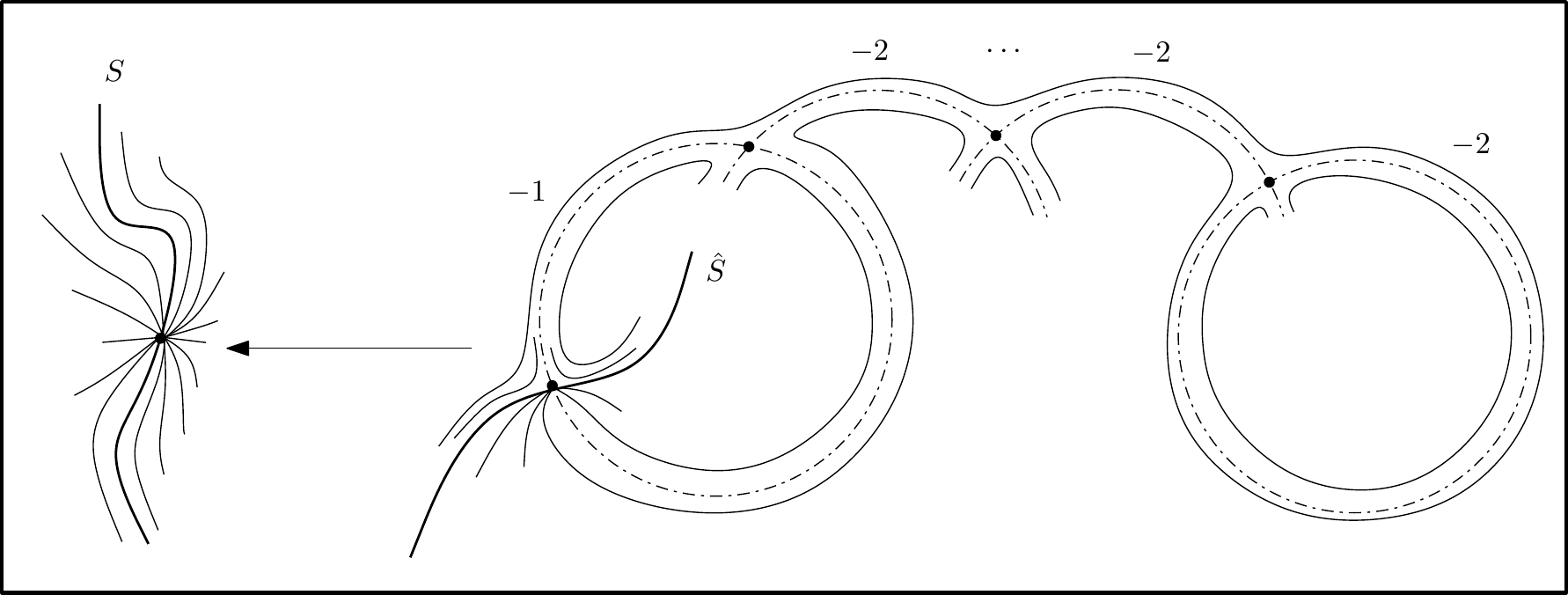}

\caption{Une branche morte obtenue en réduisant une seule séparatrice (en gras).
Les nombres indiquent la classe de Chern des composantes.}
\end{figure}

\begin{rem}
En Section~\ref{sub:branches_mortes} nous construisons des feuilletages
correspondant à la situation évoquée dans ce lemme. On comprend également
comment construire des feuilletages (pas fortement présentables) ayant
un nombre arbitraire de branches mortes.\end{rem}
\begin{proof}
~
\begin{enumerate}
\item Par une singularité réduite passent deux séparatrices transverses,
irréductibles et lisses, sauf si la singularité est un nœud\textendash{}col
divergent.
\item Nous montrons le lemme par récurrence sur le nombre $d\geq1$ d'éclatements
ponctuels de $s$ nécessaire à la réduire. Si $s$ est réduite en
un éclatement on obtient un diviseur $D^{1}$ ne portant qu'une seule
singularité $\hat{s}$ par laquelle passe $\hat{S}$. L'holonomie
de $\hat{\fol{}}$ le long de $D^{1}$ doit être l'identité et son
indice de Camacho\textendash{}Sad égal à $-1$ d'après la formule
de Camacho\textendash{}Sad. Les seuls types locaux de singularité
ayant cette propriété sont les suivants~:

\begin{figure}[H]
\hfill{}\subfloat[Les selles linéarisables localement conjuguées à $\dd{\left(xy\right)=0}$,
avec le diviseur correspondant à une branche de $\left\{ xy=0\right\} $~;
dans ce cas $\fol{}$ est régulier, ce qui est exclu.]{\includegraphics[height=6cm]{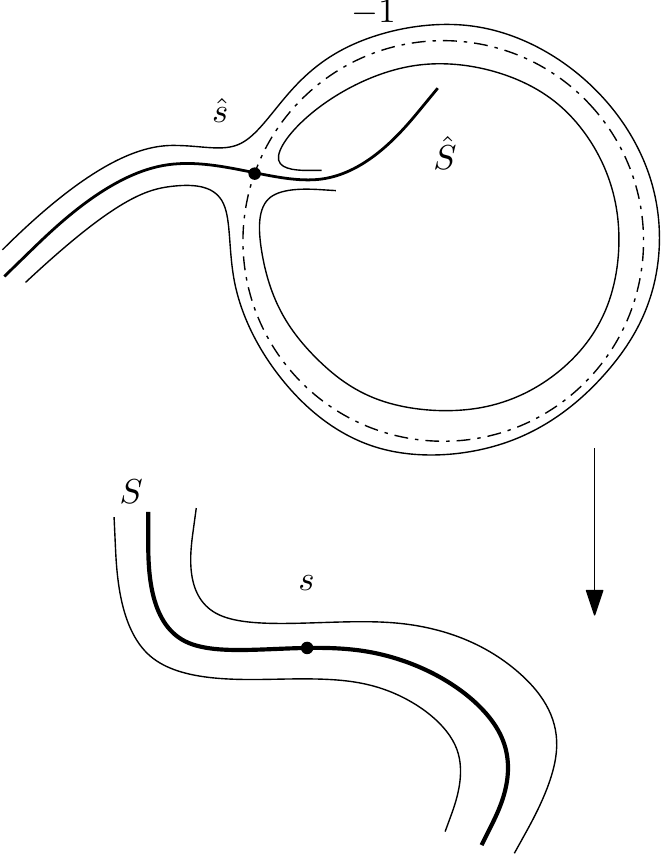}

}\hfill{}\subfloat[Les nœuds\textendash{}cols dont le diviseur est la séparatrice faible,
d'invariant formel $\mu:=-1$. Ainsi $\hat{S}$ est la séparatrice
forte de $\hat{s}$.]{\includegraphics[height=6cm]{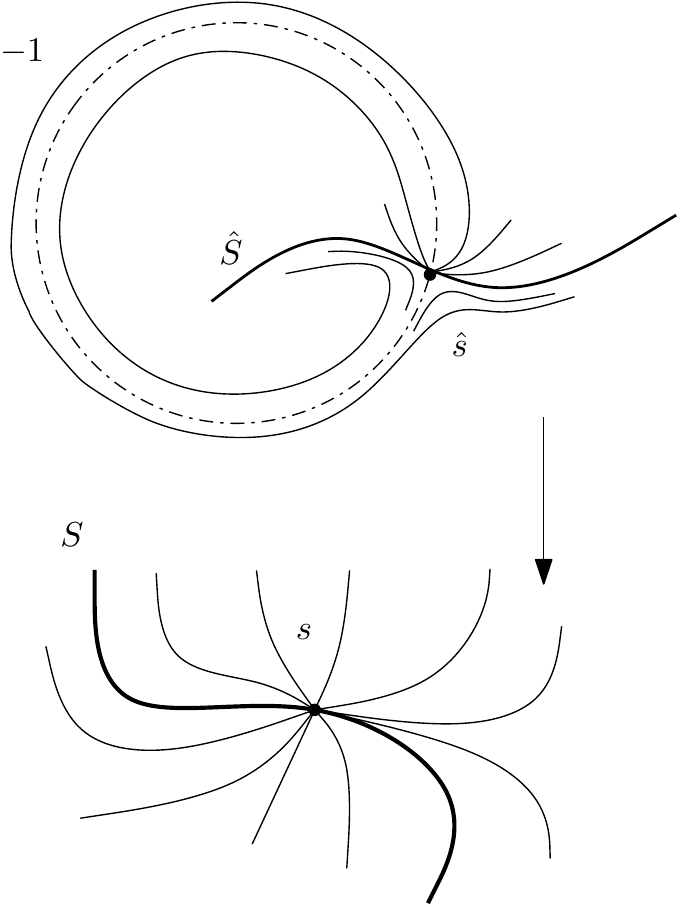}

}\hfill{}
\end{figure}

Le lemme est vrai pour $d\leq1$. Supposons qu'il le soit pour toute
singularité réduite en au plus $d\geq1$ éclatements, et prenons une
singularité $s$ réduite en (exactement) $d+1$ éclatements, dont
l'arbre de réduction est une branche morte attachée à $\hat{S}$.
Après le premier éclatement de $s$, le diviseur $D^{1}$ (puisque
c'est un maillon d'une branche morte) contient au plus deux singularités,
réduites en au plus $d$ éclatements. Le transformé strict $S_{1}$
de $S$ passe par une d'elles, disons $s_{1}$, en coupant transversalement
le diviseur $D^{1}$. Nécessairement la réduction de $s_{1}$ produit
le point d'attache, donc la réduction de l'autre singularité $\tilde{s}$
éventuelle est une branche morte. Comme la seule séparatrice de $\tilde{s}$
est $D^{1}$, l'hypothèse de récurrence appliquée à $\tilde{s}$ contredit
celle de forte présentabilité, donc $s_{1}$ est la seule singularité
présente sur $D^{1}$.

Éclatons $s_{1}$ une fois de plus~: on obtient un nouveau diviseur
$D^{2}$ attaché à $D^{1}$. Le même argument que précédemment indique
que les seules singularités présentes sur $D^{1}\cup D^{2}$ sont~:
\begin{itemize}
\item le coin $p_{1}:=D^{2}\cap D^{1}$,
\item le point d'intersection $s_{2}$ entre le transformé strict $S_{2}$
de $S_{1}$ et $D^{2}$.
\end{itemize}

Comme $p_{1}\neq s_{2}$ le coin est réduit~:
\begin{lem}
\label{lem:croisement_reduit}Soit $p$ un coin du diviseur exceptionnel
d'une étape intermédiaire de la réduction d'un feuilletage fortement
présentable, par laquelle ne passe aucune autre séparatrice que les
branches du diviseur exceptionnel. Alors la singularité $p$ est réduite.\end{lem}
\begin{proof}
D'après le Lemme~\ref{lem:CamaLinsNetoSad}, si $p$ n'était pas
réduite il y aurait un nœud\textendash{}col dans sa réduction. Puisque
aucune autre séparatrice ne passe par $p$, la séparatrice forte de
ce nœud\textendash{}col est incluse dans un diviseur. Cette situation
est exclue pour les feuilletages fortement présentables. 
\end{proof}

D'après le raisonnement développé pour le cas $d=1$, la singularité
réduite présente en $p_{1}$ est une selle rationnelle linéarisable
puisqu'elle ne peut être un nœud\textendash{}col (les deux séparatrices
passant par $p_{1}$ sont des composantes du diviseur).

La prochaine étape pour réduire $s$ est d'éclater $s_{2}$. On montre
ainsi de proche en proche qu'après $d+1$ étapes le transformé strict
$\hat{S}$ de $S$ intersecte en $\hat{s}$ la dernière composante
$D^{d+1}$ créée, et les autres singularités de $\hat{\fol{}}$ correspondent
aux points de croisement $p_{j}:=D^{j}\cap D^{j+1}$, en lesquelles
$\hat{\fol{}}$ est réduit et admet une selle rationnelle linéarisable.
La classe de Chern de chaque $D^{j}$ est $-2$ pour $j\leq d$ alors
que pour $D^{d+1}$ elle vaut $-1$. La singularité $p_{j}$ est une
selle linéarisable avec 
\begin{eqnarray*}
\csad[\hat{\fol{}}][D^{j+1}][p_{j}] & = & -\frac{j}{j+1}\,.
\end{eqnarray*}
Puisque $s_{1}$ n'est pas réduite le Lemme~\ref{lem:CamaLinsNetoSad}
indique que $\hat{s}$ est un nœud\textendash{}col. Son indice 
\begin{eqnarray*}
\csad[\hat{\fol{}}][D^{d+1}][\hat{s}] & = & -\frac{1}{d+1}
\end{eqnarray*}
est non nul donc $D^{d+1}$ est sa séparatrice faible. 

\end{enumerate}
\end{proof}

\subsection{Preuve du Théorème~\ref{thm:composante_init}}

Écrivons le morphisme $E$ comme le composé $E_{d}\circ\cdots\circ E_{2}\circ E_{1}$
d'éclatements ponctuels $E_{n}\,:\,\mathcal{M}_{n+1}\to\mathcal{M}_{n}$
avec $\mathcal{M}_{1}:=\left(\ww C^{2},0\right)$ et $\mathcal{M}_{d}:=\mathcal{M}$,
chacun faisant apparaître une nouvelle composante $\mathcal{D}_{n}$
du diviseur exceptionnel final $E^{-1}\left(0\right)$. On dit que
$n$ est l'\textbf{étape d'apparition} de $\mathcal{D}_{n}$, et on
considère l'application définie sur l'ensemble des composantes de
$E^{-1}\left(0\right)$~:
\begin{eqnarray*}
\tx{comp}\left(E^{-1}\left(0\right)\right) & \longrightarrow & \ww N\\
\mathcal{D}_{n} & \longmapsto & \frak{n}\left(\mathcal{D}_{n}\right):=n\,.
\end{eqnarray*}
On écrira $\fol n:=E_{n}^{*}\fol{}$ le transformé de $\fol{}$ par
$E_{n}$.

La preuve se déroule en deux épisodes, correspondants aux différents
temps de formation de la composante initiale. Une conséquence du Corollaire~\ref{cor:pas_attach_BM}
est qu'une composante initiale apparaît toujours après que toutes
les branches mortes s'y attachant ont été créées par le processus
de réduction.

\subsubsection{Apparition de la première branche morte}

De toutes les branches mortes $\left(B_{\ell}\right)_{1\leq\ell\leq m}$
attachées à la composante initiale $\mathcal{C}$, $m\geq2$, on choisit
celle qui apparaît en premier et on la nomme $B$. On procède de la
manière suivante~: en identifiant une branche $B_{\ell}$ à la collection
des maillons qui la forme, on choisit $B$ de sorte que
\begin{eqnarray*}
\min\frak{n}\left(B\right) & = & \min\frak{n}\left(\bigcup_{\ell\leq m}B_{\ell}\right)\,.
\end{eqnarray*}
Numérotons les composantes $\left(D^{j}\right)_{1\leq j\leq k}$ formant
$B$ de sorte que $j\mapsto n_{j}:=\frak{n}\left(D^{j}\right)$ soit
croissante. Si $D^{1}$ est le premier diviseur obtenu dans la réduction
de $\fol{}$, c'est\textendash{}à\textendash{}dire $n_{1}=1$, on
définit $S$ comme étant une séparatrice de $\fol{}$ (dont l'existence
est garantie par le théorème de Camacho\textendash{}Sad~\citep{CamaSad}).
Dans le cas contraire on prend pour $S$ une composante de $E_{n_{1}}^{-1}\left(0\right)$
adjacente à $D^{1}$.
\begin{lem}
\label{lem:naissance_branche}L'ordre d'adjacence des $D^{j}$ coïncide
avec leur ordre d'apparition~:
\begin{eqnarray*}
D^{j}\cap D^{\ell}\neq\emptyset & \Longleftrightarrow & \left|j-\ell\right|\leq1\,.
\end{eqnarray*}
Le coin $D^{j+1}\cap D^{j}$ est une selle rationnelle linéarisable.
De plus $D^{k}$ contient exactement deux singularités (une seule
si $k=1$)~: le coin $D^{k}\cap D^{k-1}$ (si $k>1$) et une singularité
$s_{k}$ par laquelle passe une séparatrice $S_{k}$ différente de
$D^{k}$. Cette séparatrice provient du transformé strict de $S$
dans $\fol{n_{k}}$.\end{lem}
\begin{proof}
$D^{1}$ contient au moins une singularité $s_{1}$ de $\fol{n_{1}}$
par laquelle passe le transformé strict $S_{1}$ de $S$ (si $n_{1}\neq1$
on a directement $S_{1}=S$). Le Corollaire~\ref{cor:pas_attach_BM}
montre que $\fol{n_{1}}$ n'a pas d'autre singularité sur $D^{1}$.
Si $k\geq2$, à l'étape $n_{2}$ aucune autre séparatrice que $D^{1}$
et $D^{2}$ ne passe par le coin $p:=D^{2}\cap D^{1}$, puisque sinon
son transformé strict dans $\hat{\fol{}}$ appartiendrait à une union
connexe de composantes reliant $D^{1}$ et $D^{2}$, contredisant
le fait que $B$ est une branche morte. Par suite la singularité $p$
est réduite (Lemme~\ref{lem:croisement_reduit}) et $D^{2}$ contient
exactement une singularité $s_{2}$ de $\fol{n_{2}}$, par laquelle
passe le transformé strict de $S_{1}$, distinct de $D^{2}$. L'holonomie
de $\fol{n_{1}}$ en $p$ le long de $D^{1}$ est l'identité et l'indice
de Camacho\textendash{}Sad $\csad[\fol{n_{1}}][D^{1}][p]$ est un
entier négatif. Puisque $p$ ne peut être un nœud\textendash{}col
la seule possibilité est que $p$ soit une selle rationnelle linéarisable,
comme dans la preuve du Corollaire~\ref{lem:sing_pas_morte}. Un
raisonnement par induction achève la démonstration du lemme.
\end{proof}

\subsubsection{Apparition de la seconde branche morte}

Soit $n_{k+1}$ l'étape de la réduction de $\fol{}$ qui correspond
à l'éclatement de $s_{k}$. On construit de proche en proche une chaîne
de composantes $\left(\hat{D}^{j}\right)_{1\leq j\leq\hat{k}}$ dont
l'ordre d'apparition $n_{k+j}$ correspond à l'ordre d'adjacence,
en utilisant les règles suivantes.
\begin{itemize}
\item \textbf{Par le coin $\hat{p}_{j}:=D^{k}\cap\hat{D}^{j}$ passe une
séparatrice $\hat{S}_{j}$ distincte de $D^{k}$ et $\hat{D}^{j}$.}\\
Alors $\hat{p}_{j}$ n'est pas réduite et $\hat{S}_{j}$ n'est pas
une composante du diviseur exceptionnel $E_{n_{k+j}}^{-1}\left(0\right)$
(Remarque~\ref{rem_reduction=00003Darbre}). À l'étape $n_{k+j+1}$
l'éclatement de $\hat{p}_{j}$ produit un diviseur $\hat{D}^{j+1}$
adjacent à $\hat{D}^{j}$ et $D^{k}$. Si une séparatrice distincte
de $\hat{D}^{j}\cup\hat{D}^{j+1}$ passait par le coin $\hat{D}^{j}\cap\hat{D}^{j+1}$
alors $B$ ne serait pas maximale~; cette singularité est donc réduite
(Lemme~\ref{lem:croisement_reduit}). C'est en fait une selle rationnelle
linéarisable d'après le même argument qu'employé pour prouver le Lemme~\ref{lem:naissance_branche}.
Si $\hat{D}^{j+1}$ porte une troisième singularité alors $\hat{D}^{j+1}=\mathcal{C}$
et le processus s'arrête.

\begin{figure}[H]
\hfill{}\includegraphics[height=5cm]{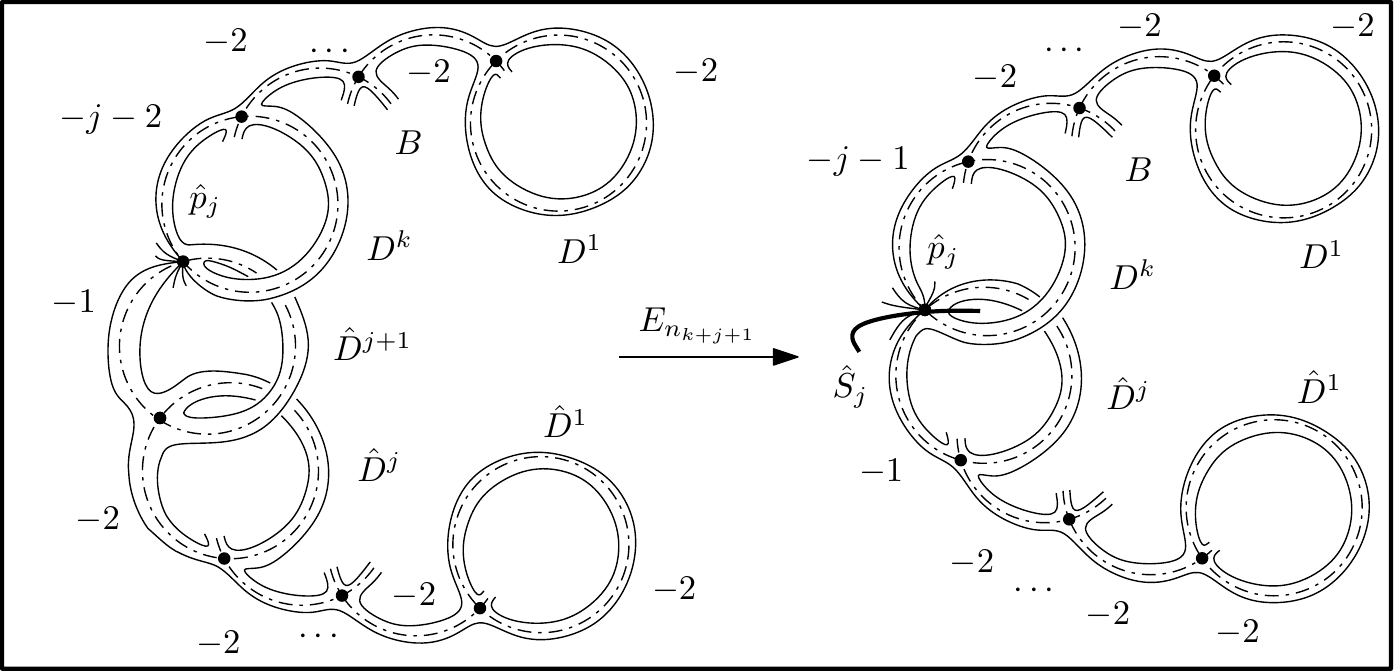}\hfill{}
\end{figure}

\item \textbf{Par $\hat{p}_{1}$ ne passe aucune autre séparatrice que $D^{k}$
et $\hat{D}^{1}$.}\\
D'après le Lemme~\ref{lem:croisement_reduit} $\hat{p}_{1}$ est
réduite. Puisque $B$ est maximale, $\hat{D}^{1}$ ne peut être un
maillon de $B$ et $\hat{p}_{1}$ est le point d'attache de $B$ sur
la composante initiale $\mathcal{C}=\hat{D}^{1}$. Le Corollaire~\ref{cor:pas_attach_BM}
interdit l'existence d'une deuxième branche morte attachée à $\mathcal{C}$.
Cette configuration ne peut donc se produire.

\begin{figure}[H]
\hfill{}\includegraphics[height=5cm]{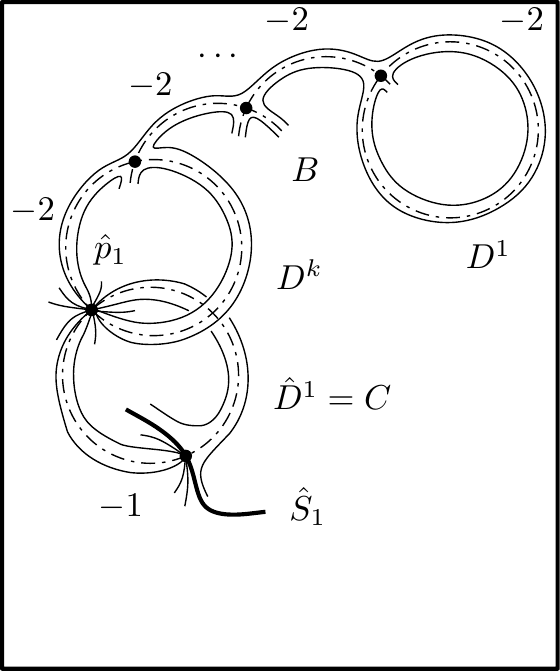}\hfill{}
\end{figure}

\item \textbf{Par $\hat{p}_{j}$, $j\geq2$, ne passe aucune autre séparatrice
que $D^{k}$ et $\hat{D}^{j}$.}\\
D'après le Lemme~\ref{lem:croisement_reduit}, $\hat{p}_{j}$ est
réduite. Dans ce cas $\hat{D}^{j}=\mathcal{C}$ et le Corollaire~\ref{cor:pas_attach_BM}
oblige la deuxième branche morte à contenir $\bigcup_{\ell=1}^{j-1}\hat{D}^{j}$.
Dans ce cas les seules séparatrices passant par le coin $\hat{D}^{j}\cap\hat{D}^{j-1}$
sont $\hat{D}^{j}$ et $\hat{D}^{j-1}$~: le coin est réduit. Par
une troisième singularité de $\mathcal{C}$ passent tous les transformés
stricts des séparatrices de $\fol{}$ et de la séparatrice $S$ considérée
pour construire la première branche morte. Le Corollaire~\ref{cor:pas_attach_BM}
interdit à $\mathcal{C}$ de porter une autre singularité de $\fol{n_{k+\hat{k}+1}}$
(et donc il ne peut y avoir plus de $m:=2$ branches mortes attachée
à $\mathcal{C}$). La construction s'achève.

\begin{figure}[H]
\hfill{}s\includegraphics[height=5cm]{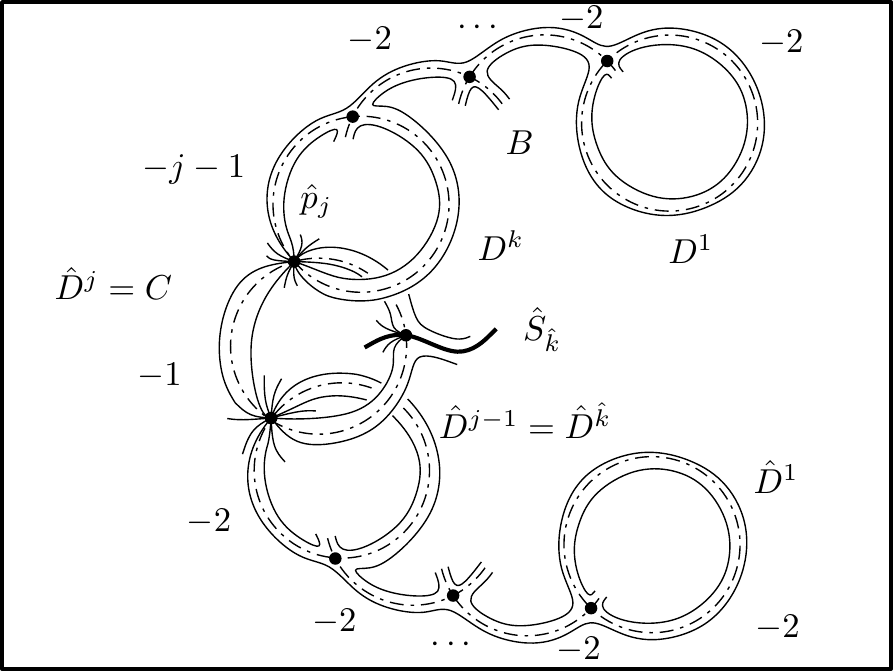}\hfill{}
\end{figure}

\end{itemize}
Ce procédé s'arrête après un nombre fini $j=\hat{k}+1\geq2$ d'étapes.
En prenant pour $\hat{S}_{\hat{k}}$ le transformé strict de $S$
dans $E_{n_{k+\hat{k}}}$ on prouve que $S$ ne pouvait être une composante
du diviseur $E_{n_{1}}^{-1}\left(0\right)$ (sinon $\hat{S}_{\hat{k}-1}$
serait une composante du diviseur en coupant deux autres, contredisant
la Remarque~\ref{rem_reduction=00003Darbre}) et
\begin{eqnarray*}
n_{1} & = & 1\,.
\end{eqnarray*}
Aussi l'extrémité de $B$ est le premier diviseur créé par la réduction
de $\fol{}$, entraînant l'unicité de $\mathcal{C}$ et le résultat.

\section{\label{sec:presentable}Feuilletages fortement présentables}

On montre maintenant le Théorème~C. Le premier ingrédient est le
suivant~:
\begin{thm}
\label{thm:bloc_adapt}Considérons un germe de feuilletage holomorphe
au voisinage d'une singularité élémentaire admettant deux séparatrices%
\footnote{Cette classe regroupe les singularités non dégénérées et les noeuds\textendash{}cols
convergents.%
}, induit par la $1$\textendash{}forme $\omega_{R}$ donnée en Section~\ref{sec:Notations},
avec $\lambda\notin\ww R_{>0}$. Il existe $\rho,\, r>0$ tels que
les propriétés suivantes sont satisfaites sur le polydisque $\mathcal{U}\left(\rho,r\right):=\rho\ww D\times r\ww D$.
\begin{enumerate}
\item Le feuilletage est holomorphe avec une seule singularité sur un voisinage
de $\adh{\mathcal{U}\left(\rho,r\right)}$.
\item Pour $x_{*}\in\rho\ww S^{1}$ considérons la paramétrisation $\gamma\,:\, t\in\left[0,1\right]\mapsto x_{*}\exp\left(2\ii\pi t\right)$
du cercle $\rho\ww S^{1}$, ainsi que l'ouvert maximal sur lequel
$\holo{\gamma}$ est définie~:
\begin{eqnarray*}
\Sigma & := & \left\{ \left(x_{*},y_{*}\right)\,:\,\left|y_{*}\right|<r\,,\,\gamma\mbox{ se relève dans }\mbox{\ensuremath{\fol{}}}|_{\mathcal{U}\left(\rho,r\right)}\mbox{ à partir de }\left(x_{*},y_{*}\right)\right\} \,.
\end{eqnarray*}
Alors $\Sigma$ est un disque analytique contenant $\left(x_{*},0\right)$,
à bord analytique par morceaux. De plus à $\rho$ fixé il existe $M_{\rho}>0$
telle que pour tout $r$ assez petit
\begin{eqnarray*}
\mathbf{e}\left(\partial\adh{\Sigma}\right) & < & M_{\rho}r\,.
\end{eqnarray*}

\item $\left(\mathcal{U}\left(\rho,r\right)\backslash\left\{ xy=0\right\} \right)\cup\tx{Susp}_{\gamma}\left(\Sigma\right)$
est un bloc feuilleté contrôlé.
\item Si cette classe de singularité apparaît dans une composante $K_{\alpha}$
de la réduction $\hat{\fol{}}$ alors $K_{\alpha}$ est contrôlable.
\end{enumerate}
\end{thm}
Ce théorème, démontré en Section~\ref{sub:blocs_adapt_cont}, permet
de placer des blocs contenant les singularités nœuds\textendash{}cols
et quasi\textendash{}résonnantes dans la construction de Mar\'in\textendash{}Mattei.
Il faut toutefois obtenir des résultats plus précis pour le passage
des coins. Pour un feuilletage fortement présentable cela ne peut
se produire qu'avec des singularités non dégénérées~; nous préciserons
cela en Section~\ref{sub:coins}. Dès lors un feuilletage présentable
 est incompressible dans l'ouvert 
\begin{eqnarray*}
\mathcal{U} & := & E\left(\bigcup_{\alpha\in\mathcal{A}}\mathcal{B}_{\alpha}\right)\cup\mathcal{S}
\end{eqnarray*}
où~:
\begin{itemize}
\item les blocs $\mathcal{B}_{\alpha}$ correspondant à des singularités
$s$ autres que des nœuds ($\lambda>0$) sont ceux donnés par~(3)
de ce théorème (\emph{modulo }le passage des coins), on prendra alors
$\rho_{s}$ et $r_{s}$ inférieurs à $\rho$ et $r$,
\item les autres blocs $\mathcal{B}_{\alpha}$ sont pris comme dans la construction
d'origine de Mar\'in\textendash{}Mattei.
\end{itemize}
\bigskip{}

Il est tout à fait probable que certains feuilletages non fortement
présentables soient incompressibles (par exemple s'il n'y a qu'un
seul nœud\textendash{}col dans la réduction). Dans ces cas\textendash{}là
c'est la localisation de la construction qui est mise en défaut~:
\begin{prop}
\label{prop:bord_NC_pas_1_connexe}Soit $\fol{}$ un germe de feuilletage
holomorphe au voisinage d'une singularité de type nœud\textendash{}col,
induit par la $1$\textendash{}forme $\omega_{R}$ donnée en Section~\ref{sec:Notations}.
On note $\mathcal{S}$ la séparatrice faible d'un nœud\textendash{}col
convergent, et $\mathcal{S}:=\emptyset$ dans le cas divergent. Soit
$0\in D\subset r\ww D$ un disque fermé conforme et $\mathcal{V}$
un voisinage simplement connexe de $\left(0,0\right)$, assez petit,
fibré au\textendash{}dessus de $D$ par la projection $\left(x,y\right)\mapsto y$.
Prenons une composante $V$ de $\partial\mathcal{V}$, de type suspension
au\textendash{}dessus de $\left\{ 0\right\} \times\partial D$. Alors
$V$ ne peut\textendash{}être $1$\textendash{}connexe dans $\mathcal{V}\backslash\mathcal{S}$.
\end{prop}
Cette proposition, prouvée en Section~\ref{sub:chemins_inamov},
interdit de plonger un nœud\textendash{}col dont la séparatrice forte
est une composante du diviseur exceptionnel, dans un bloc local suffisamment
petit vérifiant~(BF4). La condition de forte présentabilité seule
ne permet donc pas de caractériser les feuilletages incompressibles.

\bigskip{}

Le Théorème~6.1.1 de~\citep[p900]{MarMat} nous permet d'obtenir
une transversale $\mathcal{C}'$ complètement connexe dans $\bigcup_{\alpha\notin{\tt NC}}\mathcal{B}_{\alpha}$,
où ${\tt NC}$ est le sous\textendash{}ensemble de $\mathcal{A}$
correspondant aux blocs contenant les nœuds\textendash{}cols (y compris
en présence de selles quasi\textendash{}résonnantes). Soit $s$ une
singularité nœud\textendash{}col de $\hat{\fol{}}$, appartenant à
une composante $D$ du diviseur exceptionnel. On note $\fol{}'$ le
feuilletage poussé\textendash{}en\textendash{}avant de $\hat{\fol{}}$
par le biholomorphisme $\psi_{s,D}$ introduit en Section~\ref{sub:decoup}.
Il est aisé de voir que chaque transversale $\left\{ y=y_{0}\right\} $
sature, par $\fol{}'$, un voisinage épointé de $D=\left\{ y=0\right\} $.
Rajouter $\psi_{s,D}^{-1}\left(\left\{ y=y_{0}\right\} \right)$ à
$\mathcal{C}'$ est pourtant une approche vouée à l'échec~:
\begin{cor}
\label{cor:transv_pas_1_connexe}Sous les hypothèses de la proposition
précédente aucune transversale $\left\{ y=y_{0}\right\} $ n'est $1$\textendash{}connexe
dans $\mathcal{V}\backslash\mathcal{S}$.
\end{cor}
Ce corollaire sera démontré en Section~\ref{sub:chemins_inamov}. 

\bigskip{}

Pour obtenir une transversale complète $\mathcal{C}$ dans $\mathcal{U}$
on adjoint à $\mathcal{C}'$, pour chaque singularité nœud\textendash{}col
$s\in D$ où $D$ est la séparatrice faible de $\hat{\fol{}}$, le
disque conforme $\Omega\left(x_{0},y_{0},\varepsilon\right)$ (ou
plutôt, son image par le morphisme de réduction), dépendant de $\left(x_{0},y_{0}\right)\in\partial\adh{\mathcal{U}\left(\rho,r\right)}$
et $r\geq\varepsilon>0$, défini dans la coordonnée $\psi_{s,D}$
par~: 
\begin{eqnarray*}
\Omega\left(x_{0},y_{0},\varepsilon\right) & := & \left\{ \left(x,y_{0}\right)\,:\,0<\left|x\right|<\rho\,,\,\sin\left(k\arg x\right)>\nf 12\right\} \cup\left\{ x_{0}\right\} \times\varepsilon\ww D\,\,\,.
\end{eqnarray*}

\begin{cor}
\label{cor:TCC}Il existe $\varepsilon>0$ tel que, pour tout $\left(x_{0},y_{0}\right)\in\partial\adh{\mathcal{U}\left(\rho,r\right)}$,
l'ouvert $\Omega_{\left(x_{0},y_{0},\varepsilon\right)}$ 
\begin{enumerate}
\item sature un voisinage de $\left(0,0\right)$ épointé de $\left\{ x=0\right\} $,
\item est $1$\textendash{}connexe dans $U$$\left(\rho,r\right)$.
\end{enumerate}
\end{cor}
Ce corollaire, montré en Section~\ref{sub:chemins_inamov}, couplé
au fait que la $1$\textendash{}connexité est transitive~\citep[Remarque 1.2.3, p861]{MarMat},
garantit alors que $\mathcal{C}$ est une transversale complètement
connexe. Cela clôt la preuve de l'assertion~(1) du Théorème~C.

\subsection{\label{sub:blocs_adapt_cont}Blocs feuilletés adaptés contrôlés~:
preuve du Théorème~\ref{thm:bloc_adapt}}

Dans cette section nous ne considérons que des feuilletages ayant
une singularité de type
\begin{itemize}
\item non dégénérée donnée par la forme $\omega_{R}$ définie en~\eqref{eq:prepa_non-degenere},
avec $\lambda\notin\ww R_{>0}$,
\item nœud\textendash{}col convergent donné par la forme $\omega_{yR}$
définie en~\eqref{eq:prepa_degenere}.
\end{itemize}
D'après les choix effectués en début de section, la famille de polydisques
\begin{eqnarray*}
\mathcal{U}\left(\rho,r\right) & := & \rho\ww D\times r\ww D
\end{eqnarray*}
fournit une base de voisinages $\mathcal{U}\left(\rho,r\right)$ de
$\left(0,0\right)$, dépendant de deux paramètres $\rho,r>0$ suffisamment
petits, pour laquelle les conditions~(BF3) et~(BC1), ainsi que~(1)
du Théorème~\ref{thm:bloc_adapt}, sont satisfaites. De plus les
conditions (BF1), (BF2), (BC3) et (BC4) sont trivialement vérifiées
dés que~(2) du Théorème~\ref{thm:bloc_adapt} est établi. Les propriétés
suivantes doivent donc être satisfaites pour démontrer~(3)~:
\begin{itemize}
\item (BF4),
\end{itemize}
et pour démontrer~(4)~:
\begin{itemize}
\item (BC3'),
\item (BC4').
\end{itemize}

\subsubsection{Preuve de~(2)}

Soit $\zeta\left(t\right):=z_{*}+2\ii\pi t$ pour $x_{*}=\exp z_{*}$
et $t\in\left[0,1\right]$. Considérons le feuilletage analytique
réel $\fol C$ induit sur le cylindre $C:=\left[0,1\right]\times r\ww D$
par la $1$\textendash{}forme 
\begin{eqnarray*}
\omega & := & \dd y-yG\left(t,y\right)\dd t\,,
\end{eqnarray*}
où 
\begin{eqnarray*}
G\left(t,y\right) & := & 2\ii\pi\frac{1+\tilde{R}\left(\zeta\left(t\right),y\right)}{\tilde{P}\left(\zeta\left(t\right)\right)}\,,
\end{eqnarray*}
avec $P\left(x\right):=\lambda$ pour les singularités dégénérées,
et $P\left(x\right):=x^{k}$ pour les nœuds\textendash{}cols. Celui\textendash{}ci
est régulier, transverse aux fibres de la projection $\Pi\,:\,\left(t,y\right)\mapsto t$.
On identifie canoniquement $\Sigma$ à un sous\textendash{}ensemble
de $r\ww D\simeq\left\{ 0\right\} \times r\ww D$, c'est\textendash{}a\textendash{}dire
l'ensemble des $y_{*}\in r\ww D$ tels que le relevé $t\mapsto y\left(t,y_{*}\right)$
de $\left[0,1\right]\times\left\{ 0\right\} $ dans $\fol C$ en s'appuyant
sur $\left(0,y_{*}\right)$ existe pour tout $t\in\left[0,1\right]$.
L'holonomie $\holo C$ le long de $\fol C$ est holomorphe sur un
voisinage de $\adh{\Sigma}$. Par transversalité $y_{*}\in\Sigma$
si, et seulement si, $\left|y\left(y_{*},t\right)\right|<r$ pour
tout $t\in\left[0,1\right]$.

Clairement $\Sigma$ est ouvert et contient $0$. Si $\Gamma\,:\,\left[0,1\right]\to\Sigma$
est un lacet lisse et simple, délimitant un disque analytique $U\subset r\ww D$,
alors la surface 
\begin{eqnarray*}
T & := & \bigcup_{s\in\left[0,1\right]}y\left(\left[0,1\right],\Gamma\left(s\right)\right)
\end{eqnarray*}
est un tube homéomorphe à $\left[0,1\right]\times\partial U$. Mais
pour tout $y_{*}\in U$ le chemin $t\mapsto y\left(t,y_{*}\right)$
est prisonnier de $T$ et donc $y_{*}\in\Sigma$, ce qui montre que
$\Sigma$ est simplement connexe.

Pour $y_{*}\in\Sigma$ écrivons 
\begin{eqnarray*}
\varphi\left(t,y_{*}\right) & := & \ln\left|y\left(t,y_{*}\right)\right|
\end{eqnarray*}
 de sorte que
\begin{eqnarray*}
\dot{\varphi}\left(t,y_{*}\right) & = & \re{G\left(t,y\left(t,y_{*}\right)\right)}\,.
\end{eqnarray*}
Dès lors, ou bien $y_{*}\in r\ww S^{1}$, ou bien $y_{*}\in\holo C^{-1}\left(r\ww S^{1}\right)$,
ou bien encore il existe $\tau\in]0,1[$ minimal tel que 
\begin{eqnarray*}
\begin{cases}
\varphi\left(\tau,y_{*}\right) & =r\\
\dot{\varphi}\left(\tau,y_{*}\right) & =0
\end{cases} &  & .
\end{eqnarray*}
Ainsi $y_{*}$ est inclus dans le pré\textendash{}image par transport
holonome de l'ensemble analytique réel $Z:=\partial\adh C\cap\left\{ \re F=0\right\} $.
Écartons les cas triviaux, $Z=\partial\adh C$ ou $\#Z<\infty$, pour
étudier la courbe analytique $Z$. Si $\left(\tau_{0},y\left(\tau_{0},y_{*}\right)\right)$
est un point régulier de $Z$ alors $Z$ est localement paramétré
par un arc lisse $\tau\mapsto\left(\tau,y\left(\tau,y_{*}\left(\tau\right)\right)\right)$
résolvant le système pour $y_{*}:=y_{*}\left(\tau\right)$. Dès lors
$\partial\Sigma$ est inclus dans une union finie d'adhérences de
courbes analytiques par morceaux. En particulier $\Sigma$ n'a qu'un
nombre fini de composantes connexes. Il suffit alors de réduire $r$
pour obtenir la connexité de $\Sigma$.

Nous finissons la preuve de~(2) en donnant une estimation de la rugosité
de $\partial\Sigma$. Tout d'abord il est clair que pour tout $y_{*}\in\Sigma$
et $t\in\left[0,1\right]$ on a la relation 
\begin{eqnarray}
y\left(t,y_{*}\right) & = & y_{*}\exp\left(-F\left(t,y_{*}\right)\right)\label{eq:sol_implicite}
\end{eqnarray}
où
\begin{eqnarray*}
F\left(t,y_{*}\right) & := & \int_{0}^{t}G\left(u,y\left(u,y_{*}\right)\right)\dd u\,.
\end{eqnarray*}
La formule suivante est une conséquence immédiate des définitions~:
\begin{lem}
\label{lem:frml_rugo_holo}Pour un chemin $\chi\,:\,\left[0,1\right]\to\adh{\Sigma}$
et pour tout $s\in\left[0,1\right]$ on a
\begin{eqnarray*}
\mathbf{e}\left(\holo C\circ\chi,s\right) & = & \mathbf{e}\left(\chi,s\right)+\left\{ \left\{ \arg\left(1-\chi\left(s\right)\ppp{y_{*}}F\left(1,\chi\left(s\right)\right)\right)\right\} \right\} \,.
\end{eqnarray*}

\end{lem}
On calcule aisément
\begin{eqnarray*}
\ppp{y_{*}}F\left(1,y_{*}\right) & = & 2\ii\pi\int_{0}^{1}\ppp{y_{*}}y\left(t,y_{*}\right)\ppp y{\tilde{R}}\left(\zeta\left(t\right),y\left(t,y_{*}\right)\right)\frac{\dd t}{P\left(\zeta\left(t\right)\right)}\,.
\end{eqnarray*}
Puisque 
\begin{eqnarray*}
A\left(x,y\right) & := & \frac{\ppp yR\left(x,y\right)}{P\left(x\right)}
\end{eqnarray*}
est holomorphe sur $\mathcal{U}\left(\rho,r\right)$, on pose
\begin{eqnarray*}
D_{\rho} & := & \sup_{C}\left|\ppp{y_{*}}y\right|\\
M & := & \sup_{\mathcal{U}\left(\rho,r\right)}\left|A\right|
\end{eqnarray*}
 de sorte que
\begin{eqnarray*}
\left|\ppp{y_{*}}F\left(1,y_{*}\right)\right| & \leq & 2\pi MD_{\rho}\,.
\end{eqnarray*}
Ainsi pour $\rho_{0}\geq\rho>0$ fixé et $r_{0}\geq r>0$ assez petit
on a 
\begin{eqnarray*}
\left|\left\{ \left\{ \arg\left(1-y_{*}\ppp{y_{*}}F\left(1,y_{*}\right)\right)\right\} \right\} \right| & \leq & \arcsin\left(2\pi MD_{\zeta}r\right)\leq\pi^{2}MD_{\rho}r\,.
\end{eqnarray*}
En considérant $\holo C^{\circ-1}$, et $\chi$ une paramétrisation
d'un arc de $r\ww S^{1}$, on majore la rugosité de $\partial\Sigma\cap\holo C^{-1}\left(r\ww S^{1}\right)$
par $\pi^{2}MD_{\rho}r$. Les autres points du bord se traitent de
manière similaire.

\subsubsection{Preuve de (3) dans le cas non dégénéré}

Rappelons que 
\begin{eqnarray*}
\delta & := & \arccos\sup\left|R\left(\mathcal{U}\left(\rho,r\right)\right)\right|
\end{eqnarray*}
et reprenons les notions introduites en Section~\eqref{sec:Incompress_ND}.
D'après le Lemme~\eqref{lem:stability_beam_ND} on peut garantir
qu'un rayon $z_{\theta}\left(\ww R_{\geq0}\right)$ d'un faisceau
de stabilité $\beam$ coupe $\left\{ \re z=\ln\rho\right\} $, en
choisissant $\theta$ de sorte que $\re{-\lambda\theta}>0$ et $\left|\arg\theta\right|<\delta$.
Cela se produit dès que $\re{\lambda}\leq0$ ou, dans le cas contraire,
$\cos\text{arg}\theta<\sin\left|\arg\lambda\right|$. Ainsi la trace
de chaque feuille sur $\rho\ww S^{1}\times r\ww D$ est un connexe
non vide, ce qui entraîne la $1$\textendash{}connexité de $\rho\ww S^{1}\times r\ww D$
dans $\adh{\mathcal{U}\left(\rho,r\right)}$.

Soit $B:=\tx{Susp}_{\rho\ww S^{1}}\left(\Sigma\right)\subset\rho\ww S^{1}\times r\ww D$.
L'obstruction à la $1$\textendash{}connexité de $B$ dans $B\cup\mathcal{U}\left(\rho,r\right)$
provient de l'existence d'une feuille $\lif{}$ telle que $\adh{\lif{}}\cap B$
ne soit pas connexe. Mais ceci ne peut se produire puisque $\Sigma$
est maximal pour la suspension de points de $\left\{ x_{*}\right\} \times r\ww D$,
donnant la propriété~(BF4).

\subsubsection{Preuve de (3) dans le cas des nœuds\textendash{}cols}

Nous raisonnons exactement comme précédemment, la difficulté ici encore
provient de l'absence de direction fixe incluse dans les faisceaux
de stabilité. 
\begin{defn}
\label{def_bandes_NC}On se donne $\varepsilon\in]0,1[$. On définit
les \textbf{régions nœuds }et\textbf{ cols} respectivement par
\begin{eqnarray*}
N_{\varepsilon} & := & \left\{ \cos\left(kz\right)\geq\varepsilon\,:\,\re z\leq\ln\rho\right\} \\
C_{\varepsilon} & := & \left\{ \cos\left(kz\right)<\varepsilon\,:\,\re z\leq\ln\rho\right\} \,.
\end{eqnarray*}
Chacune de ces région s'écrit comme une partition dénombrable de \textbf{bandes}
$\left(N_{\varepsilon}^{\ell}\right)_{\ell\in\ww Z}$ et $\left(C_{\varepsilon}^{\ell}\right)_{\ell\in\ww Z}$
correspondant aux déterminations de $\arg\left(x^{k}\right)$. 
\end{defn}
Nous prouvons maintenant le
\begin{lem}
\label{lem:1_cnx_NC}Soit $\tilde{\lif{}}$ une feuille telle que
$L:=\tilde{\Pi}\left(\adh{\tilde{\lif{}}}\right)\cap\left\{ \re z=\ln\rho\right\} $
coupe deux bandes cols distinctes, disons en des points $z_{*}$ et
$z^{*}$. Alors $\left[z_{*},z^{*}\right]\subset L$.\end{lem}
\begin{proof}
Soit $\gamma\,:\,\left[0,1\right]\to\tilde{\Pi}\left(\adh{\tilde{\mathcal{L}}}\right)$
la projection d'un chemin tangent avec $\gamma\left(0\right)=z_{*}$
et $\gamma\left(1\right)=z^{*}$. On considère l'union de faisceaux
de stabilité
\begin{eqnarray*}
S & := & \bigcup_{t\in\left[0,1\right]}\beam[\gamma\left(t\right)]
\end{eqnarray*}
 qui est incluse dans $\tilde{\Pi}\left(\tilde{\mathcal{L}}\right)$.
Les trajectoires $z_{1}$ issues de $S\cap C_{\varepsilon}$ intersectent
$\left\{ \re z=\ln\rho\right\} $. En prenant $\varepsilon:=\sin\delta$
on obtient même l'inclusion $\left\{ \re z=\ln\rho\right\} \backslash L\subset N_{\varepsilon}$.
Dès lors il suffit de traiter le cas $z_{*}:=\ln\rho-\ii\frac{\pi}{2k}$
et $z^{*}:=\ln\rho+\ii\frac{\pi}{2k}$.

On reprend l'équation variationnelle de $\varphi:=\ln\left|y\right|$
intégrant le feuilletage au\textendash{}dessus du chemin $t\in\left[0,\frac{\pi}{2k}\right]\mapsto z\left(t\right):=z_{*}+\ii t$.
On obtient
\begin{eqnarray*}
\dot{\varphi}\left(t\right) & = & \re{\ii\exp\left(-kz\left(t\right)\right)\left(1+\tilde{R}\left(z\left(t\right),y\left(t\right)\right)\right)}\\
 & = & -\rho^{-k}\cos\left(kt\right)+\im{\exp\left(-kz\left(t\right)\right)\tilde{R}\left(z\left(t,y\left(t\right)\right)\right)}\,.
\end{eqnarray*}
Notons 
\begin{eqnarray*}
C_{\rho} & := & \sup_{\mathcal{U}\left(\rho,r\right)}\left|x^{-k}R\left(x,y\right)\right|
\end{eqnarray*}
de sorte que
\begin{eqnarray*}
\dot{\varphi}\left(t\right) & \leq & -\rho^{-k}\cos\left(kt\right)+C_{\rho}
\end{eqnarray*}
 puis
\begin{eqnarray*}
\varphi\left(t\right)-\ln r & \leq & -\frac{\rho^{-k}}{k}\sin\left(kt\right)+C_{\rho}t\\
 & \leq & \left(C_{\rho}-\frac{2}{\pi\rho^{k}}\right)t\,.
\end{eqnarray*}
Puisque $R\in x^{k}\germ{x,y}$ il est possible de réaliser l'inégalité
\begin{eqnarray*}
C_{\rho} & < & \frac{2}{\pi\rho^{k}}
\end{eqnarray*}
quitte à prendre $\rho$ assez petit. Ainsi $\varphi\left(t\right)\leq\ln r$
pour tout $t\in\left[0,\frac{\pi}{2k}\right]$. En raisonnant de même
pour le chemin $t\in\left[0,\frac{\pi}{2k}\right]\mapsto z\left(t\right):=z^{*}-\ii t$
on montre finalement $\left[z_{*},z^{*}\right]\subset L$.
\end{proof}
Il se peut que le bord $\rho\ww S^{1}\times r\ww D$ ne soit pas $1$\textendash{}connexe
dans $\mathcal{U}\left(\rho,r\right)$, mais les chemins mettant en
défaut cette propriété doivent avoir leurs extrémités dans des arcs
nœuds, et l'intersection du bord d'une feuille contenant un tel chemin
avec $\rho\ww S^{1}\times r\ww D$ ne doit pas rencontrer plus d'une
bande col dans son revêtement universel. En particulier $\tx{Susp}_{\rho\ww S^{1}}\left(\Sigma\right)$
est $1$\textendash{}connexe dans $\adh{\mathcal{U}\left(\rho,r\right)}$,
puisque il y a $k+1$ bandes cols croisant une région $\left\{ 0\leq\im{z-z_{*}}\leq2\pi\right\} $,
donnant~(BF4).

\subsubsection{Preuve de (4)~: rabotage}

Puisque chaque feuille attachée à un ensemble de type suspension au\textendash{}dessus
de $\rho\ww S^{1}\times\left\{ 0\right\} $ se rétracte tangentiellement
sur le bord $\rho\ww S^{1}\times r\ww D$, le processus de rabotage
de Mar\'in\textendash{}Mattei~\citep[Section 4.3, p877]{MarMat}
permet de montrer l'existence d'un sous\textendash{}ensemble de type
suspension 1\textendash{}connexe dans un autre, donné à l'avance dans
la composante d'entrée. Le contrôle sur la rugosité est alors garanti
par l'estimation~(2).

\subsection{\label{sub:chemins_inamov}Chemins inamovibles}

Nous montrons la Proposition~\ref{prop:bord_NC_pas_1_connexe} lorsque
$\mathcal{U}\left(\rho,r\right)\subset\mathcal{V}\subset\adh{\mathcal{U}\left(\rho,r\right)}$.
En effet si $\varphi\,:\,\overline{\ww D}\to D$ est une représentation
conforme de $\tx{int}\left(D\right)$, fixant $0$, alors le changement
de coordonnées $\phi\,:\,\left(x,y\right)\mapsto\left(x,\varphi\left(y\right)\right)$
transforme $\mathcal{V}\cup V$ en un tel voisinage et $\phi^{*}\omega_{R}$
est encore sous la forme~\eqref{eq:prepa_degenere}.

On a $V\subset\rho\ww D\times r\ww S^{1}$ et les feuilles $\lif{}$
dont l'adhérence coupe $V$ en un point $\left(x_{*},y_{*}\right)$
avec $x_{*}^{k}<0$ (dans une partie col) coupent également le bord
$\rho\ww S^{1}\times r\ww D$ (il suffit de suivre dans $\tilde{\lif{}}$
le relevé du chemin $z_{1}$ issu de $\log x_{*}$). Dès lors $\tilde{\Pi}\left(\adh{\tilde{\lif{}}}\cap\left\{ \left|y\right|=r\right\} \right)$
n'est pas connexe, de sorte que tout chemin $\gamma$ dans $\lif{}$,
tel que $\tilde{\gamma}$ relie deux composantes distinctes, n'est
pas tangentiellement homotope à un chemin dans $V$. Nous renvoyons
à la Figure~\ref{fig:pas_1-connexe}. Pour mettre en défaut la $1$\textendash{}connexité
de $V$ dans $\mathcal{V}\backslash\mathcal{S}$ il faut montrer que
$\gamma$ est homotope dans $\mathcal{V}\backslash\mathcal{S}$ à
un chemin dans $V$. Le point clef ici est que $\mathcal{V}\backslash\mathcal{S}$
contient le disque $\left\{ 0\right\} \times r\ww D$, c'est\textendash{}à\textendash{}dire
que l'enlacement de $\gamma$ autour de $\left\{ x=0\right\} $ n'importe
pas. Il reste donc à contrôler l'indice de $\gamma$ autour de $\mathcal{S}$
dans le cas convergent. Le Corollaire~\ref{cor:transv_pas_1_connexe}
découlera du fait que l'on peut choisir les extrémités $\tilde{p}$
et $\tilde{q}$ de $\tilde{\gamma}$ dans une même transversale $\left\{ y=\tx{cst}\right\} $,
comme l'indiquera la construction.
\begin{defn}
\label{def_inamov}Un chemin tangent $\Gamma$, d'extrémité dans une
même transversale $\left\{ y=\tx{cst}\right\} $, dont le relevé $\tilde{\Gamma}$
relie deux composantes connexes distinctes de $\adh{\tilde{\lif{}}}\cap\left\{ \left|y\right|=r\right\} $,
sera qualifié d'\textbf{inamovible}.
\end{defn}
Le Corollaire~\ref{cor:TCC} est quant à lui une conséquence du fait
que pour tout chemin inamovible $\Gamma$ l'intervalle $\im{\tilde{\Pi}\left(\tilde{\Gamma}\left(\left[0,1\right]\right)\right)}$
a une longueur supérieure à $\frac{\pi}{k}$.

\begin{figure}[H]
\hfill{}\includegraphics[width=6cm]{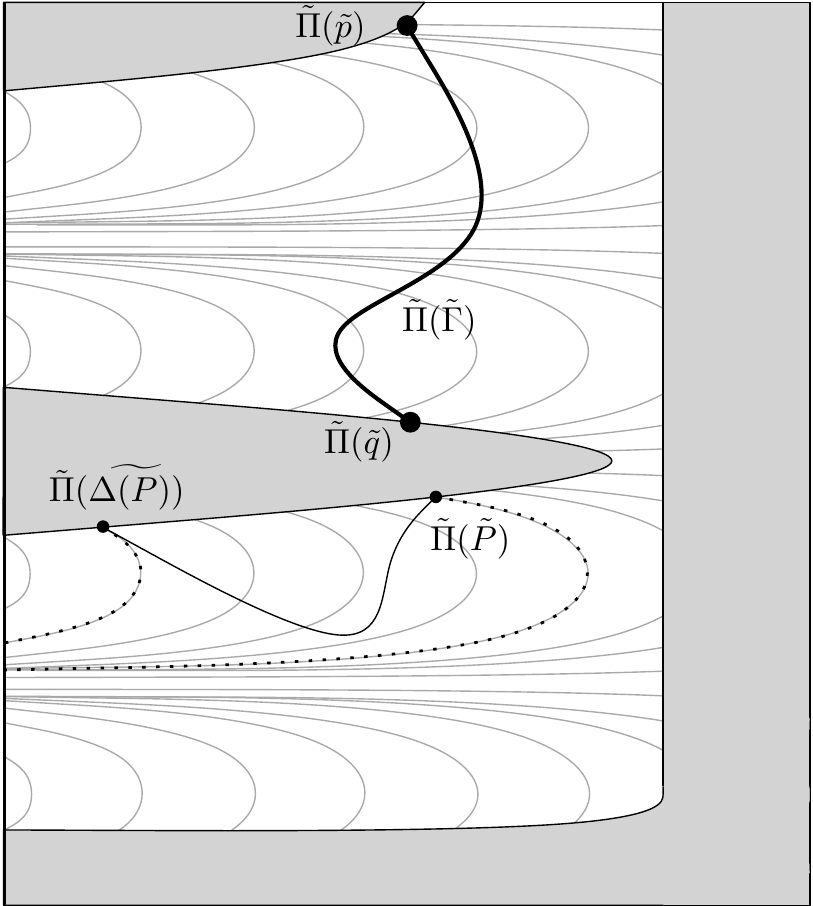}\hfill{}\includegraphics[width=7cm]{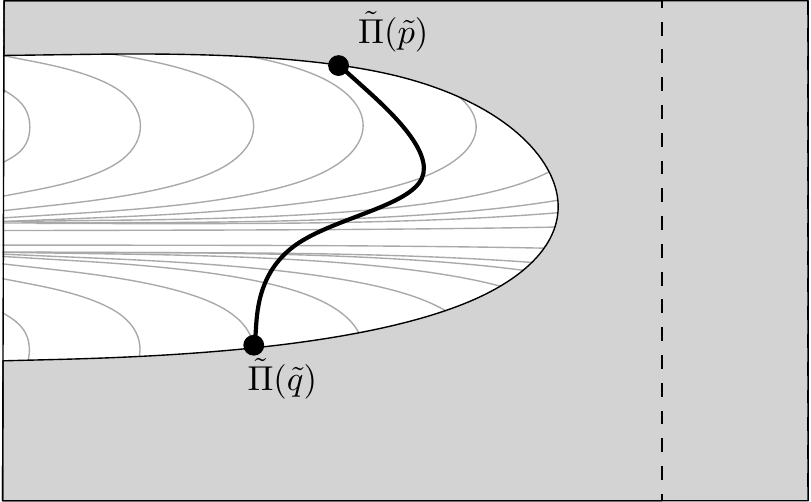}\hfill{}

\caption{\label{fig:pas_1-connexe}La projection en coordonnées logarithmiques
d'un chemin inamovible $\Gamma$ (en gras et en haut, figure de gauche)
à comparer avec la projection d'un chemin tangent portant l'holonomie
forte $\Delta$ (en bas). Les courbes grises correspondent aux courbes
d'iso\textendash{}argument de l'ordonnée de la feuille, celles représentées
en pointillés correspondant à des valeurs différant de $2\pi$. La
figure de droite présente un chemin qui n'est pas inamovible ($x_{*}^{k}>0$
assez petit).}
\end{figure}

\subsubsection{Le cas du modèle formel }

Commençons par prouver le résultat pour le modèle $\left(k,0\right)$
(on renvoie également à la Section~\ref{sub:ex_cvg} pour un traitement
plus détaillé de ce cas là). Les courbes $z_{1}$ définies en~\eqref{eq:iso_arg}
sont les courbes d'iso\textendash{}argument des ordonnées des feuilles
du modèle
\begin{eqnarray*}
y\left(z\right) & = & c\exp\left(-\nf{\exp\left(-kz\right)}k\right)\,\,\,\,\,,\, c\in\ww C\,.
\end{eqnarray*}
Au contraire les courbes $z_{\ii}$ sont les courbes d'iso\textendash{}module
de ces fonctions. Étant donné $y_{*}\in r\ww S^{1}$ on définit $z_{*}:=\ln\nf{\rho}2+\ii\nf{\pi}k$.
Il existe $c\in\ww C_{\neq0}$ tel que $y\left(z_{0}\right)=y_{0}$.
Le long du chemin $z\,:\, t\in\left[-\nf{\pi}k,\nf{\pi}k\right]\mapsto\ln\nf{\rho}2-\ii t$
on a 
\begin{eqnarray*}
\arg y\left(z\left(t\right)\right) & = & \arg c-\frac{2^{k}}{k\rho^{k}}\sin\left(kt\right)\\
\left|y\left(z\left(t\right)\right)\right| & = & \left|c\right|\exp\left(-\frac{2^{k}}{k\rho^{k}}\cos\left(kt\right)\right)\,.
\end{eqnarray*}
Ainsi $\Gamma:=\left(\exp\circ z,y\right)$ est un chemin tangent
reliant $p:=\left(\exp z_{*},y_{*}\right)$ à un autre point $q:=\left(\exp\left(z_{*}+\nf{2\ii\pi}k\right),y_{*}\right)$
dans la même transversale $\Sigma:=\left\{ y=y_{*}\right\} $. Celui\textendash{}ci
n'est tangentiellement homotope à aucun chemin de $\Sigma$. Il est
pourtant homotope dans $\adh{\mathcal{U}}\backslash\mathcal{S}$ à
un chemin reliant $p$ à $q$ dans $\Sigma$, puisque la variation
d'argument de $y$ est nulle.

On passe du modèle $\left(k,0\right)$ au modèle $\left(k,\mu\right)$
par le changement de coordonnées 
\begin{eqnarray*}
z & \longmapsto & z-\frac{1}{k}\log\left(1-\mu kz\exp\left(kz\right)\right)\,.
\end{eqnarray*}
Celui\textendash{}ci est un biholomorphisme (pourvu que $\rho$ soit
assez petit) et donc le résultat persiste pour tous les modèles formels.

\subsubsection{Le cas convergent}

La variation de $\alpha\left(t\right):=\arg y\left(t\right)$ au\textendash{}dessus
d'un chemin $t\in\ww R\mapsto z\left(t\right)$ qui paramètre une
composante connexe de $L:=\adh{\tilde{\lif{}}}\cap\left\{ \left|y\right|=r\right\} $
avec $\left|\dot{z}\right|=1$ est donnée par
\begin{eqnarray*}
\dot{\alpha}\left(t\right) & = & \pm1\,,
\end{eqnarray*}
le signe dépendant du sens de parcours. Dès lors si $\adh{\tilde{\lif{}}}\cap\left\{ \re z=\ln\rho\right\} $
rencontre au moins deux bandes nœuds et si le chemin $z$ paramètre
la composante du bord présente dans une bande col entre les bandes
nœuds alors le chemin $z$ est défini sur $\ww R$ et $\alpha\left(\ww R\right)=\ww R$.
Ainsi on peut relier par un chemin tangent $\Gamma$ n'importe quel
point $\left(z_{*},y_{*}\right)$ d'une autre composante de $L$ à
un unique $\left(z\left(t\right),y\left(t\right)\right)$ avec $\left|y\left(t\right)\right|=r$
et $\arg y\left(t\right)=\arg y_{*}$. Cette dernière égalité montre
que $\Gamma$ est homotope à un chemin de $\left\{ y=y_{*}\right\} $
dans $\mathcal{U}\left(\rho,r\right)\backslash\left\{ y=0\right\} $.

\subsubsection{Le cas divergent}

Traitons finalement le cas des nœuds\textendash{}cols divergents.
On fait apparaître les resommées sectorielles des séparatrices faibles
par $\tau_{j}\,:\,\left(z,y\right)\mapsto\left(z,y-s_{j}\left(\exp z\right)\right)$
sur les bandes 
\begin{eqnarray*}
 & \left\{ \left|\im z-\nf{\ii\pi\left(2j+1\right)}k\right|<\nf{\pi}k+\beta\right\}  & \,,
\end{eqnarray*}
pour se ramener au cas convergent sur une union d'une bande col et
de deux bandes nœuds adjacentes. L'application de $\tau_{j}$ revient
à ajouter à la constante d'intégration $c$ d'un secteur un coefficient
de Stokes $\chi_{j}\in\ww C$ (qui provient de la partie translation
de l'invariant de Martinet\textendash{}Ramis~\citep{MaRa-SN}) ce
qui permet de recoller la feuille du secteur consécutif correspondant
à la constante d'intégration $c+\chi_{j}$. En modifiant l'extrémité
d'arrivée du chemin $\Gamma$, par exemple en la prenant à une affixe
de partie réelle plus négative, on peut encore réaliser la construction
d'un chemin inamovible.

\subsection{\label{sub:coins}Passage des coins}

~

Si on veut appliquer la méthode de Mar\'in\textendash{}Mattei pour
\og passer un coin \fg{} associé à une singularité non\textendash{}linéarisable
$s\in K_{\alpha}$ alors il faut prendre un autre type de blocs $\mathcal{B}_{\alpha}$.
Ceux\textendash{}ci s'obtiennent en poussant la composante d'entrée
du bord par le champ de vecteurs radial $x\ppp x{}$. Lorsque $\lambda<0$
cette opération fournit un voisinage épointé des séparatrices (feuilles
de type «collier»). L'opération de rabotage~\citep[Section 4.3, p877]{MarMat}
s'effectue alors de façon similaire et la composante «de sortie» aura
une rugosité contrôlée par l'application de Dulac $\mathcal{D}_{s}$
associée. Nous ne rentrerons pas dans les détails, nous indiquons
seulement que ce contrôle découle de la majoration du reste du développement
asymptotique de $\mathcal{D}_{s}$ se trouvant par exemple dans~\citep[Theorem 24.38, p462]{IlYako},
à travers la Proposition~4.2.3 de~\citep[p877]{MarMat} (plus précisément
le Lemme~4.4.3). 
\begin{rem}
La preuve de la Proposition~4.2.3 de~\citep[p877]{MarMat} invoque
la Proposition~1.2.4~\citep[p861]{MarMat}. Cependant les hypothèses
de cette proposition ne sont pas vérifiées pour certaines selles quasi\textendash{}résonantes.
En effet J.\textendash{}C.~\noun{Yoccoz} a montré que certaines d'entre
elles contiennent des familles de feuilles fermées (pas simplement
connexes) accumulant la singularité~\citep{Ricard}. Néanmoins un
examen légèrement plus poussé montre que l'on peut se passer d'invoquer
cette proposition sauf dans le cas de singularités linéarisables.
\end{rem}

\section{\label{sec:Exemples}Exemples de feuilletages compressibles}

\subsection{\label{sub:Synthese}Synthèse}

Nous aurons besoin d'une version particulière du théorème de réalisation
de \noun{A.~Lins\textendash{}Neto,} afin d'obtenir des feuilletages
réduits après un éclatement, possédant en chacune de ses singularités
réduites un type analytique local prescrit. Cette version a été écrite
par F.~Loray dans~\citep{Lolo}. Plutôt que de donner l'énoncé général,
nous nous bornons à exposer une version adaptée à notre contexte.
Pour réaliser des feuilletages de classe analytique et d'holonomie
prescrites placés dans sa réduction, il faut satisfaire aux formules
d'indices de Camacho\textendash{}Sad (Section~\ref{sub:Camacho-Sad}). 
\begin{thm}
\label{thm:synthese}\emph{\citep[p159]{Lolo}} On se donne un groupe
de type fini $G:=\left\langle \Delta_{0},\cdots,\Delta_{n}\right\rangle <\diff{\ww C,0}$,
avec $n\in\ww N_{>0}$, tel que $\bigcirc_{\ell=0}^{n}\Delta_{\ell}=\id$.
On se donne également une collection de feuilletages holomorphes singuliers
réduits $\fol{\ell}$, chacun accompagné d'un germe de séparatrice
$S_{\ell}$, de sorte que, dans une bonne coordonnée locale, l'holonomie
de $\fol{\ell}$ le long de $S_{\ell}$ soit précisément $\Delta_{\ell}$.
Supposons finalement que la relation 
\begin{align*}
\sum_{\ell=0}^{n}\csad[\fol{\ell}][S_{\ell}][p_{\ell}] & =-1\,.\tag{{\tt CS}}
\end{align*}
est satisfaite. Alors il existe un germe de feuilletage singulier
$\fol{}$ du plan complexe, non dicritique, qui satisfait les conclusions
suivantes:
\begin{enumerate}
\item $\fol{}$ est réduit après un éclatement et possède $n+1$ points
singuliers $\left(p_{\ell}\right)_{\ell\leq n}$ sur le diviseur exceptionnel
$\mathcal{D}\simeq\ww P_{1}\left(\ww C\right)$,
\item il existe un germe de transversale $\Sigma$ attachée à $\mathcal{D}$
en un point régulier $p$ tel que la représentation d'holonomie projective
$\pi_{1}\left(\mathcal{D}\backslash\left\{ p_{\ell}\,:\,0\leq\ell\leq n\right\} ,p\right)\to\diff{\Sigma,p}$
coïncide avec $G$. Plus précisément, l'image d'un générateur $\gamma_{\ell}$
du groupe fondamental partant de $p$, d'indice $1$ autour de $p_{\ell}$
et d'indice nul autour des autres singularités, est $\Delta_{\ell}$,
\item le type analytique local au voisinage de la singularité $p_{\ell}$
est $\fol{\ell}$,
\item la séparatrice $S_{\ell}$ est incluse dans une composante du diviseur
exceptionnel $\mathcal{D}$.
\end{enumerate}
\end{thm}
\begin{rem}
L'égalité exprimée par $\left({\tt CS}\right)$ est la condition nécessaire
donnée par la formule de Camacho\textendash{}Sad. Il est en général
assez facile de garantir que cette condition tient. Cependant \noun{Y.~Il'Yashenko~\citep{IlYaHedgehog}}
a décrit un sous\textendash{}groupe $\left\langle \Delta_{1},\Delta_{2},\Delta_{3}\right\rangle $
engendré par des germes non linéarisables tangents à une rotation
irrationnelle à petits diviseurs, tel que $\Delta_{1}\circ\Delta_{2}\circ\Delta_{3}=\id$,
mais dont la somme des indices de Camacho\textendash{}Sad de n'importe
quelle réalisation locale (comme l'holonomie d'un feuilletage $\fol{\ell}$)
est toujours inférieure à $-2$.
\end{rem}

\subsection{\label{sub:ex_dvg}Avec au moins un nœud\textendash{}col divergent}

Ces exemples ont été construits avec l'aide précieuse de \noun{D.~Mar\'{i}n}.

\subsubsection{Cas d'un seul nœud\textendash{}col}

Il existe un nœud\textendash{}col $\fol 1$ divergent, avec $k=2$
et $\mu=0$, dont l'invariant de classification orbitale de Martinet\textendash{}Ramis~\citep{MaRa-SN}
est 
\begin{eqnarray*}
\left(\varphi_{0}^{0},\varphi_{0}^{\infty},\varphi_{1}^{0},\varphi_{1}^{\infty}\right) & = & \left(\id,\id+1,\id,\id+1\right)\,.
\end{eqnarray*}
Celui\textendash{}ci est invariant par la permutation des indices
$\,_{j}^{\sharp}\leftrightarrow\,_{1-j}^{\sharp}$ pour $\sharp\in\left\{ 0,\infty\right\} $.
Il existe donc $\Delta_{2}\in Diff\left(\ww C^{2},0\right)$, périodique
d'ordre $2$, qui commute à l'holonomie forte $\Delta_{1}$ de $\fol 1$
\citep[Corollaire 2.8.4 p60]{Lolo}. On note $\Delta_{0}:=\Delta_{1}^{\circ-1}\circ\Delta_{2}^{\circ-1}$.

Puisque dans une bonne coordonnée 
\begin{eqnarray*}
\Delta_{0}\left(h\right) & = & -h+o\left(h\right)\\
\Delta_{1}\left(h\right) & = & h+h^{3}+o\left(h^{3}\right)\\
\Delta_{2}\left(h\right) & = & -h+o\left(h\right)
\end{eqnarray*}
il est possible de trouver des feuilletages locaux donnés par des
selles $\fol 2$ (linéarisable) et $\fol 0$ (résonnante) avec un
indice de Camacho\textendash{}Sad égal à $-\frac{1}{2}$, dont les
holonomies sont précisément $\Delta_{2}$ et $\Delta_{0}$ respectivement
(voir~\citep{MaRa-Res}). En appliquant le Théorème~\ref{thm:synthese}
on obtient un feuilletage $\fol{}$ dont la réduction induit un feuilletage
$\hat{\fol{}}$ défini sur un voisinage du diviseur exceptionnel $\mathcal{D}$
ayant trois singularités $p_{\ell}$, $\ell\in\left\{ 0,1,2\right\} $.
On peut de plus supposer que chaque séparatrice locale de $\hat{\fol{}}$
est soit incluse dans $\mathcal{D}$, soit un petit disque $S_{\ell}\subset\hat{\Pi}^{-1}\left(p_{\ell}\right)$
transverse à $\mathcal{D}$ pour $\ell\in\left\{ 0,2\right\} $, $\hat{\Pi}$
désignant la projection canonique sur $\mathcal{D}$. Notons que les
autres feuilles de $\hat{\fol{}}$ sont transverses aux fibres de
$\hat{\Pi}$, sauf sur le long d'une courbe analytique $T\pitchfork\mathcal{D}$
localisée près de $p_{1}$. 

\bigskip{}

Il est possible de trouver une petite transversale $\Sigma=\hat{\Pi}^{-1}\left(p_{*}\right)$,
munie de la coordonnée $h\in\ww D$, sur lequel la représentation
d'holonomie associée à $\hat{\Pi}$ et $\mathcal{D}^{*}:=\mathcal{D}\backslash\left\{ p_{0},p_{1},p_{2}\right\} $
coïncide avec $\left\langle \Delta_{0},\Delta_{1},\Delta_{2}\right\rangle $.
Le groupe fondamental $\pi_{1}\left(\mathcal{D}^{*}\right)$ est un
groupe libre de rang $2$, de générateurs $a_{1}$ et $a_{2}$ faisant
respectivement un tour autour de $p_{1}$ et $p_{2}$ mais aucun autour
de l'autre point. Clairement
\begin{eqnarray*}
\holo{a_{2}^{2}} & = & \id\\
\holo{\left[a_{1},a_{2}\right]} & = & \id
\end{eqnarray*}
et les mots $u:=a_{2}^{2}$ et $v:=\left[a_{1},a_{2}\right]$ sont
sans puissance commune. Il s'ensuit que les relevés respectifs $\tilde{u}$
et $\tilde{v}$ de $u$ et $v$ par $\hat{\Pi}$ dans une feuille
$\mathcal{L}$ (assez proche du diviseur) sont des cycles de la feuille.
En choisissant un voisinage $\Omega$ de $p_{*}$ sur lequel $\holo u$
et $\holo v$ sont définies on s'aperçoit que chaque élément de $\holo{\left\langle u,v\right\rangle }$
est holomorphe et injectif sur $\Omega$. Ainsi pour chaque feuille
$\mathcal{L}$ de $\hat{\fol{}}$ assez proche du diviseur on a 
\begin{eqnarray*}
\left\langle \tilde{u},\tilde{v}\right\rangle  & < & \pi_{1}\left(\mathcal{L}\right)\,.
\end{eqnarray*}
Ces éléments sont sans puissances communes sinon $u$ et $v$ le seraient.

Une feuille $\mathcal{L}$ étant une surface de Riemann à bord son
groupe fondamental est libre. Puisque $\left\langle u,v\right\rangle $
n'est pas monogène cela signifie qu'aucun morphisme injectif ne peut
exister de $\pi_{1}\left(\mathcal{L}\right)$ dans $\pi_{1}\left(\mathcal{V}^{*}\right)=\ww Z\oplus\ww Z$.
Pour conclure, ce phénomène se produisant aussi près du diviseur que
l'on souhaite, le feuilletage n'est incompressible dans aucun voisinage
de la singularité.

\subsubsection{Avec un nœud\textendash{}col divergent et un convergent}

On peut construire un exemple un peu plus explicite que le précédent
mais en utilisant le même argument de grandeur des groupes libres
de rang $2$. On considère une équation de Riccati, définie sur $\ww D\times\ww CP^{1}$
et plus connue sous le nom d'équation d'Euler, donnée par la $1$\textendash{}forme
\begin{eqnarray*}
\omega_{E}\left(x,y\right) & := & x^{2}\dd y-\left(y+x\right)\dd x\,.
\end{eqnarray*}
Celle\textendash{}ci admet deux singularités de type nœud\textendash{}col.
La première, en $\left(0,0\right)$, est de type divergent alors que
celle située en $\left(0,\infty\right)$ est de type convergent. Les
holonomies fortes associées sont inverses l'une de l'autre. On peut
prendre comme troisième singularité la selle linéaire
\begin{eqnarray*}
\omega_{L}\left(x,y\right) & := & x\dd y+y\dd x
\end{eqnarray*}
dont l'holonomie est l'identité. On synthétise alors un éclatement
de singularité de feuilletage ayant ces trois singularités locales,
dont les invariants de Camacho\textendash{}Sad le long du diviseur
exceptionnel sont $0$ pour les nœuds\textendash{}cols situés en $p_{0}$
et $p_{1}$ et $-1$ pour la selle située en $p_{2}$.

Cette fois les mots 
\begin{eqnarray*}
u & := & a_{2}\\
v & := & \left[a_{1},a_{2}\right]
\end{eqnarray*}
sont dans le noyau de $\holo{\bullet}$ et sans puissance commune.
Cela signifie encore qu'il existe des feuilles $\mathcal{L}$ telles
que $\pi_{1}\left(\mathcal{L}\right)$ contienne un groupe libre de
rang $2$, et comme $\pi_{1}\left(\mathcal{V}^{*}\right)=\ww Z\oplus\ww Z$
l'incompressibilité est impossible. En fait il est possible de montrer
directement l'existence d'un lacet non trivial de $\mathcal{L}$ qui
l'est dans un voisinage épointé $\mathcal{V}^{*}$ des séparatrices.
\begin{prop}
Notons $\gamma$ le relevé de $\left[a_{1},a_{2}\right]$ par $\hat{\Pi}$
dans une feuille $\mathcal{L}$ de $\hat{\fol{}}$ suffisamment proche
de $\mathcal{D}$. On écrit $\iota^{*}\,:\,\pi_{1}\left(\mathcal{L}\right)\to\pi_{1}\left(\mathcal{V}^{*}\right)$
le morphisme naturel induit par l'inclusion $\mathcal{L}\subset\mathcal{V}^{*}$.
Alors $\gamma\in\ker\iota^{*}\backslash\left\{ 1\right\} $.\end{prop}
\begin{proof}
Soit $\Sigma$ une transversale sur lequel est réalisée l'holonomie
$\holo{\left[a_{1},a_{2}\right]}$ et prenons une feuille $\mathcal{L}$
coupant $\Sigma$ ; notons également $\tilde{a}_{2}$ le relevé de
$a_{2}$ dans $\mathcal{L}$. Puisque la seule séparatrice de $\hat{\fol{}}$
passant par $p_{1}$ est le diviseur, $\gamma$ est homotope dans
$\mathcal{V}^{*}$ à $\tilde{a}_{2}\tilde{a}_{2}^{-1}=1$ (autrement
dit $\gamma\in\ker\iota^{*}$). Pour autant $\gamma\neq1$ dans $\pi_{1}\left(\mathcal{L}\right)$,
comme nous allons le montrer. Partant d'un point $A\in\Sigma\cap\mathcal{L}$
le relevé de $\left[a_{1},a_{2}\right]=a_{1}a_{2}a_{1}^{-1}a_{2}^{-1}$
on obtient la configuration décrite en Figure~\ref{fig:contrex}.
Comme l'holonomie forte $\Delta_{1}$ du nœud\textendash{}col n'a
pas de point fixe on a toujours $A\neq\Delta_{1}^{\circ-1}\left(A\right)$.
Mais si $\gamma$ était trivial dans $\pi_{1}\left(\mathcal{L}\right)$
alors ces points devraient être homotopes dans $\mathcal{L}\cap\Sigma$
(voir la Figure~\ref{fig:contrex}), ce qui n'est pas possible puisque
$\Sigma\cap\mathcal{L}$ est discret.

\begin{figure}[H]
\hfill{}\includegraphics[width=10cm]{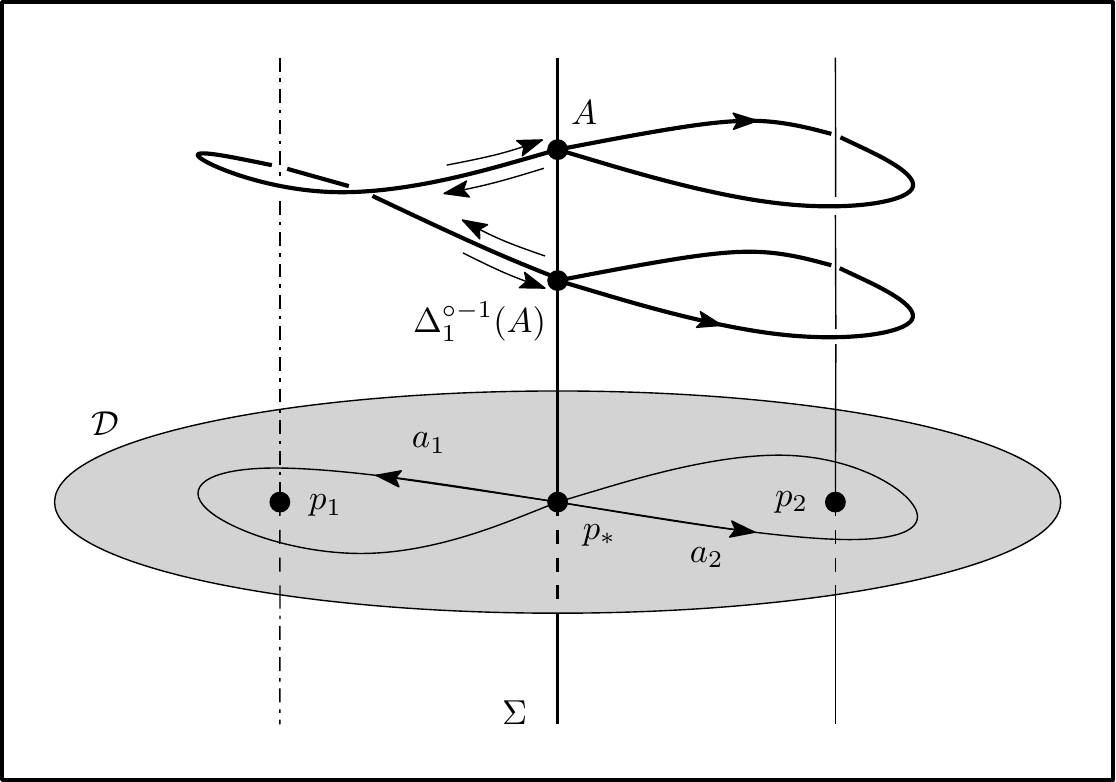}\hfill{}

\caption{\label{fig:contrex}Relevé dans le feuilletage du commutateur $\left[a_{1},a_{2}\right]$
à partir d'un point $A$ d'une feuille $\mathcal{L}$. }
\end{figure}

\end{proof}

\subsubsection{Encore plus de contre\textendash{}exemples ?}

Plus généralement, dès qu'il existe un nœud\textendash{}col divergent
dans la réduction de $\fol{}$ alors peuvent apparaître des cycles
forçant la compressibilité du feuilletage, lorsque existent des relations
entre les holonomies provenant d'autres parties du feuilletage. D'une
manière plus précise, si on place le nœud\textendash{}col en un point
$p$ d'un diviseur $\mathcal{D}$ et si $a$ un générateur du groupe
fondamental de $\mathcal{D}^{*}$ tournant autour de $p$, alors le
noyau de la flèche naturelle 
\begin{eqnarray*}
\iota^{*}\,:\,\pi_{1}\left(\mathcal{D}^{*}\right) & \to & \pi_{1}\left(\mathcal{D}^{*}\cup\left\{ p\right\} \right)
\end{eqnarray*}
doit contenir $a$.

\subsection{\label{sub:ex_cvg}Sans nœud\textendash{}col divergent}

Ces exemples s'appuient le modèle $\left(1,0\right)$
\begin{eqnarray*}
\omega_{0}\left(x,y\right) & := & x^{2}\dd y-y\dd x\,,
\end{eqnarray*}
mais la même construction se généralise quasiment à l'identique pour
chaque modèle $\left(k,0\right)$ avec $k\geq1$. Le feuilletage convergent
induit possède la famille $\left(\Gamma_{c}\right)_{c\in]0,1[}$ de
chemins inamovibles
\begin{eqnarray*}
\Gamma_{c}\,:\,\left[-\pi,\pi\right] & \longrightarrow & \adh{\ww D\times\ww D}\\
t & \longmapsto & \left(c\exp\left(\ii t\right)\,,\,\exp\left(-\frac{1}{c}\left(1+\exp\left(-\ii t\right)\right)\right)\right)
\end{eqnarray*}
(voir la Section~\ref{sub:chemins_inamov} et la Définition~\ref{def_inamov}).
Leur extrémité $\left(-c\,,\,1\right)$ appartient à la même transversale
$\left\{ y=1\right\} $. L'explication géométrique de l'existence
de ces chemins inamovibles est la présence dans l'équation
\begin{eqnarray*}
y\frac{\dd x}{\dd y} & = & x^{2}
\end{eqnarray*}
d'un pôle mobile de la solution
\begin{eqnarray*}
x\left(y\right) & = & \frac{-c}{c\log y+1}\,.
\end{eqnarray*}
 Leur projection sur $\left\{ x=0\right\} $ a la forme d'un haricot
(disons si $c>\frac{1}{\pi}$) d'indice nul autour de $\left\{ y=0\right\} $
mais entourant la position du pôle mobile $y_{c}:=\exp\nf{-1}c$ (voir
la Figure~\ref{fig:haricot_1}). Pour cette raison ces cycles ont
été suggérés par \noun{E.~Paul} comme candidats pour trouver un nœud\textendash{}col
compressible. Ils ne fournissent cependant pas de contre\textendash{}exemple
car chaque $\Gamma_{c}$ est non trivial\emph{ }dans $\adh{\ww D\times\ww D}\backslash\left\{ x=0\right\} $.
Par contre en adjoignant une autre singularité on peut former des
cycles produisant les exemples attendus.

\begin{figure}[H]
\hfill{}\includegraphics[width=10cm]{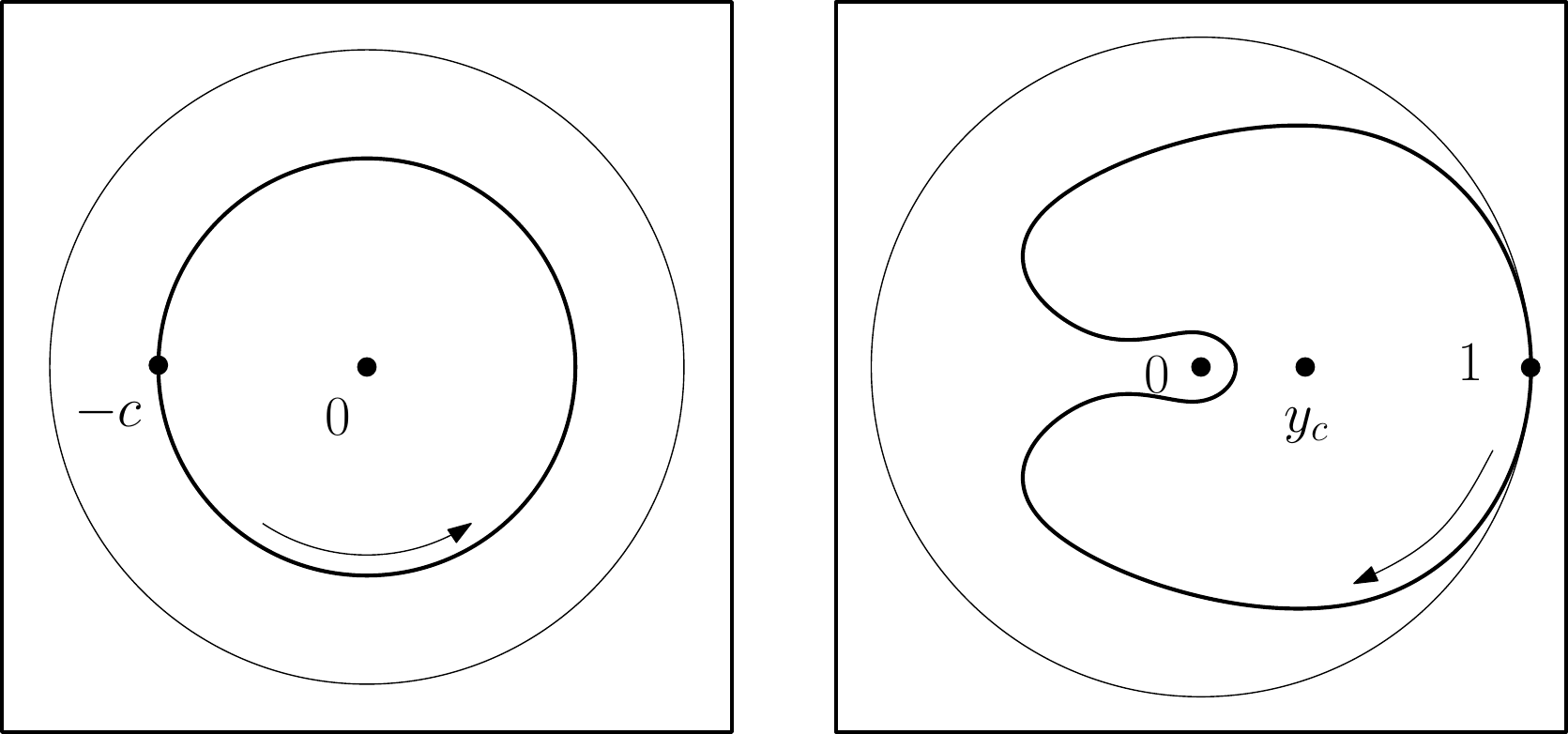}\hfill{}

\caption{\label{fig:haricot_1}Projeté de $\Gamma_{c}$ sur $\left\{ y=0\right\} $
(à gauche) et sur $\left\{ x=0\right\} $ (à droite). }
\end{figure}

\bigskip{}
Effectuons le tiré\textendash{}en\textendash{}arrière de $\omega_{0}$
par une application $\psi$ de degré $2$ afin de symétriser la situation
en plaçant deux tels feuilletages en vis\textendash{}à\textendash{}vis.
Explicitement,
\begin{eqnarray*}
\psi\,:\,\left(x,y\right) & \longmapsto & \left(x,1-y^{2}\right)
\end{eqnarray*}
transforme $\omega_{0}$ en 
\begin{eqnarray*}
\psi^{*}\omega_{0} & = & \left(y^{2}-1\right)\dd x-2yx^{2}\dd y\,.
\end{eqnarray*}
Ce feuilletage possède les trois séparatrices $\left\{ x=0\right\} $
et $\left\{ y=\pm1\right\} $. Le chemin $\Gamma_{c}$ admet deux
pré\textendash{}images par $\psi$ 
\begin{eqnarray*}
\Gamma_{c}^{\pm}\,:\,\left[-\pi,\pi\right] & \longrightarrow & \ww D\times\ww C\\
t & \longmapsto & \left(c\exp\left(\ii t\right)\,,\,\pm\sqrt{1-\exp\left(-\frac{1}{c}\left(1+\exp\left(-\ii t\right)\right)\right)}\right)\,.
\end{eqnarray*}
Comme l'image de $\Gamma_{c}$ est contenue dans $\adh{\ww D\times\ww D}$
on peut considérer une détermination fixée de la racine carrée sur
la coupure $\ww C\backslash\ww R_{<0}$. Maintenant on fabrique le
chemin
\begin{eqnarray*}
\gamma_{c}\,:\,\left[-2\pi,2\pi\right] & \longrightarrow & \ww D\times\ww C\\
t\leq0 & \longmapsto & \Gamma_{c}^{-}\left(\pi+t\right)\\
t\geq0 & \longmapsto & \Gamma_{c}^{+}\left(\pi-t\right)\,.
\end{eqnarray*}
 Par construction ce cycle est d'indice nul autour de chaque branche
de $\left\{ x\left(y^{2}-1\right)=0\right\} $, mais étant la concaténation
de deux chemins inamovibles «distants» il n'est pas tangentiellement
trivial. Puisque $c$ peut\textendash{}être pris arbitrairement proche
de $0$, le feuilletage $\tilde{\fol{}}$ défini par $\psi^{*}\omega_{0}$
n'est incompressible dans aucun voisinage de $S:=\left\{ x=0\right\} $.

\begin{figure}[H]
\hfill{}\includegraphics[width=10cm]{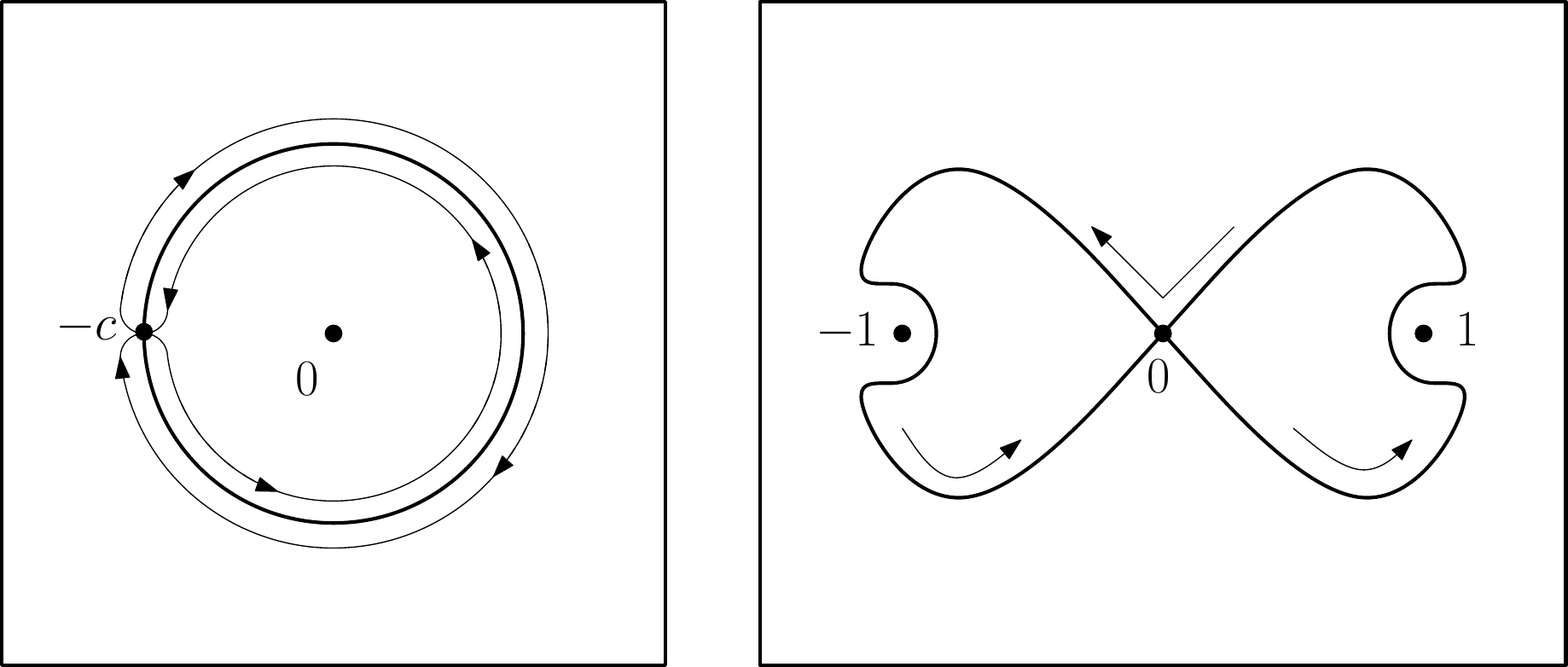}\hfill{}

\caption{\label{fig:haricot_2}Projeté de $\gamma_{c}$ sur $\left\{ y=0\right\} $
(à gauche) et sur $\left\{ x=0\right\} $ (à droite). }
\end{figure}

Cette association de deux singularités $p_{0}=\left(0,-1\right)$
et $p_{1}=\left(0,1\right)$ forme un bloc, muni du feuilletage $\tilde{\fol{}}$,
que nous allons plonger dans la réduction d'un germe de feuilletage
singulier en lui adjoignant une troisième singularité $p_{2}$, le
diviseur exceptionnel\emph{ }coïncidant avec $S$. Puisque 
\begin{eqnarray*}
\csad[\tilde{\fol{}}][S][p_{0}] & = & \csad[\tilde{\fol{}}][S][p_{1}]\\
 & = & 0\,,
\end{eqnarray*}
et comme l'holonomie de $\tilde{\fol{}}$ le long du chemin $t\in\left[0,2\pi\right]\mapsto\left(0,2\exp\left(\ii t\right)\right)$
est conjuguée à $\Delta_{0}^{\circ2}$, en notant $\Delta_{0}$ l'holonomie
forte de $\omega_{0}$, on choisit pour $p_{2}$ une selle résonnante
tangente à $x\dd y+y\dd x$ et d'holonomie $\Delta_{0}^{\circ-2}$
tangente à l'identité (voir~\citep{MaRa-Res}). L'invocation du Théorème~\ref{thm:synthese}
achève alors la construction.

\subsection{\label{sub:branches_mortes}Un feuilletage possédant beaucoup de
composantes initiales}

\subsubsection{Une branche morte \emph{ex nihilo}}

Ici nous construisons une famille d'exemples correspondant à la situation
décrite dans le Lemme~\ref{lem:sing_pas_morte} et correspondant
à la figure ci\textendash{}dessous.

\begin{figure}[H]
\includegraphics[height=5cm]{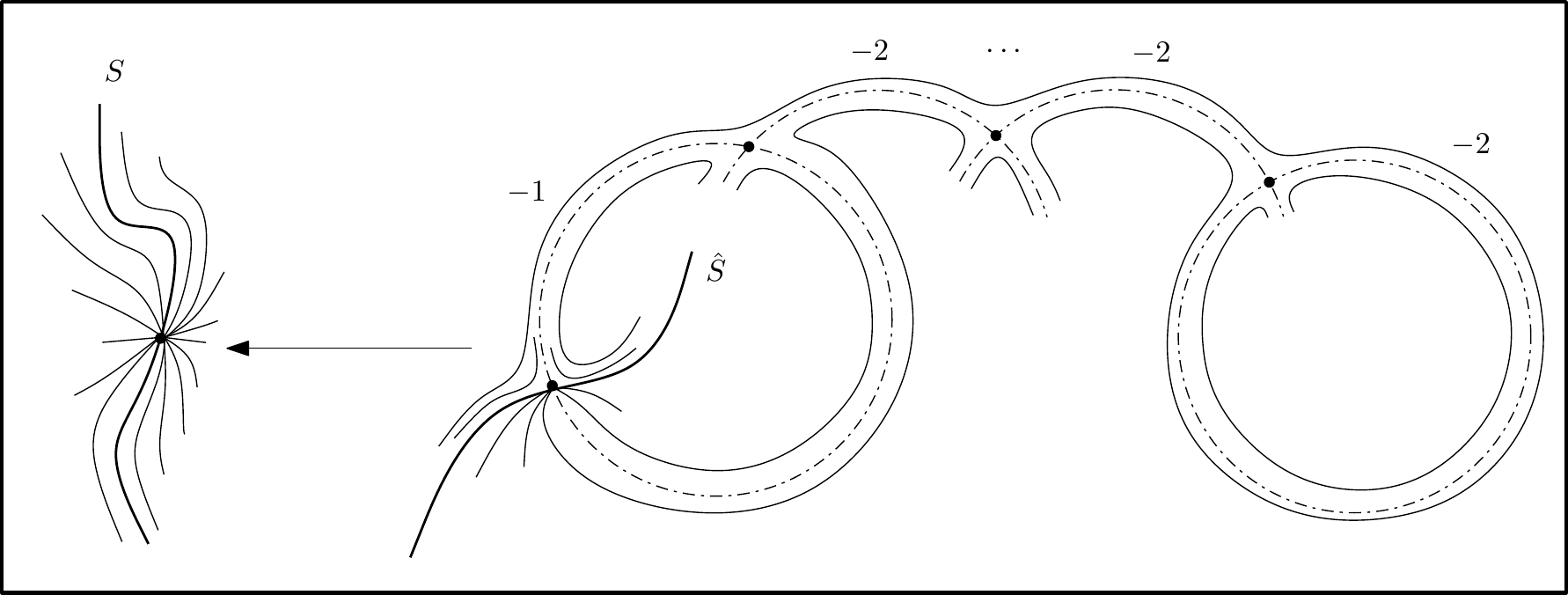}
\caption{Une branche morte obtenue en réduisant une seule séparatrice (en gras).
Les nombres indiquent la classe de Chern des composantes.}
\end{figure}

Considérons pour commencer le feuilletage donné par la forme différentielle,
correspondant à un nœud résonant linéaire,
\begin{eqnarray*}
\omega_{1}\left(x,y\right) & := & \left(x-y\right)\dd x+x\dd y\,.
\end{eqnarray*}
Le feuilletage sous\textendash{}jacent est réduit en un éclatement,
et sa réduction contient une seule singularité (un nœud\textendash{}col
modèle avec $\left(k,\mu\right)=\left(1,-1\right)$). 

Plus généralement on vérifie sans peine que, pour chaque $n\in\ww N_{>0}$,
la $1$\textendash{}forme différentielle
\begin{eqnarray*}
\omega_{n}\left(x,y\right) & := & \left(\frac{x}{n}-y^{n}\right)\dd x+nxy^{n-1}\dd y
\end{eqnarray*}
 est réduite en $n$ éclatement et sa réduction est une branche morte.
Le dernier éclatement crée un diviseur ayant d'une part un coin, d'autre
part un nœud\textendash{}col modèle 
\begin{eqnarray*}
\hat{\omega}_{n}\left(t,y\right) & := & t^{2}\dd y+y\left(1-\frac{1}{n}t\right)\dd t
\end{eqnarray*}
 par lequel passe le transformé strict de l'unique séparatrice $\left\{ x=0\right\} $
du feuilletage de départ.

\subsubsection{Une composante initiale avec beaucoup de branches mortes}

De tels feuilletages $\fol{}$ se construisent grâce à la version
plus générale du Théorème~\ref{thm:synthese} donnée dans~\citep{Lolo}.
Il suffit ici de considérer un voisinage d'un diviseur $D$ de classe
de Chern $-1$, contenant $m\in\ww N_{>0}$ singularités $\left(p_{j}\right)_{1\leq j\leq m}$,
chacune localement conjuguée à celle donnée par $\omega_{n_{j}}$
(définie dans le paragraphe précédent), et une singularité selle résonante
$p_{0}$ que nous caractérisons ci\textendash{}après. Le diviseur
$D$ sera le premier diviseur créé lors de la réduction de $\fol{}$.

Après réduction de chaque $p_{j}$, le diviseur $D$ est attaché à
$m$ branches mortes et sa classe de Chern devient
\begin{eqnarray*}
\frak{c} & := & -1-\sum_{j=1}^{m}n_{j}\in\ww Z_{<-1}\,.
\end{eqnarray*}
La singularité d'attache de la branche morte en $p_{j}$ est un nœud\textendash{}col
dont la séparatrice forte coïncide avec $D$~: son indice de Camacho\textemdash{}Sad
par rapport à $D$ est nul~; on note $\Delta_{j}$ son holonomie
forte (tangente à l'identité). On place en $p_{0}$ une selle résonante
de partie linéaire $\frak{c}x\dd y-y\dd x$ (en choisissant des coordonnées
locales dans lesquelles $D=\left\{ x=0\right\} $) dont l'holonomie
le long de $D$ 
\begin{eqnarray*}
\Delta_{0}\left(h\right) & = & \exp\left(2\ii\pi\frak{c}\right)h+\cdots=h+\cdots
\end{eqnarray*}
est l'inverse de $\bigcirc_{j=1}^{m}\Delta_{j}$. Le théorème de réalisation
de Martinet\textemdash{}Ramis~\citep{MaRa-Res} prouve qu'une telle
singularité de feuilletage existe.

\bibliographystyle{preprints}
\bibliography{biblio}

\vfill{}

\lyxaddress{Laboratoire I.R.M.A., 7 rue \noun{R.~Descartes}, Université de Strasbourg,
67084 Strasbourg cedex, France\\
\emph{Email:} \texttt{teyssier@math.unistra.fr}}

\vfill{}

\end{document}